\documentclass[a4paper,oneside,onecolumn]{article}
\usepackage{cancel}
 
\usepackage{soul}
\usepackage[left=2.0cm,top=2.5cm,right=2.0cm,bottom=2.5cm,bindingoffset=0.0cm]{geometry}

\usepackage{float} 
\usepackage{multirow}
\usepackage[T1]{fontenc} 

\usepackage[utf8]{inputenc} 
\newcommand{\vertiii}[1]{{\left\vert\kern-0.25ex\left\vert\kern-0.25ex\left\vert #1 
    \right\vert\kern-0.25ex\right\vert\kern-0.25ex\right\vert}}
\usepackage{dsfont}
\usepackage{graphics,graphicx,epstopdf} 

\usepackage{enumitem} 

\usepackage{subfig} 

\usepackage{amsmath,amssymb,amsthm,amsfonts} 

\usepackage{lineno, hyperref} 

\usepackage{color, xcolor}

\usepackage{authblk} 

\theoremstyle{plain} 
\newtheorem{theorem}{Theorem}[section]
\newtheorem{lemma}{Lemma}[section]

\newtheorem{corollary}{Corollary}[section]
\theoremstyle{definition} 

\newtheorem{example}{Example}[section]
\theoremstyle{remark} 
\newtheorem{remark}{Remark}[section]

\definecolor{c1}{rgb}{0,0,1} 
\definecolor{c2}{rgb}{0,0.3,0.9} 
\definecolor{c3}{rgb}{0.3,0,0.9} 

\hypersetup{
	pdfauthor={Aditi Tomar},
	pdfsubject={},
	pdftitle={},
	pdfkeywords={},
	breaklinks = true, 
	linktocpage=true, 
	colorlinks=true,       
	linkcolor={c1},          
	citecolor={c2},        
	filecolor=black,      
	urlcolor={c3},       
}
\numberwithin{equation}{section}

\allowdisplaybreaks
\makeatletter
\def\namedlabel#1#2{\begingroup
    #2%
    \def\@currentlabel{#2}%
    \phantomsection\label{#1}\endgroup
}
\makeatother
\begin{document}
\title{\textbf{On a Non-Uniform $\alpha$-Robust IMEX-L1 Mixed FEM for Time-Fractional PIDEs}}

\author{\textsc{Lok Pati Tripathi }\thanks{\href{mailto: lokpati@iitgoa.ac.in}{lokpati@iitgoa.ac.in}}}
\author{\textsc{Aditi Tomar}\thanks{\href{mailto: aditi183321002@iitgoa.ac.in}{aditi183321002@iitgoa.ac.in}} }
\affil{School of Mathematics and Computer Science, Indian Institute of Technology Goa, Goa-403401, India.}
\author{\textsc{Amiya K. Pani }\thanks{\href{mailto: amiyap@goa.bits-pilani.ac.in}{amiyap@goa.bits-pilani.ac.in}}}
\affil{Department of Mathematics, BITS-Pilani, KK Birla Goa Campus, Goa-403726, India.}

\date{} 

\maketitle 

\begin{abstract}
A non-uniform implicit-explicit L1 mixed finite element method (IMEX-L1-MFEM) is investigated for a class of time-fractional partial integro-differential equations (PIDEs) with space-time dependent coefficients and non-self-adjoint elliptic part. The proposed fully discrete method combines an IMEX-L1 method on a graded mesh in the temporal variable with a mixed finite element method in spatial variables. The focus of the study is to analyze stability results and to establish optimal error estimates, up to a logarithmic factor, for both the solution and the flux in $L^2$-norm when the initial data $u_0\in H_0^1(\Omega)\cap H^2(\Omega)$. Additionally, an error estimate in $L^\infty$-norm is derived for 2D problems. All the derived estimates and bounds in this article remain valid as $\alpha\to 1^{-}$, where $\alpha$ is the order of the Caputo fractional derivative. Finally, the results of  several  numerical experiments conducted at the end of this paper are confirming our theoretical findings.
\end{abstract}

\vspace{1em}
 \section{Introduction}\label{section1}
This paper develops and analyzes an IMEX-L1 mixed finite element method (MFEM) for the following class of time-fractional partial integro-differential equations (PIDEs):

\begin{align}\label{pide}
\begin{cases}
\partial^{\alpha}_{t} u +  \mathcal{L}u - \lambda\mathcal{I}u \;=\; f & \quad \text{in} \quad \Omega\times J,\\[5pt]
u ~= ~0 &\quad \text{on}\quad \partial\Omega\times J, \\[5pt]
u(\cdot,\;0) ~= ~u_0 &\quad \text{in}\quad\Omega,
\end{cases}
\end{align}
where $\partial^{\alpha}_{t} u$ is the Caputo fractional derivative of order $\alpha,\; 0<\alpha < 1,$ with respect to $t$,
\begin{align}
    \label{caputoderivative} \partial^{\alpha}_{t} u:&=\;\int_{0}^{t}k_{1-\alpha}(t-s)\partial_{s}u\;ds\quad \text{with kernel }\; k_{\beta}(t):=\;\frac{t^{\beta-1}}{\Gamma{(\beta)}},\;t>0,\; \beta>0,\\
    \label{EllipticOperator} \mathcal{L}u:&=\;  -\nabla\cdot(\boldsymbol{A}\nabla u) + \boldsymbol{b}\cdot\nabla u + cu, \quad \text{and } \quad
    \mathcal{I}u(\cdot,t):=\; \int_{\Omega}u(\boldsymbol{y},t)g(\cdot,\boldsymbol{y})\;d\boldsymbol{y}.
\end{align}
 Here, $\partial\Omega$ is the boundary of a convex polygonal or polyhedral bounded domain $\Omega\subset\mathbb{R}^{d}(d=1,2,3),$ $\lambda\in \mathbb{R}$ and $J:=(0,\; T]$ for $0<T<\infty$. The symbol $\Gamma(\cdot)$ represents the gamma function. For the case of $\lambda = 0$, the problem (\ref{pide}) reduces to a time-fractional linear general parabolic PDE with variable coefficients. To ensure the well-posedness and regularity results for the solution $u$ of the problem (\ref{pide}), appropriate assumptions on $\boldsymbol{A}:\Omega\times J\to \mathbb{R}^{d\times d}$, $\boldsymbol{b}:\Omega\times J\to \mathbb{R}^{d}$, $c:\Omega\times J\to \mathbb{R}$, $g:\Omega\times\Omega\to\mathbb{R}$, $f:\Omega\times J\to \mathbb{R}$, and the initial condition $u_0$ are mentioned in the Theorem~\ref{Regularity_condition} of Section~\ref{section2}. For a.e. $(\boldsymbol{x},\;t)\in\Omega\times J$, $\boldsymbol{A}(\boldsymbol{x},\;t)$ is assumed to be a real symmetric and uniformly positive definite matrix in the sense that there exists a positive constant $\kappa_{0}$ and $\kappa_{1}$ such that
\begin{align*}
    &\kappa_{0}\;\|\boldsymbol{y}\|^{2} \;\leq\; \boldsymbol{y}^{T}\boldsymbol{A}\boldsymbol{y}\; \leq\; \kappa_{1}\;\|\boldsymbol{y}\|^{2} \quad \forall~\boldsymbol{y}\in \mathbb{R}^{d}\quad \text{in }~ \Omega\times J.
\end{align*}
 
In option pricing problems under Merton's and Kou's jump-diffusion models \cite{ MR3633783, MR3959545, kou2002jump, Merton1976}, there is a presence of a non-local integral operator $\mathcal{I}: L^{1}(\Omega)\to L^{\infty}(\Omega)$ which is characterized by a specific kernel $g(\boldsymbol{x},\boldsymbol{y})$. Compared to fully implicit methods, the main computational advantage of IMEX methods is that, at each time level, sparse matrix inversion is sufficient instead of dealing with dense matrices \cite{MR3633783, MR2873249}. The well-posedness of problem (\ref{pide}), with $\lambda \neq 0$, can be established by utilizing the properties of the non-local integral operator $\mathcal{I}$ and employing techniques from \cite{MR4290515, MR3814402} (as described in Section~\ref{section2}).

In the literature, several numerical approaches have been developed to compute an approximate solution to the problem (\ref{pide}). Notably, the error analysis of the L1 scheme, which accounts for the initial singularity, has been investigated using different temporal mesh strategies, see, a recent survey \cite{MR4499608}. In \cite{MR3894161, MR3957890}, a uniform temporal mesh with $\Delta t = T/N$ was employed to analyze the error. On the other hand, in \cite{ MR3957889, MR4090355, MR3639581}, graded time grids of the form $t_n = (n/N)^{\gamma} T$ with corresponding time step size $\Delta t_n = t_n - t_{n-1}$ have been utilized for error analysis, where $\gamma \geq 1$ is the grading parameter, and $N$ is the number of grid points in the time direction. These studies contribute to the understanding of error behavior in the L1 scheme by considering different temporal discretization approaches, highlighting the impact of grid structures on the accuracy of numerical solutions to the problem (\ref{pide}). Furthermore, for spatial discretizations, there are numerous numerical schemes, including finite difference \cite{ MR3790081, MR2349193, MR3639581}, finite element  \cite{MR3742890, MR3335216, MR3033018, MR4125980, MR3904430, MR3802434},
finite volume element method \cite{MR3649431, MR3834443}, and for up-to-date references, please refer to \cite{MR4572861}. One recent study by Stynes et al. \cite{MR3639581} analyzed the L1 formula, combined with finite difference methods (FDM) in the spatial direction for approximating the problem (\ref{pide}), where $\mathcal{L}u= -u_{xx}+cu$. Subsequently, Mustapha \cite{MR4090355} has discussed the L1 method for a reaction subdiffusion problem and derived second-order convergence in time. In \cite{TOMAR2024137}, the authors have discussed optimal convergence analysis of an IMEX-L1 method with a general elliptic operator of the form (\ref{EllipticOperator}) having time and space-dependent coefficients.  

Mixed methods offer a simultaneous approximation of both the solution and the flux, representing a widely utilized approach for obtaining approximate solutions to systems that encompass multiple unknowns (see, \cite{MR1115205}, \cite{MR483555}). Johnson and Thom\'{e}e \cite{MR0610597} have investigated semi-discrete mixed finite element method (MFEM) for parabolic and non-stationary Stokes equations and provided optimal convergence in velocity and pressure in $L^{\infty}((L^2)^2)$ and $L^{\infty}(L^2)$ norms, see [\cite{MR0610597}, Theorem 2.1]. In \cite{MR2551194}, Sinha et al. have obtained a priori $L^2$-error estimates for pressure and velocity. Goswami and Pani, in \cite{MR2823474}, have obtained a priori error estimates for
semi-discrete Galerkin approximations to a general second-order linear parabolic initial and
boundary value problem. 

 In the context of MFEM for time-fractional diffusion problems, limited
publications are available in the literature \cite{MR3592147, MR3816184, MR4107214, karaa2023mixed,MR4108629}. In \cite{MR3592147}, the authors have discussed a fully-discrete approximate scheme using a uniform mesh with conforming MFEMs and nonconforming MFEMs for the problem (\ref{pide}) with $\mathcal{L} = -\Delta$. They have also established error analysis under higher regularity assumptions on the solution $u$, that is, at least when $u \in L^{\infty}(H^4
)\cap W^{1,\infty}(H^2)$. Karaa \cite{MR3816184} has proposed a unified energy-based analysis of error estimates for the semi-discrete mixed finite element method for the sub-diffusion problem with $\mathcal{L}=- \Delta$ using the properties of the inverse of the elliptic operator. Subsequently, he has established the convergence analysis of a fully discrete method by using a semigroup-type approach and exploiting the properties of the inverse of the associated elliptic operator in \cite{MR4107214}. Recently, in \cite{karaa2023mixed}, the authors have exploited MFEM for the time-fractional Fokker-Planck equation and obtained point-wise in time-optimal $L^2$-norm estimate for the solution. However, these works are limited to time-independent coefficients. In \cite{MR4108629}, Karaa and Pani have proposed a mixed finite element method for a class of time-fractional diffusion problems, where $\mathcal{L}$ is a second-order self-adjoint elliptic operator with time-dependent
coefficients and have exploited energy arguments to obtain optimal error estimates in semi-discrete cases only. In contrast to the aforementioned literature, in our case, $\mathcal{L}$ represents a general non-self-adjoint linear uniformly elliptic operator with time-dependent coefficients. Most of the papers mentioned in this paragraph are not $\alpha$ robust in the sense that the constants that appeared in the error estimates blow up as $\alpha \to 1^-$. Considering the initial singularity, this combination presents challenges in analyzing the standard nonuniform approximations of equation (\ref{pide}). This motivates extending the framework to provide optimal $L^2$-norm error estimates for both the solution $u$ and the flux $\boldsymbol{\sigma}$ in a fully-discrete method applied to the problem (\ref{pide}), complementing the work in \cite{MR4108629} and providing $\alpha$-robustness.

To the best of our knowledge, when the initial data $u_0\in H_0^1(\Omega)\cap H^2(\Omega)$, there is hardly any literature available focusing on the optimal convergence analysis of the non-uniform (IMEX) L1 mixed FEM for time-fractional PDEs or PIDEs that include a general elliptic operator, as given by equation (\ref{EllipticOperator}), with coefficients varying in both time and space. This article aims to bridge the existing gap by establishing the stability and optimal convergence analysis of the non-uniform IMEX-L1-MFEM for the problem (\ref{pide}). The main contributions of the work can be summarized as follows:
\begin{itemize}
    \item A precise discrete fractional Gr\"{o}nwall inequality (see Theorem~\ref{DFGI}) is derived, which allows us to establish the stability and convergence of the proposed IMEX-L1-MFEM on a graded mesh for general time-fractional PIDEs (\ref{pide}), under a reasonable time-step condition. 
    \item Stability results for the proposed method applied to the problem (\ref{pide}) are established (see, Theorem~\ref{stabilitytheorem1} and Theorem~\ref{stabilitytheorem2}), with a particular attention to estimating the flux (see Lemma~\ref{dfdoi} and Theorem~\ref{stabilitytheorem2}) due to the dependence of the non-self-adjoint elliptic operator $\mathcal{L}$ in (\ref{pide}) on space and time variables and their impact on the discrete fractional derivative.
    \item Up to a logarithmic factor, optimal error estimates of order $O\left((h^{2} + N^{-\min(\gamma\alpha,\;2-\alpha)})\log_e(N) \right)$, $1\leq \gamma \leq \frac{2(2-\alpha)}{\alpha}$, for both the solution $u$ and the flux $\boldsymbol{\sigma}$ are derived in $L^2$-norm under the assumption $u_0 \in \dot{H}^2(\Omega) = H_0^1(\Omega)\cap H^2(\Omega)$,  where $\gamma $ is the grading parameter, $N+1$ denotes the number of grid points in the temporal direction and $h$ is the maximum diameter of finite elements (see, Theorem~\ref{L2H1errortheorem}). 
    \item Under the assumption of $u_0 \in \dot{H}^{2+\epsilon}(\Omega)$, $0<\epsilon < \frac{1}{2}$, an optimal order error estimate of order $O\big(h^{2} + N^{-\min(\gamma\alpha,\;2-\alpha)} \big)$, $\frac{(2-\alpha)}{\alpha} < \gamma \leq \frac{2(2-\alpha)}{\alpha}$, is obtained for both the solution $u$ and the flux $\boldsymbol{\sigma}$ in $L^2$-norm (see, Theorem~\ref{L2H1errortheorem}).
     \item A maximum norm error estimate is established for 2D problems(see, Corollary~\ref{Linftynormestimate}). 
     \item In light of recent developments emphasizing the robustness of error estimates as $\alpha\to 1^{-}$ (see \cite{MR4246866}, \cite{MR4287911}, and references therein), we have ensured that all the estimates and bounds established in this article remain valid under this limit.
\end{itemize}
The article primarily addresses linear problems, but there is potential to extend the proposed method and analysis to semi-linear time-fractional PDEs/PIDEs with appropriate modifications and regularity assumptions on the non-linear function. Here, $H^{m}(\Omega)$ and $ H_{0}^{m}(\Omega),\;m\in\mathbb{N}\cup\{0\},$ are standard Sobolev spaces equipped with standard norm $\|\cdot\|_{m}$, $H_{0}^{0}(\Omega)\;=\; H^{0}(\Omega)\;=\; L^{2}(\Omega)$, and $\|\cdot\|_{0} = \|\cdot\|$ with $(\cdot,\cdot)$ denotes the $L^{2}$-inner product, $\langle\cdot,\cdot\rangle$ is the duality paring between $H_{0}^{1}(\Omega)$ and its topological dual $(H_{0}^{1}(\Omega))^{\ast}:=H^{-1}(\Omega),$ $\langle \phi,\psi\rangle = (\phi,\psi) \;\forall \phi\in L^{2}(\Omega)$ and $\forall \psi\in H_{0}^{1}(\Omega).$ $H(div;\Omega) = \{ \boldsymbol{w}\in (L^2(\Omega, \mathbb{R}^d): \nabla \cdot \boldsymbol{w} \in L^2(\Omega) \}$ is a Hilbert space equipped with the norm $\| \boldsymbol{w}\|_{H(div; \Omega)}= (\|\boldsymbol{w}\|^2 + \| \nabla \cdot \boldsymbol{w}\|^2)^{\frac{1}{2}}$. For a given Hilbert space $\mathcal{H}$, $W^{m,p}(J;\mathcal{H}), m\in\mathbb{N}\cup\{0\}, 1\leq p \leq \infty$, $W^{0,2}(J;\mathcal{H}):=L^2(J;\mathcal{H})$, denotes the standard Bochner-Sobolev spaces. The Hilbert space $\dot{H}^s(\Omega),$ $s\geq 0,$ equipped with the induced norm $\|v\|_{\dot{H}^s(\Omega)} = \|(-\Delta)^{s/2} v\| = \sqrt{\sum_{j=1}^{\infty} \lambda_j^s (v,\; \phi_j)^2}$ is defined by 
$\dot{H}^s
(\Omega) = \left\{v \in L^2(\Omega) : \sum_{j=1}^{\infty} \lambda_j^s (v,\; \phi_j)^2\;<\;\infty\right\}$ (see, \cite{MR2249024}, \cite{MR4290515}, and \cite{MR4125980}), where $\{(\lambda_j,\;\phi_j)\}_{j=1}^{\infty}$ are the eigenpairs of the eigenvalue problem $-\Delta w = \lambda w\;\text{in}\;\Omega$, $w=0\;\text{on}\;\partial \Omega$, with multiplicity counted, and $\{\phi_j\}_{j=1}^{\infty}$ is an orthonormal basis for $L^2(\Omega)$. In particular, $\dot{H}^0
(\Omega)=L^2(\Omega)$, $\dot{H}^1
(\Omega)=H_0^1(\Omega)$, and $\dot{H}^2
(\Omega)=H_0^1(\Omega)\cap H^2(\Omega)$. Throughout this article, the symbol $C$ denotes a positive generic constant (not necessarily the same at each occurrence) which may depend on the grading parameter $\gamma$ and on the data, such as, $\alpha$, $\lambda$, $\Omega$, $\boldsymbol{A}$, $\boldsymbol{b}$, $c$, and $g$. However, it remains bounded as $\alpha\to 1^{-}$ and is independent of the approximation parameters, such as the maximum diameter $h$ of finite elements, the number of grid points $N+1$ in the temporal direction etc. 

The rest of the article is organized as follows. The section~\ref{section2} presents the variational formulation of the problem (\ref{pide}) and provides some results that will be utilized in the subsequent analysis. Non-uniform IMEX-L1-MFEM is proposed in section~\ref{section3}. Subsection~\ref {section4} presents a precise discrete fractional Gr\"{o}nwall inequality and applies it to analyze the stability of the proposed method. In subsection~\ref{section5}, optimal error estimates are derived. Numerical tests are conducted to validate and demonstrate the efficiency of the proposed method based on the theoretical analysis in section~\ref{section7}. Finally, this paper is concluded in section~\ref{section8}.

%
\section{Mixed formulation and some useful results}\label{section2}
Introducing the flux variable $\boldsymbol{\sigma} = \boldsymbol{A}\nabla u$, the problem (\ref{pide}) can be re-written as
\begin{align*}
& \boldsymbol{\sigma} - \boldsymbol{A}\nabla u =0\quad\text{in}\quad\Omega\times J,\\
&\partial^{\alpha}_t u - \nabla \cdot \boldsymbol{\sigma} + \boldsymbol{b}\cdot \boldsymbol{A}^{-1} \boldsymbol{\sigma} + cu -\lambda \mathcal{I}u=\;f \quad\text{in}\quad\Omega\times J,\\
& u = 0 \quad\text{on}\quad\partial\Omega\times J,\\
& u(\cdot,0) = u_0(\cdot) \quad\text{in}\quad\Omega.
\end{align*}
With $\boldsymbol{B}:= \boldsymbol{A}^{-1}$, $V:=L^2(\Omega)$, and  $\boldsymbol{W}:=H(div; \Omega)$, a mixed formulation of the problem (\ref{pide}) is to find the solution pair $(u,\boldsymbol{\sigma}):J\to V\times \boldsymbol{W}$, such that $u(0)=u_0$ and 
\begin{align}
\label{variation1} (\boldsymbol{B}\boldsymbol{\sigma}, \boldsymbol{w}) + ( u,  \nabla \cdot \boldsymbol{w})~ &=~0 \quad \forall \boldsymbol{w} \in \boldsymbol{W},\; t\in J,\\
    \label{variation2}\langle\partial_t^{\alpha} u,v \rangle  -(\nabla \cdot \boldsymbol{\sigma},v)+( \boldsymbol{b} \cdot \boldsymbol{B} \boldsymbol{\sigma},v) + (cu,v) ~&=~\lambda (\mathcal{I}u,v ) + \langle f,v\rangle  \quad \forall v \in V,\; t\in J,
\end{align}
where the boundary condition $u=0\;\text{on}\;\partial\Omega\times J$ is implicitly contained in (\ref{variation1}). As $\boldsymbol{A}$ is symmetric and uniformly positive definite, $\boldsymbol{B}$ is also symmetric and satisfy the following estimate 
\begin{align}\label{supdc1}
    &\beta_0\|\boldsymbol{y}\|^2 \leq \boldsymbol{y}^T\boldsymbol{B}\boldsymbol{y}\leq \gamma_0\|\boldsymbol{y}\|^2\quad \forall \boldsymbol{y}\in \mathbb{R}^d\quad \text{in }\; \Omega\times J,
\end{align}
for some positive constants $\beta_0$ and $\gamma_0$. Further, since $\boldsymbol{A}: J\to L^\infty(\Omega; \mathbb{R}^{d\times d})$ is Lipschitz continuous and positive definite, the relation $\boldsymbol{B}(t) - \boldsymbol{B}(s) = \boldsymbol{A}^{-1}(t)(\boldsymbol{A}(s) - \boldsymbol{A}(t))\boldsymbol{A}^{-1}(s),\;s,t\in J,$ implies that $\boldsymbol{B}: J\to L^\infty(\Omega; \mathbb{R}^{d\times d})$ is also Lipschitz continuous, i.e., there exists a positive constant $L_B$ such that
\begin{align}
    \label{Blips}& \|\boldsymbol{B}(t) - \boldsymbol{B}(s)\|_{L^\infty(\Omega; \mathbb{R}^{d\times d})} \leq L_B |t-s|\quad \forall t, s\in J.
\end{align}
As it is well known that the regularity results for the solution of time-fractional evolution equations like (\ref{pide}) play a significant role in the error analysis of any method (see, \cite{MR2832607, MR4023101, MR4054212, MR3957890, MR4125980, TOMAR2024137}), for our subsequent use, we state the following well-posedness and regularity results.
\begin{theorem}\label{Regularity_condition} Let $f \in W^{3,1}(J;L^2(\Omega))\cap W^{2,1}(J;H^1(\Omega))$, $u_0 \in \dot{H}^2(\Omega)=H_0^1(\Omega)\cap H^2(\Omega)$, 
\begin{align}\label{intbound}
  \|\mathcal{I}\phi\|_j\;\leq\;C_{\mathcal{I}}\|\phi\|_j\quad \forall \phi\in \dot{H}^j(\Omega)=H_0^j(\Omega),\;j=0,1, 
\end{align}
and for $k=0,1,2,3,$
\begin{align}\label{coefbounds}
        & \left\|\partial^k_t \boldsymbol{A}(\cdot,t)\right\|_{W^{1,\infty}(\Omega,\mathbb{R}^{d\times d})} \leq C_{\boldsymbol{A}},\;\left\|\partial^k_t\boldsymbol{b}(\cdot,t)\right\|_{L^{\infty}(\Omega,\mathbb{R}^{d})} \leq C_{\boldsymbol{b}},\;\left\|\partial^k_tc(\cdot,t)\right\|_{L^{\infty}(\Omega)}\leq C_c,\quad  \forall t\in J,
\end{align}
for some positive constants $C_{\mathcal{I}}$, $C_{\boldsymbol{A}}$, $C_{\boldsymbol{b}}$ and $C_c$. Then, the problem (\ref{pide}) is well-posed, and its solution $u$ satisfies the following regularity results
\begin{align*} 
&\begin{cases}
 &u\in C^{2}(J; H^2(\Omega)),\; \|u(t)\|_2 +t^{\alpha}\|\partial^{\alpha}_t u(t) \|_2 + t\|\partial_{t}u(t)\|_2 \leq C,\; t^l\|\partial_{t}^{l}u(t)\|_m\leq C t^{(2-m)\alpha/2},\\
&t^k\|\partial^k_t\left(\nabla\cdot\boldsymbol{A}(t)\nabla u(t)\right)\|_m\leq C t^{-m\alpha/2}\quad \forall t\in J,\;l=1,2,\;k=0,1\;\&\; m=0,1,
\end{cases}
\end{align*}
for some positive constant $C$ which remains bounded as $\alpha\to 1^{-}$.
\end{theorem}
\begin{proof}
    The well-posedness of the problem (\ref{pide}) is addressed in Section~2 of \cite{TOMAR2024137}. Moreover, except for the two estimates $\;t^{\alpha}\|\partial^{\alpha}_t u \|_2 \leq C\;$ and $\;t^k\|\partial^k_t\left(\nabla\cdot\boldsymbol{A}(t)\nabla u(t)\right)\|_m\leq C t^{-m\alpha/2}$, remaining estimates follows from \cite{TOMAR2024137} (see, Section~2, Theorem~2.1). To prove $\;t^{\alpha}\|\partial^{\alpha}_t u \|_2 \leq C$, we appeal to the identity $t\partial^{\alpha}_t u = \partial^{\alpha}_t (tu) - \alpha \int_0^t k_{1-\alpha}(t-s)u\;ds - tk_{1-\alpha}(t)u (0) $ as follows
    \begin{align}
        \nonumber  \|t\partial^{\alpha}_t u\|_2 &\leq \|\partial^{\alpha}_t (tu)\|_2 + \alpha \int_0^t k_{1-\alpha}(t-s)\|u\|_2\;ds + tk_{1-\alpha}(t)\|u (0) \|_2\\
        &\nonumber \leq \int_0^tk_{1-\alpha}(t-s)\|\partial_s(su)\|_2\;ds + C\frac{\alpha t^{1-\alpha}}{\Gamma{(2-\alpha)}} + C\frac{t^{1-\alpha}}{\Gamma{(1-\alpha)}}\\
        \nonumber &\leq Ct^{1-\alpha}.
    \end{align}
Now, we derive the final estimates $\;t^k\|\partial^k_t\left(\nabla\cdot\boldsymbol{A}(t)\nabla u(t)\right)\|_m\leq C t^{-m\alpha/2}$. Differentiate (\ref{pide}) with respect to $t$ to obtain
\begin{align}
    \label{finalregestofgradsigma_middle}& \partial^k_t (\nabla\cdot \boldsymbol{A}\nabla u) \;=\; \partial^k_t ( \partial^{\alpha}_t u  ) + \partial^k_t (  \boldsymbol{b}\cdot \nabla u )  + \partial^k_t ( c u  ) -\lambda \partial^k_t  \mathcal{I} u  - \partial^k_t f,\;k=0,1,
\end{align}
and then, after multiplying $t^k,\;k=0,1$, on both the sides of (\ref{finalregestofgradsigma_middle}), apply triangle inequality to derive
\begin{align}
    \label{finalregestofgradsigma}& t^k\|\partial^k_t (\nabla\cdot \boldsymbol{A}\nabla u)\|_m \;\leq\; t^k\|\partial^k_t ( \partial^{\alpha}_t u  )\|_m + t^k\|\partial^k_t (  \boldsymbol{b}\cdot \nabla u )\|_m  + t^k\|\partial^k_t ( c u  )\|_m + t^k\|\partial^k_t  \mathcal{I} u  \|_m + t^k\|\partial^k_t f\|_m,\;m=0,1.
\end{align}
The hypothesis $f \in  W^{2,1}(J;H^1(\Omega))$ with (\ref{intbound}) and (\ref{coefbounds}) yields 
\begin{align}
   \nonumber &t^k\|\partial^k_t (  \boldsymbol{b}\cdot \nabla u )\|_m  + t^k\|\partial^k_t ( c u  )\|_m + t^k\|\partial^k_t  \mathcal{I} u  \|_m  + t^k\|\partial^k_t f\|_m\\
\label{finalregestofgradsigma_aux2_middle} &\leq\; C_1( t^k\|\partial^k_t u \|_{m+1}  + t^k\|\partial^k_t  u  \|_m +  t^k\|\partial^k_t  u  \|_m ) + C_2,
\end{align}
for some positive constants $C_1$ and $C_2$ which remain bounded as $\alpha\to 1^{-}$. Thus, an appeal to the estimates $t^k\|\partial_{t}^{k}u(t)\|_m\leq C t^{(2-m)\alpha/2},\;k=0,1,\;m=0,1,$ and $t\|\partial_{t}u(t)\|_2\leq C$ in (\ref{finalregestofgradsigma_aux2_middle}) shows
\begin{align}
   \label{finalregestofgradsigma_aux2} &t^k\|\partial^k_t (  \boldsymbol{b}\cdot \nabla u )\|_m  + t^k\|\partial^k_t ( c u  )\|_m + t^k\|\partial^k_t  \mathcal{I} u  \|_m  + t^k\|\partial^k_t f\|_m\;\leq\; C.
\end{align}
Let us now estimate the first term on the right-hand side of (\ref{finalregestofgradsigma}). As $t^k\|\partial_{t}^{k}u(t)\|_m\leq C t^{(2-m)\alpha/2}$, we obtain 
\begin{align}
   \nonumber & \left\| \int_0^t k_{1-\alpha}(t-s)\partial_s u(s)\;ds\right\|_m\;\leq\; \int_0^t k_{1-\alpha}(t-s)\|\partial_s u(s)\|_m\;ds\\
   \label{finalregestofgradsigma_aux3} & \;\leq\; \frac{C}{\Gamma(1-\alpha)} \int_0^t (t-s)^{-\alpha}\;s^{\frac{(2-m)\alpha}{2}-1}\;ds\;=\;   C\Gamma{\left(\frac{(2-m)\alpha}{2}\right)}\;t^{-\frac{m}{2}\alpha} ,\;\text{ and}\\
   \nonumber & \left\| \partial_t\left(t\int_0^t k_{1-\alpha}(t-s)\partial_s u(s)\;ds\right)\right\|_m\;\leq\; \left\| \partial_t\int_0^t (t-s)k_{1-\alpha}(t-s)\partial_s u(s)\;ds\right\|_m + \left\| \partial_t\int_0^t k_{1-\alpha}(t-s)s \partial_s u(s)\;ds\right\|_m\\
   \nonumber&\hspace{1cm}\leq\; (2-\alpha)\int_0^t k_{1-\alpha}(t-s)\|\partial_s u(s)\|_m\;ds +\int_0^t k_{1-\alpha}(t-s)\|\partial_s u(s)\|_m\;ds+ \int_0^t k_{1-\alpha}(t-s)\|s\partial_s^2 u(s)\|_m\;ds\\
\label{finalregestofgradsigma_aux4}&\hspace{1cm}\leq\; \frac{C}{\Gamma(1-\alpha)}\int_0^t (t-s)^{-\alpha}\;s^{\frac{(2-m)\alpha}{2}-1}\;ds \;=\;  C\Gamma{\left(\frac{(2-m)\alpha}{2}\right)}\;t^{-\frac{m}{2}\alpha}.
\end{align}
By combining the estimates (\ref{finalregestofgradsigma_aux3}) and (\ref{finalregestofgradsigma_aux4}), we obtain the following estimate
\begin{align}
   \label{finalregestofgradsigma_aux5} & t^k\|\partial^k_t ( \partial^{\alpha}_t u  )\|_m \;\leq\; Ct^{-\frac{m\alpha}{2}}, \; k=0,1, \; m=0,1.
\end{align}
Finally, an application of the estimates (\ref{finalregestofgradsigma_aux2}), and (\ref{finalregestofgradsigma_aux5}) in (\ref{finalregestofgradsigma}) yields the desired result.
\end{proof}
\begin{theorem}\label{Regularity_condition1} Let $u_0 \in \dot{H}^{2+\epsilon}(\Omega)$, $0<\epsilon < \frac{1}{2}$. Then, under the assumptions of Theorem~\ref{Regularity_condition}, the following improved regularity results hold
\begin{align*} 
&\begin{cases}
 &t^{\alpha}\|\partial^{\alpha}_t u(t) \|_2 + t\|\partial_{t}u(t)\|_2  \leq C\;t^{\epsilon\alpha/2},\; t\|\partial_{t} u(t)\|_1\leq C t^{(\epsilon+1) \alpha/2},\\
&\|\nabla\cdot\boldsymbol{A}(t)\nabla u(t)\|_1\leq C t^{(\epsilon-1) \alpha/2}\quad \forall t\in J, 
\end{cases}
\end{align*}
for some positive constant $C$ which remains bounded as $\alpha\to 1^{-}$.
\end{theorem}
\begin{proof}
    The Theorem~A.1, Theorem~A.2 and Theorem~A.3 in \cite{TOMAR2024137}, and the interpolation properties of $\dot{H}^s(\Omega),\;s=2,3,$ (see, Lemma~2.2. in \cite{MR0349038}) yields the following estimates
    \begin{align}\label{addreg1}
        &t\|\partial_{t}u(t)\|_2  \leq C\;t^{\epsilon\alpha/2},\; t\|\partial_{t} u(t)\|_1\leq C t^{(\epsilon+1) \alpha/2}\;t\in J,\; 0<\epsilon <1/2.
    \end{align}
Now, by applying above estimates (\ref{addreg1}), we obtain
\begin{align}
    \nonumber \|\partial_t^{\alpha}u(t)\|_m& \leq \frac{1}{\Gamma(1-\alpha)}\int_0^t (t-s)^{-\alpha}\|\partial_s u(s)\|_m\;ds, \; m=1,2,\\
    \nonumber &\leq C\frac{1}{\Gamma(1-\alpha)}\int_0^t (t-s)^{-\alpha}\;s^{(\epsilon +2-m)\alpha/2-1}\;ds, \; m=1,2,\\
    \label{addreg2} &\leq C\frac{\Gamma((\epsilon +2-m)\alpha/2)}{\Gamma(1-\alpha + (\epsilon +2-m)\alpha/2)}t^{(\epsilon +2-m)\alpha/2-\alpha}, \; m=1,2.
\end{align}
Finally, an appeal to the above estimate (\ref{addreg2}) with $m=1$ and Theorem~\ref{Regularity_condition} yields the estimate
\begin{align*}
    &\|\nabla\cdot\boldsymbol{A}(t)\nabla u(t)\|_1\leq C t^{(\epsilon-1) \alpha/2}\quad \forall t\in J.
\end{align*}
This completes the rest of the proof.
\end{proof}
\begin{remark}
  Under a suitable regularity assumption on the boundary $\partial \Omega$ of $\Omega$, we have $\dot{H}^{2+\epsilon}(\Omega) = H_0^1(\Omega)\cap H^{2+\epsilon}(\Omega)$, $0<\epsilon <\frac{1}{2}$ (See, equation (3.2) in \cite{MR2832607}). 
\end{remark}
%
\section{Non-uniform IMEX-L1-MFEM}\label{section3}
%
\textit{Spatial discretization:} Let $\mathcal{T}_h$ be a regular family of decomposition of $\Omega$ (see, \cite{Ciarlet1978}) into closed $d$-simplexes $\mathbb{T}$ of size $h = \max\{\text{diam}(\mathbb{T}); \mathbb{T} \in \mathcal{T}_h\}$. Further, let $V_h\times\boldsymbol{W}_h$ be a finite element subspace of $V\times\boldsymbol{W}$ having the following three properties:
\begin{enumerate}
    \item[(i)] $\nabla\cdot \boldsymbol{W}_h \subseteq V_h$.
    \item[(ii)] There exists a projection $\Pi_h : \boldsymbol{W} \rightarrow \boldsymbol{W}_h$, called Fortin projection satisfying $\nabla \cdot \Pi_h = P_h(\nabla \cdot)$, where $P_h :V \rightarrow V_h$ is the $L^2$-projection defined by $(P_h v -v,v_h)=0\; \forall v_h \in V_h, \; v \in V$, and hence 
    \begin{align*}
    (\nabla \cdot (\Pi_h \boldsymbol{w} - \boldsymbol{w}  ), v_h) =0 \; \forall v_h \in V_h
\;
\text{ and }\;
(P_h v-v, \nabla \cdot \boldsymbol{w}_h) =0 \; \forall \boldsymbol{w}_h \in \boldsymbol{W}_h.
\end{align*}
    \item[(iii)] Approximation properties:
    \begin{align}\label{approxprop}
        &   \|\boldsymbol{w}- \Pi_h \boldsymbol{w} \| \leq C h^r \| \nabla \cdot \boldsymbol{w}\|_{r-1}; \quad \|v -P_h v\| + h \|v -P_h v\|_1 \leq Ch^r\|v\|_r,\;1\leq r \leq 2.
    \end{align}
\end{enumerate}
Examples of such finite-dimensional sub-spaces $V_h\times\boldsymbol{W}_h$ of $V\times\boldsymbol{W}$ having the above properties can be found in \cite{MR483555} and \cite{MR1115205}.

\noindent\textit{Temporal Discretization:} Consider a partition $\{t_n\}_{n=0}^N$ of the interval $\bar{J}=[0, T]$ such that 
\begin{align}\label{gradedmesh1}
    & t_n = \left(\frac{n}{N}\right)^{\gamma}T,\quad \gamma\geq 1,\;0\leq n \leq N.
\end{align}
The well-known L1-formula \cite{MR3639581, MR3790081} for approximating the Caputo fractional derivative $\partial_{t}^{\alpha} \phi(t_n)$ is given by
\begin{align}
   \label{L1scheme}  &D_{t_n}^{\alpha} \phi(t_{n}) := {\sum_{j=1}^{n}}K^{n,j}_{1-\alpha} \left(\phi(t_j)-\phi(t_{j-1})\right) = {\sum_{j=1}^{n}} \int_{t_{j-1}}^{t_j}k_{1-\alpha}(t_n - s)ds\frac{\left(\phi(t_j)-\phi(t_{j-1})\right)}{\Delta t_j}\\
   \nonumber & \approx \; {\sum_{j=1}^{n}} \int_{t_{j-1}}^{t_j}k_{1-\alpha}(t_n - s)\partial_{s}\phi(s)ds \; = \; \int_{0}^{t_n}k_{1-\alpha}(t_n - s)\partial_{s}\phi(s)ds  \; = \; :\partial_{t}^{\alpha} \phi(t_n),
\end{align}
where the discrete kernels are defined as
\begin{align}
   \label{discretekernel}& K^{n,j}_{\alpha} = \frac{1}{\Delta t_j}\int_{t_{j-1}}^{t_j}k_{\alpha}(t_n - s)ds = \frac{k_{1+\alpha}(t_n - t_{j-1})-k_{1+\alpha}(t_n - t_{j})}{\Delta t_j}\;\text{ with }\;\Delta t_j = t_j -t_{j-1}.
   \end{align}
These discrete kernels satisfy the property:
   \begin{align}\label{Kproperty}
    0 \leq K_{1-\alpha}^{n,j-1}< K_{1-\alpha}^{n,j}, \quad 2 \leq j \leq n \leq N .  
   \end{align}
Now, the non-uniform IMEX-L1-MFEM for the problem (\ref{pide}) is to find a pair $(u_h^n, \boldsymbol{\sigma}^n_h) \in V_h \times \boldsymbol{W}_h,\;0\leq n \leq N$, such that $u_h^0=P_h u_{0}$ and 
\begin{align}
\label{fullydiscrete1} (\boldsymbol{B}^n\boldsymbol{\sigma}^n_h, \boldsymbol{w}_h) + ( u_h^n, \nabla \cdot  \boldsymbol{w}_h) &=0 \quad \forall \boldsymbol{w}_h \in \boldsymbol{W}_h,\;0\leq n\leq N,\\
    \label{fullydiscrete2}(D_{t_n}^{\alpha} u_h^n,v_h)-(\nabla \cdot \boldsymbol{\sigma}_h^n,v_h)+( \boldsymbol{b}^n \cdot \boldsymbol{B}^n \boldsymbol{\sigma}_h^n,v_h) + (c^n u_h^n,v_h) &=\lambda (\mathcal{I}Eu_h^n,v_h ) + (Ef^n,v_h )  \; \forall v_h \in V_h,\; 1\leq n \leq N,
\end{align}
where
   \begin{align}
    \label{2ndOrderExtrapolation} &E\phi^n = \begin{cases}
        \phi^{n-1} & \text{if}\; 1\leq n\leq n_\alpha:= \min\left(\lfloor{\frac{1}{\alpha}\rfloor},N\right) , \\
        (1+\mu_n)\phi^{n-1} - \mu_n \phi^{n-2} & \text{if}\; n \geq n_\alpha +1,
        \end{cases} 
\end{align}
$\mu_{n}=\Delta t_{n}/\Delta t_{n-1}, \; 2\leq n \leq N$, $\boldsymbol{B}^n(\boldsymbol{x}):=\boldsymbol{B}(\boldsymbol{x},t_n)$, $\boldsymbol{b}^n(\boldsymbol{x}):=\boldsymbol{b}(\boldsymbol{x},t_n)$, $c^n(\boldsymbol{x}):=c(\boldsymbol{x},t_n)$, and $f^n(\boldsymbol{x}):= f(\boldsymbol{x},t_n)$.

The following result will frequently be used in the stability and convergence analysis of the proposed method. For each $j,\; 0\leq j \leq N,$ $\vertiii{\cdot}_j:=\sqrt{(\boldsymbol{B}^j\cdot,\cdot)}$ is a norm on $L^2(\Omega)$ which is equivalent to $\|\cdot\|$-norm. In fact, from (\ref{supdc1}), there exist positive constants $\beta_0$ and $\gamma_0$ such that
\begin{align}
\label{equivalentnorms}&\beta_{0}\|\phi\|^{2}\leq \vertiii{\phi}_j^{2}\leq \gamma_{0}\|\phi\|^{2}\quad \forall \phi\in L^2(\Omega),\;\forall j:1\leq j\leq N.
\end{align}
\begin{remark} The choice $n_\alpha = \min\left(\lfloor{\frac{1}{\alpha}\rfloor}, N\right)$ in (\ref{2ndOrderExtrapolation}) is motivated by the discrete fractional Gr\"{o}nwall inequality (see Lemma~\ref{DFGI}). Furthermore, instead of applying the extrapolation (\ref{2ndOrderExtrapolation}), one may alternatively choose $Ef^n = f^n$ or $\mathcal{I}(Eu^n_h) = \mathcal{I}u_h^n$ without affecting the theoretical results presented in this article.
\end{remark} 
We now state the following result, which can be obtained by following a similar argument as in [Theorem~3.2, equation (3.13) and Section~3.3, equation (3.18), \cite{TOMAR2024137}]. However, for the reader's convenience, the proof is provided in Appendix~\ref{proofofLemma3.1}.
\begin{lemma}\label{dfdoi}
    Let $\phi_j\in L^2(\Omega),\;j=0,1,2,\ldots,N$. Then the discrete fractional differential operator $D_{t_n}^\alpha$ satisfies the following estimates-
\begin{align}
\nonumber &\frac{1}{2}D_{t_{n}}^{\alpha}\|\phi^{n}\|^{2} \leq \left(D_{t_{n}}^{\alpha}\phi^{n},\phi^{n}\right),\quad\text{and}\\
\nonumber &  \frac{1}{2}D_{t_{n}}^{\alpha}\vertiii{\phi^{n}}_n^{2}  \leq\left(D_{t_{n}}^{\alpha}\boldsymbol{B}^n\phi^{n},\phi^{n}\right) - \frac{1}{2} \sum_{j=0}^{n-1} K_{1-\alpha}^{n,j+1}(\vertiii{\phi^n}^2_{j+1}-\vertiii{\phi^n}^2_{j})\quad \forall n\in\{1,2,\ldots,N\}.
\end{align}
Further, if $\boldsymbol{B}$ is Lipschitz continuous with respect to the time variable and having the Lipschitz constant $L_B$, then
\begin{align}\label{pollutionterm}
&  \frac{1}{2}D_{t_{n}}^{\alpha}\vertiii{\phi^{n}}_n^{2}  \leq\left(D_{t_{n}}^{\alpha}\boldsymbol{B}^n\phi^{n},\phi^{n}\right) + \frac{L_B t_n^{1-\alpha}}{2\beta_0\Gamma(2-\alpha)}\vertiii{\phi^n}^2_{n}\quad \forall n\in\{1,2,\ldots,N\}.
\end{align}
\end{lemma}
%
\subsection{Stability Analysis}\label{section4}
%
To establish the stability of the proposed method, we first introduce certain properties of the discrete kernels $K_{1-\alpha}^{n,j}$ and their complementary discrete kernel, $P_{\alpha}^{n,i}$, as defined below: 
\begin{align}\label{complementarydiscretekernel}
&P_{\alpha}^{n,i} = \frac{1}{K_{1-\alpha}^{i,i}}\begin{cases}  
\displaystyle{\sum_{j=i+1}^n} P_{\alpha}^{n,j} \left(K_{1-\alpha}^{j,i+1} - K_{1-\alpha}^{j,i} \right) & \text{ : }  ~ 1 \leq i \leq n-1,\\
1 & \text{ : }  ~  i = n.
\end{cases}
\end{align}
These properties, presented in the following lemma, serve as key tools for deriving the subsequent stability and error estimates.
\begin{lemma}\label{discretekernelproperties}
The discrete kernels $K_{1-\alpha}^{n,j}$ and $P_{\alpha}^{n,j}$, $1\leq j\leq n,$ $1\leq n\leq N,$ satisfy the following results:
\begin{enumerate}
    \item[\namedlabel{itm:a}{(a)}] $0 \leq P_{\alpha}^{n,j} \leq \Gamma(2-\alpha) \Delta t_j^{\alpha}, ~ 1 \leq j \leq n. $
    \item[\namedlabel{itm:b}{(b)}] $\sum_{j=i}^n P_{\alpha}^{n,j} K^{j,i}_{1-\alpha} = 1, ~ 1 \leq i \leq n.$
    \item[\namedlabel{itm:c}{(c)}] $\sum_{j=1}^n P_{\alpha}^{n,j}k_{1+m\alpha-\alpha} (t_j) \leq k_{1+m\alpha}(t_n) \; \text{ for any non-negative integer } 0 \leq m \leq \lfloor{1/{\alpha}}\rfloor$.
    \item[\namedlabel{itm:e}{(d)}] $\nu \sum_{j=1}^{n-1}P_{\alpha}^{n,j}E_{\alpha}(\nu t_j^{\alpha}) \leq E_{\alpha}(\nu t_n^{\alpha})-1\quad \forall  \nu>0,\; \text{provided } \Delta t_{n-1}\leq \Delta t_{n}$, where $E_{\alpha}(z) := \sum^{\infty}
_{j=0}\frac{z^j}{\Gamma(j\alpha + 1)}$ is the Mittag-Leffler function.
    \item[\namedlabel{itm:f}{(e)}]$\sum_{j=1}^{n}P_{\alpha}^{n,j}\;t_j^{\beta-\alpha}\;\leq \; \frac{\Gamma{(1+\beta - \alpha)}}{\Gamma(1+\beta)}\;t_n^\beta\quad \forall \beta\in (0,1).$
    \item[\namedlabel{itm:g}{(f)}] $\sum_{j=1}^{n}P_{\alpha}^{n,j}\;t_j^{-\alpha}\;\leq \; 4e^\gamma\;\log_e(n+2), \;\text{ provided }\; t_n=\left(\frac{n}{N}\right)^\gamma T,\;\gamma\geq 1.$
    \item[\namedlabel{itm:h}{(g)}] $P^{n,n}_\alpha/ P^{n,n-1}_\alpha\;\leq\;\displaystyle\mu_n^\alpha/\alpha,\;n\geq 2.$ For a uniform mesh, i.e., when $\mu_{n} = 1,$ we have $$P^{n,n}_\alpha / P^{n,n-1}_\alpha\; =\;1/{(2-2^{1-\alpha})},\;n\geq 2.$$   
\end{enumerate} 
\end{lemma}
\begin{proof}
    The proof of \ref{itm:a}, \ref{itm:b}, \ref{itm:c} and \ref{itm:e} are provided in Lemma~2.1 and Corollary~4.1 of \cite{MR3790081}. The result \ref{itm:f} can be obtained by following the approach outlined in the proof of Lemma~5.3 in \cite{MR4246866}. The estimate \ref{itm:g} follows from \ref{itm:f}; however, for the sake of completeness and the reader's convenience, the proof is provided below. Define $\delta_n := \frac{1}{\log_e (n+2)}$ and consider
    \begin{align}\label{proplogfact}
         \sum_{j=1}^{n}P_{\alpha}^{n,j}\;t_j^{-\alpha}\;&=\; \sum_{j=1}^{n}P_{\alpha}^{n,j}\;t_j^{\delta_n-\alpha}\;t_j^{-\delta_n}\;\leq \; t_1^{-\delta_n}\;\sum_{j=1}^{n}P_{\alpha}^{n,j}\;t_j^{\delta_n-\alpha}.
    \end{align}
    Thus, an application of the estimate \ref{itm:f} with $\beta=\delta_n$ in (\ref{proplogfact}) yields the estimate \ref{itm:g} as follows
    \begin{align*}
         \sum_{j=1}^{n}P_{\alpha}^{n,j}\;t_j^{-\alpha}\;&\leq\; \frac{\Gamma(1+\delta_n -\alpha)}{\Gamma(1+\delta_n)} \;t_1^{-\delta_n}\;t_n^{\delta_n}\;=\; \frac{\Gamma(2+\delta_n -\alpha)}{(1+\delta_n -\alpha)\Gamma(1+\delta_n)} \;n^{\gamma\delta_n}\;\leq\; 4e^\gamma\log_e(n+2) ,
    \end{align*}    
    where we have used $\Gamma(1+\delta_n)\geq 1/2$, $\Gamma(2+\delta_n -\alpha)\leq 2$, $1+\delta_n -\alpha\geq \delta_n$ and $n^{\gamma\delta_n} \leq e^\gamma$ to obtain the last term in the above estimate. To stablish result \ref{itm:h}, we utilize relation (\ref{discretekernel}) to derive
    \begin{align}
        & K^{n,n}_{1-\alpha} = \frac{1}{\Delta t_n^\alpha\;\Gamma(2-\alpha)},\quad K^{n,n-1}_{1-\alpha}\;\leq\; \frac{1}{\Delta t_n^\alpha\;\Gamma(1-\alpha)},\label{lemma3.2additional1}
    \end{align}
    and for a uniform mesh(i.e., $\mu_{n} = 1$),
    \begin{align}
        & K^{n,n-1}_{1-\alpha} = \frac{2^{1-\alpha} -1}{\Delta t_n^\alpha\;\Gamma(2-\alpha)}.\label{lemma3.2additional2}
    \end{align}
    Using the complementary relation (\ref{complementarydiscretekernel}), we further derive
     \begin{align}
        & P^{n,n-1}_\alpha = P^{n,n}_\alpha \;\left(1-\frac{K^{n,n-1}_{1-\alpha}}{K^{n,n}_{1-\alpha}}\right)/\mu_n^\alpha.\label{lemma3.2additional3}
    \end{align}
Based on the estimates (\ref{lemma3.2additional1}), (\ref{lemma3.2additional2}), and (\ref{lemma3.2additional3}), we arrive at \ref{itm:h}, thus completing the remainder of the proof. 
\end{proof}

Due to the presence of the term $\mu_n := \Delta t_n / \Delta t_{n-1}$ in the extrapolation operator $E$ (see (\ref{2ndOrderExtrapolation})), previously established discrete fractional Gr\"{o}nwall inequalities, as discussed in \cite{MR4402734, MR3790081, MR3904430, MR4199354, TOMAR2024137}, impose severe time-step restrictions when applied to higher-order IMEX numerical methods (order > 1) in the temporal direction for PIDEs (\ref{pide}). 
In order to relax sever time-step restriction to a milder one, we present below a refined discrete fractional 
Gr\"{o}nwall inequality, which allows us to prove the stability and convergence of the proposed IMEX-L1-MFEM on a graded mesh for general time-fractional PIDEs (\ref{pide}), under a less restrictive time-step condition.
\begin{lemma}\label{DFGI}{(Discrete fractional Gr\"{o}nwall inequality).} 
    Let $\{v^{n}\}_{n=0}^{N}$, $\{\xi^{n}\}_{n=1}^{N}$, $\{\eta^{n}\}_{n=1}^{N}$ and $\{\zeta^{n}\}_{n=1}^{N}$ be non-negative finite sequences such that 
\begin{align}
    \label{dfgi1} 
     D_{t_n}^{\alpha}(v^{n})^2  \;&\leq \; \sum_{i=0}^{n}\lambda_{n-i}^{n}(v^{i})^{2} + v^{n}\xi^{n} + (\eta^{n})^{2}  + (\zeta^{n})^{2},\quad 1\leq n \leq N,
\end{align}
where $\lambda_{j}^{n}\;\geq\; 0,\;0\leq j\leq n,$ and the discrete fractional differential operator $D_{t_{n}}^{\alpha},\; 1 \leq n\leq N,$ is given by (\ref{L1scheme}). If there exists a constant $\Lambda\;>\;0$ such that, $\sum_{j=0}^{n}\lambda_{j}^{n}\;\leq\;\Lambda,\;1\leq n \leq N,$ and if $\Delta t_{n-1}\leq\Delta t_{n},\;2\leq n \leq N,$ with the maximum time-step size 
\begin{align} \label{timecondition}
    \Delta t:=\max_{1\leq n \leq N} \Delta t_{n}\;\leq\; {\left(\delta\;\Gamma{(2-\alpha)}\max\limits_{1\leq n \leq N}\lambda_0^n\right)^{-1/\alpha}}, \text{ for some } \delta>1,
    \end{align} 
    then, for $1\leq n \leq N$,  
\begin{align}
    \label{dfgi2}  v^{n}&  \leq C_\delta E_{\alpha}\left(C_\delta\Lambda_{n} t_{n}^{\alpha}\right)\left( v^0 + \max_{1\leq j\leq n}\sum_{i=1}^{j} P^{j,i}_{\alpha}  \xi^i  + \left(2 t_n^{\alpha} \right)^{1/2}\max_{1 \leq j \leq n}\eta^j + \max_{1\leq j\leq n} \left(\sum_{i=1}^{j} P^{j,i}_{\alpha} (\zeta^i)^2\right)^{1/2} \right),
\end{align}
where $C_\delta:=\frac{\delta}{\delta-1} $, $E_{\alpha}(z) := \sum^{\infty}
_{j=0}\frac{z^j}{\Gamma(j\alpha + 1)}$ is the Mittag-Leffler function, and
\begin{align*}
     \Lambda_n:&=\begin{cases}
    \Lambda &:\;1\leq n \leq n_\alpha:=\min\left(\lfloor{\frac{1}{\alpha}\rfloor},N\right),\\
    \Lambda\left(1+\max\limits_{n_\alpha +1\leq j \leq n}\frac{P_\alpha^{j,j}}{P_\alpha^{j,j-1}}\right)&:\;n_\alpha + 1\leq n \leq N
\end{cases} \\
 \;&\leq\; \Lambda(1+ \max\limits_{n_\alpha +1\leq j \leq n} \mu_j^\alpha/\alpha),\quad \mu_n:=\frac{\Delta t_n}{\Delta t_{n-1}}.
\end{align*}
\end{lemma}
\begin{proof}
An exchange of order of summation and Lemma~\ref{discretekernelproperties}\ref{itm:b} yields 
\begin{align}
\label{dfgieqn2}
\nonumber\displaystyle{\sum_{j=1}^{n}P^{n,j}_{\alpha} D_{t_j}^\alpha} \left(v^j\right)^2 &=\displaystyle{\sum_{j=1}^{n}P^{n,j}_{\alpha}\sum_{k=1}^{j}}K^{j,k}_{1-\alpha} \left((v^k)^2-(v^{k-1})^2\right) = \sum_{k=1}^{n}\left(\sum_{j=k}^n P^{n,j}_{\alpha}K^{j,k}_{1-\alpha} \right)\left((v^k)^2-(v^{k-1})^2\right) \\
&= (v^n)^2 - (v^0)^2 , \quad 1\leq n\leq N,
\end{align}
and then, using (\ref{dfgieqn2}) and (\ref{dfgi1}), we obtain
\begin{align}
\nonumber (v^n)^2 &\leq (v^0)^2 + \sum_{j=1}^{n} P^{n,j}_{\alpha} v^j\xi^j +  \sum_{j=1}^{n} P^{n,j}_{\alpha} \sum_{i=0}^j \lambda^j_{j-i}(v^i)^2 + \sum_{j=1}^{n} P^{n,j}_{\alpha} (\eta^j)^2+ \sum_{j=1}^{n} P^{n,j}_{\alpha} (\zeta^j)^2,\\
\label{dfgieqn1} &\leq  (v^0)^2  + \sum_{j=1}^{n} P^{n,j}_{\alpha} v^j \xi^j + \sum_{j=1}^{n} P^{n,j}_{\alpha} \sum_{i=0}^j \lambda^j_{j-i}(v^i)^2 + 2t_n^\alpha\max_{1\leq j \leq n} (\eta^j)^2+ \sum_{j=1}^{n} P^{n,j}_{\alpha} (\zeta^j)^2, \quad 1\leq n \leq N,
\end{align}
where the second-to-last term in (\ref{dfgieqn1}) has been estimated using Lemma~\ref{discretekernelproperties}\ref{itm:c} with $m=1$. 

Now, define a non-decreasing finite sequence $\{\Phi_n \}_{n=1}^{N}$ as follows:
\begin{align*}
\Phi_n :=  v^0 + \max_{1 \leq j \leq n}\sum_{i=1}^j P^{j,i}_{\alpha}\xi^i + \left(2t_n^{\alpha}\right)^{1/2}\max_{1 \leq j \leq n} \eta^j + \max_{1 \leq j \leq n} \left(\sum_{i=1}^{j} P^{j,i}_{\alpha} (\zeta^i)^2\right)^{1/2}, \quad 1\leq n \leq N.
\end{align*}
 Then the proof of (\ref{dfgi2}) boils down to establishing that 
\begin{align}\label{dfgieqn5}
    v^n \leq C_\delta  E_{\alpha}\left(C_\delta \Lambda_n t_n^{\alpha}\right)\;\Phi_n \quad \forall ~ 1 \leq n \leq N
\end{align}
by using mathematical induction in the following two steps:

\textbf{Step-I ($1\leq n \leq n_\alpha$):} As $\sum_{m=0}^{n}(C_\delta\Lambda)^m k_{1+m\alpha}(t_n)\leq \sum_{m=0}^{\infty}(C_\delta\Lambda)^m k_{1+m\alpha}(t_n)=:E_\alpha(C_\delta\Lambda t_n^\alpha),\;1\leq n\leq N$, it is enough to prove that
\begin{align}\label{dfgieqncaseI1}
    v^n \leq C_\delta \left(\sum_{m=0}^{n}(C_\delta\Lambda)^m k_{1+m\alpha}(t_n)\right)\;\Phi_n \quad \forall ~ 1 \leq n\leq  n_\alpha.
\end{align}
For $n=1$, if $v^1 < v^0$ or $v^1 < \left(2t_1^\alpha\right)^{1/2}\;\eta^1$ or $v^1 < \left(P_{\alpha}^{1,1}(\zeta^1)^2\right)^{1/2}$, then the result (\ref{dfgieqncaseI1}) holds trivially. Otherwise, $v^0 \leq v^1$, $\left(2t_1^\alpha\right)^{1/2}\;\eta^1 \leq v^1$, $\left(P_{\alpha}^{1,1}(\zeta^1)^2\right)^{1/2} \leq v^1$, and hence, the inequality (\ref{dfgieqn1}) implies 
\begin{align}
\nonumber (v^1)^2
&\leq \left(v^0 + P^{1,1}_{\alpha}\lambda_{1}^{1} v^0 + P^{1,1}_{\alpha}\lambda_{0}^{1} v^1 +  P^{1,1}_{\alpha} \xi^1  + \left(2 t_1^{\alpha}\right)^{1/2}\;\eta^1 + \left(P_{\alpha}^{1,1}(\zeta^1)^2\right)^{1/2} \right)v^1\\
\label{FDGI_Step-I1}&\leq (1+P^{1,1}_{\alpha}\lambda_{1}^{1})\Phi_1 v^1 +  P^{1,1}_{\alpha}\lambda_{0}^{1} (v^1)^2.
\end{align}
From condition (\ref{timecondition}) and Lemma~\ref{discretekernelproperties}\ref{itm:c} with $j=1$, we have $P^{1,1}_{\alpha}\lambda_{0}^{1} \leq 1/\delta,\;\delta>1,$ and $1+ P^{1,1}_{\alpha}\lambda_{1}^{1} \leq \sum_{m=0}^{1}(C_\delta\Lambda)^mk_{1+m\alpha}(t_1)$. Thus, we obtain (\ref{dfgieqncaseI1}) for $n=1$ from (\ref{FDGI_Step-I1}).

 Now, assume that (\ref{dfgieqncaseI1}) holds for $1 \leq k \leq m-1$, where $2\leq m \leq n_\alpha$. Then, there exists an integer $m_0$, $1 \leq m_0 \leq m-1$, such that $v^{m_0}= \max_{1 \leq j \leq m-1} v^j$. If $v^m \leq v^{m_0}$, the induction hypothesis, along with the properties $\Phi_n \leq \Phi_{n+1}$ and $k_{1+m\alpha}(s) \leq k_{1+m\alpha}(t)$ for $s \leq t$, yields the result (\ref{dfgieqncaseI1}) for $n = m$. Let $v^m > v^{m_0}.$ If $v^m < v^0$, or $v^m < \sqrt{2 t_m^{\alpha}}\max_{1 \leq j \leq m}\eta^j$, or $v^m < \max_{1 \leq j \leq m}\sqrt{\sum_{i=1}^jP_{\alpha}^{j,i}(\zeta^i)^2}$, then the inequality (\ref{dfgieqn5}) trivially holds for $n=m$. Otherwise, since $v^0 \leq v^m$, $\sqrt{2 t_m^{\alpha}}\max_{1 \leq j \leq m}\eta^j \leq v^m$, and $\max_{1 \leq j \leq m}\sqrt{\sum_{i=1}^jP_{\alpha}^{j,i}(\zeta^i)^2} \leq v^m$, applying (\ref{dfgieqn1}) and condition (\ref{timecondition}) shows
\begin{align*}
(v^m)^2 & \leq \left(\Phi_m+ \sum_{j=1}^{m-1} P^{m,j}_{\alpha} \sum_{i=0}^j \lambda^j_{j-i}v^i  + P^{m,m}_{\alpha} \sum_{i=0}^{m-1} \lambda^m_{m-i}v^i + P^{m,m}_{\alpha}\lambda^m_{0}v^m  \right)v^m\\
 & \leq \left(\Phi_m+ \sum_{j=1}^{m-1} P^{m,j}_{\alpha} \sum_{i=0}^j \lambda^j_{j-i}v^i  + P^{m,m}_{\alpha} \sum_{i=0}^{m-1} \lambda^m_{m-i}v^i  \right)v^m + \frac{1}{\delta}(v^m)^2.  
\end{align*}
Thus, we obtain
\begin{align*}
v^m & \leq C_\delta\left(\Phi_m+ \sum_{j=1}^{m-1} P^{m,j}_{\alpha} \sum_{i=0}^j \lambda^j_{j-i}v^i  + P^{m,m}_{\alpha} \sum_{i=0}^{m-1} \lambda^m_{m-i}v^i  \right), 
\end{align*}
and by applying the induction hypothesis, along with the properties $\Phi_n \leq \Phi_{n+1}$ and $k_{1+m\alpha}(s) \leq k_{1+m\alpha}(t)$ for $s \leq t$, we arrive at
\begin{align}
\nonumber v^m & \leq C_\delta\left(\Phi_m+ C_\delta\Lambda \sum_{j=1}^{m-1} P^{m,j}_{\alpha} \sum_{i=0}^j (C_\delta\Lambda)^i k_{1+i\alpha}(t_j)\Phi_j +C_\delta\Lambda P^{m,m}_{\alpha} \sum_{i=0}^{m-1} (C_\delta\Lambda)^i k_{1+i\alpha}(t_{m-1})\Phi_{m-1}  \right)\\
\label{FDGI_Step-I3}&\leq C_\delta\left(1+ C_\delta\Lambda \sum_{j=1}^{m} P^{m,j}_{\alpha} \sum_{i=0}^{m-1} (C_\delta\Lambda)^i k_{1+i\alpha}(t_j) \right)\Phi_m.
\end{align}
From Lemma~\ref{discretekernelproperties}\ref{itm:c}, it follows that $C_\delta\Lambda \sum_{j=1}^{m} P^{m,j}_{\alpha} \sum_{i=0}^{m-1} (C_\delta\Lambda)^i k_{1+i\alpha}(t_j) \leq  \sum_{i=1}^{m} (C_\delta\Lambda)^i k_{1+i\alpha}(t_m)$ for $1\leq m\leq n_\alpha$. Applying this estimate in (\ref{FDGI_Step-I3}), we conclude that (\ref{dfgieqncaseI1}) holds for $n=m$. By mathematical induction, (\ref{dfgieqncaseI1}) and, as a consequence, (\ref{dfgieqn5}) hold for all $n$ such that $1 \leq n \leq n_\alpha$.

\textbf{Step-II ($n_\alpha +1\leq n \leq N$):} For $n=n_\alpha +1$, if $v^{n_\alpha +1}<\max_{0\leq j \leq n_\alpha} v^j$, or $v^{n_\alpha +1} < \sqrt{2 t_{n_\alpha +1}^{\alpha}}\max_{1 \leq j \leq n_\alpha +1}\eta^j$, or $v^{n_\alpha +1} < \max_{1 \leq j \leq n_\alpha +1}\sqrt{\sum_{i=1}^j P_{\alpha}^{j,i}(\zeta^i)^2}$, then by applying the non-decreasing property of $E_\alpha(\cdot)$ and $\Phi_n$, the result (\ref{dfgieqn5}) follows for $n=n_\alpha +1$. Otherwise, $\max_{0\leq j \leq n_\alpha} v^j \leq v^{n_\alpha +1}$,  $ \sqrt{2 t_{n_\alpha +1}^{\alpha}}\max_{1 \leq j \leq n_\alpha +1}\eta^j \leq v^{n_\alpha +1} $, and $ \max_{1 \leq j \leq n_\alpha +1}\sqrt{\sum_{i=1}^j P_{\alpha}^{j,i}(\zeta^i)^2}\leq v^{n_\alpha +1}$, and therefore, by applying (\ref{dfgieqn1}) and condition (\ref{timecondition}), we obtain
\begin{align}
\nonumber (v^{n_\alpha +1})^2 & \leq \left(\Phi_{n_\alpha +1}+ \sum_{j=1}^{n_\alpha } P^{n_\alpha +1,j}_{\alpha} \sum_{i=0}^j \lambda^j_{j-i}v^i  + P^{n_\alpha +1,n_\alpha +1}_{\alpha} \sum_{i=0}^{n_\alpha } \lambda^{n_\alpha +1}_{n_\alpha +1-i}v^i + P^{n_\alpha +1,n_\alpha +1}_{\alpha}\lambda^{n_\alpha +1}_{0}v^{n_\alpha +1}  \right)v^{n_\alpha +1}\\
\label{FDGI_Step-I4} & \leq \left(\Phi_{n_\alpha +1}+ \sum_{j=1}^{n_\alpha } P^{n_\alpha +1,j}_{\alpha} \sum_{i=0}^j \lambda^j_{j-i}v^i  + P^{n_\alpha +1,n_\alpha +1}_{\alpha} \sum_{i=0}^{n_\alpha } \lambda^{n_\alpha +1}_{n_\alpha +1-i}v^i  \right)v^{n_\alpha +1} + \frac{1}{\delta}(v^{n_\alpha +1})^2.  
\end{align}
Apply (\ref{dfgieqn5}) to (\ref{FDGI_Step-I4}) for $1 \leq n \leq n_\alpha$ to find that
\begin{align}
\nonumber v^{n_\alpha +1} &  \leq C_\delta\left(\Phi_{n_\alpha +1}+ C_\delta\Lambda \sum_{j=1}^{n_\alpha } P^{n_\alpha +1,j}_{\alpha} E_\alpha\left(C_\delta\Lambda t_j^\alpha\right)\Phi_j  +C_\delta\Lambda P^{n_\alpha +1,n_\alpha +1}_{\alpha} E_\alpha\left(C_\delta\Lambda t_{n_\alpha}^\alpha\right)\Phi_{n_\alpha}  \right)\\
\label{FDGI_Step-I5}& \leq C_\delta\left(1+ C_\delta\Lambda_{n_\alpha +1} \sum_{j=1}^{n_\alpha } P^{n_\alpha +1,j}_{\alpha} E_\alpha\left(C_\delta\Lambda_{n_\alpha +1} t_j^\alpha\right)   \right)\Phi_{n_\alpha +1}.
\end{align}
Thus, an appeal to Lemma~\ref{discretekernelproperties}\ref{itm:e} in (\ref{FDGI_Step-I5}) yields the result (\ref{dfgieqn5}) for $n=n_\alpha +1$.

 Now, assume that (\ref{dfgieqn5}) holds for $1 \leq k \leq m-1$, where $n_{\alpha} + 2\leq m \leq N$. Then, there exists an integer $m_0$, $0 \leq m_0 \leq m-1$, such that $v^{m_0}= \max_{0 \leq j \leq m-1} v^j$. If $v^{m}<v^{m_0}$, or $v^{m} < \sqrt{2 t_{m}^{\alpha}}\max_{1 \leq j \leq m}\eta^j$, or $v^{m} < \max_{1 \leq j \leq m}\sqrt{\sum_{i=1}^j P_{\alpha}^{j,i}(\zeta^i)^2}$, then by applying the non-decreasing property of $E_\alpha(\cdot)$ and $\Phi_n$, the result (\ref{dfgieqn5}) follows for $n=m$. Otherwise, $v^{m_0}\leq v^{m}$, $ \sqrt{2 t_{m}^{\alpha}}\max_{1 \leq j \leq m}\eta^j\leq v^{m}$, and $\max_{1 \leq j \leq m}\sqrt{\sum_{i=1}^j P_{\alpha}^{j,i}(\zeta^i)^2}\leq v^{m}$, and hence, by applying (\ref{dfgieqn1}), condition (\ref{timecondition}), and the induction hypothesis, one can obtain 
 \begin{align}
\label{FDGI_Step-I6} v^{m} & \leq C_\delta\left(1+ C_\delta\Lambda_{m} \sum_{j=1}^{m-1 } P^{m,j}_{\alpha} E_\alpha\left(C_\delta\Lambda_{m} t_j^\alpha\right)   \right)\Phi_{m}.
\end{align}
An appeal to Lemma~\ref{discretekernelproperties}\ref{itm:e} in (\ref{FDGI_Step-I6}) yields that (\ref{dfgieqn5}) holds for $n=m$. Thus, the principle of mathematical induction confirms that the result (\ref{dfgieqn5}) holds for all $n$, $1 \leq n \leq N$. Finally, using Lemma~\ref{discretekernelproperties}\ref{itm:h}, we can estimate $\Lambda_n$ in terms of $\mu_n=\Delta t_n/\Delta t_{n-1}$ as follows:
\begin{align*}
    & \Lambda_n \leq \Lambda(1+ \max\limits_{n_\alpha +1\leq j \leq n} \mu_j^\alpha/\alpha).
\end{align*}
This concludes the rest of the proof.
\end{proof}
\begin{remark}
Observe that while Lemma~\ref{DFGI} relaxes the time-step restriction to some extent, it results in a larger constant in the Mittag-Leffler function in (\ref{dfgi2}) compared to earlier discrete fractional Gr\"{o}nwall inequalities, as noted in \cite{MR4402734, MR3790081, MR3904430, MR4199354, TOMAR2024137}.
\end{remark}
\begin{remark}
    For the graded mesh $\left\{t_n\right\}_{n=1}^{N}$ defined in (\ref{gradedmesh1}), the limit $\lim\limits_{N\to\infty}\;\max\limits_{n_\alpha +1\leq j \leq N}\frac{P_\alpha^{j,j}}{P_\alpha^{j,j-1}}$ exists, and satisfies $\lim\limits_{N\to\infty}\;\max\limits_{n_\alpha +1\leq j \leq N}\frac{P_\alpha^{j,j}}{P_\alpha^{j,j-1}}\;\leq\;\mu_{n_\alpha+1}^\alpha/\alpha$. This limit depends solely on $\Lambda$, $\alpha$, and the grading parameter $\gamma$.
\end{remark} 
\begin{remark}
    For a uniform partition, i.e., when $\Delta t_{n-1} = \Delta t_n,\;n\geq 2,$ we have $\Lambda_n\leq \Lambda + \Lambda/(2-2^{1-\alpha}) $ which is comparable to the bound found in \cite{MR3870961, MR3914223, MR3827604}. 
\end{remark}
%
As the problem (\ref{fullydiscrete1}--\ref{fullydiscrete2}) is linear, to ensure stability, it is enough to derive \textit{a priori} estimates for $u_{h}^{n}$ and $\boldsymbol{\sigma}_h^n $ for $1\leq n \leq N,$ in terms of $u_{h}^{0}$ and $f^{n}$.
\begin{theorem}\label{stabilitytheorem1} If for some $\delta>1$, $\max\limits_{1\leq n \leq N}\Delta t_n\leq \left( \delta\;\lambda^S\Gamma(2-\alpha)\right)^{-1/\alpha}$, then the solution $u_h^n$ of the problem (\ref{fullydiscrete1}--\ref{fullydiscrete2}) satisfies
\begin{align}
\label{L2stabilityforu}&\|u_h^n\| \leq C_S E_{\alpha}(C_\delta \Lambda_{\infty}^S t_n^{\alpha})\left( \|u_h^0\| + \max_{1\leq j \leq n} \sum_{i=1}^jP_{\alpha}^{j,i}\|Ef^i\|\right), \quad 1 \leq n \leq N,  
\end{align}  
where 
\begin{align*}
    &C_S:=2C_\delta,\;C_\delta := \delta/(\delta-1),\;\lambda^S:=  \frac{1}{2}\|\widetilde{b}\|_{L^\infty(\Omega\times J)}  + 2\|\widetilde{c}\|_{L^\infty(\Omega\times J)} 
 + \epsilon|\lambda |,\;\epsilon>0,\\
 & \widetilde{c}(x,t):= \max(0,-c(x,t)),\; \widetilde{b}:=\boldsymbol{b}^T\boldsymbol{A}^{-1}\boldsymbol{b},\\
& \Lambda_{\infty}^S:=\lim_{n\to\infty}\Lambda_{n}^S,\quad \Lambda_{n}^S:=\begin{cases}
    \Lambda^S &:\;1\leq n \leq n_\alpha:=\min\left(\lfloor{\frac{1}{\alpha}\rfloor},N\right),\\
    \Lambda^S\left(1+\max\limits_{n_\alpha +1\leq j \leq n}\frac{P_\alpha^{j,j}}{P_\alpha^{j,j-1}}\right)&:\;n_\alpha + 1\leq n \leq N,
\end{cases} \quad\text{ and}\\
 & \Lambda^S :=\left(  \frac{1}{2}\|\widetilde{b}\|_{L^\infty(\Omega\times J)}  + 2\|\widetilde{c}\|_{L^\infty(\Omega\times J)} 
 + \epsilon|\lambda |\right) +  2\epsilon^{-1}|\lambda |C_{\mathcal{I}}^2\left((1+\mu_{n_\alpha+1})^2+\mu_{n_\alpha+1}^2\right).
\end{align*}
\end{theorem}
\begin{proof}
Choose $\boldsymbol{w}_h = \boldsymbol{\sigma}^n_h$ and $v_h = u_h^n$ in (\ref{fullydiscrete1}) and (\ref{fullydiscrete2}), respectively, and then add (\ref{fullydiscrete1}) to (\ref{fullydiscrete2}) to obtain
\begin{align*}
(D_{t_n}^{\alpha} u_h^n,u_h^n)+ (\boldsymbol{B}^n\boldsymbol{\sigma}^n_h, \boldsymbol{\sigma}^n_h) +( \boldsymbol{b}^n \cdot \boldsymbol{B}^n \boldsymbol{\sigma}_h^n,u_h^n) + (c^nu_h^n,u_h^n) ~&=~\lambda (\mathcal{I}Eu_h^n,u_h^n ) + ( Ef^n,u_h^n). 
\end{align*}
Now, using Lemma~\ref{dfdoi}, the Cauchy–Schwarz inequality, and the relation (\ref{intbound}), we obtain
\begin{align}
\nonumber \frac{1}{2}D_{t_n}^{\alpha}\|u_h^n\|^2 +\vertiii{\boldsymbol{\sigma}_h^n}_n^2 & \leq (D_{t_n}^{\alpha} u_h^n,u_h^n)+ (\boldsymbol{B}^n\boldsymbol{\sigma}^n_h, \boldsymbol{\sigma}^n_h)\\
\nonumber&= - ( \boldsymbol{b}^n \cdot \boldsymbol{B}^n \boldsymbol{\sigma}_h^n,u_h^n) - (c^n u_h^n,u_h^n) + \lambda (\mathcal{I}Eu_h^n,u_h^n ) + ( Ef^n,u_h^n)  \\
\label{stability2ast}& \leq  \vertiii{\boldsymbol{b}^n}_n\vertiii{\boldsymbol{\sigma}_h^n}_n\|u_h^n\| +\|\widetilde{c}\|_{L^\infty(\Omega\times J)} \|u_h^n\|^2+C_{\mathcal{I}}|\lambda |\|Eu_h^n\|\;\|u_h^n\| +   \|Ef^n\|\;\|u_h^n\|.
\end{align}
 Applying Young's inequality $ab\leq \epsilon a^2 + (1/{4\epsilon})b^2$, for an appropriate $\epsilon>0$, we obtain $\vertiii{\boldsymbol{b}^n}_n\vertiii{\boldsymbol{\sigma}_h^n}_n\|u_h^n\|\leq \vertiii{\boldsymbol{\sigma}_h^n}_n^2 + (1/4)\vertiii{\boldsymbol{b}^n}_n^2\|u_h^n\|^2$ and $C_{\mathcal{I}}|\lambda |\|Eu_h^n\|\;\|u_h^n\|\leq (\epsilon|\lambda |/2) \|u_h^n\|^2 + (|\lambda |/2\epsilon)(C_{\mathcal{I}}\|Eu_h^n\|)^2$. Substituting these estimates into (\ref{stability2ast}), we arrive at
\begin{align}\label{stability2}
D_{t_n}^{\alpha}\|u_h^n\|^2  & \leq \left( \frac{1}{2}\vertiii{\boldsymbol{b}^n}_n^2 + 2\|\widetilde{c}\|_{L^\infty(\Omega\times J)} 
 + \epsilon|\lambda | \right)\|u_h^n\|^2  +\epsilon^{-1}|\lambda |C_{\mathcal{I}}^2\|Eu_h^n\|^2  +   2\|Ef^n\|\;\|u_h^n\| \nonumber\\ 
 &\leq  \sum_{i=0}^{n}\lambda_{n-i}^{n}\|u_{h}^{i}\|^2+ 2\|Ef^n\|\;\|u_h^n\|,
\end{align}
where 
\begin{align*}
& \lambda^n_{n-i} := \begin{cases}
 \frac{1}{2}\|\widetilde{b}\|_{L^\infty(\Omega\times J)}  + 2\|\widetilde{c}\|_{L^\infty(\Omega\times J)} 
 + \epsilon|\lambda | &: \;  1\leq n \leq N\;\&\; i=n, \\
\epsilon^{-1}|\lambda |C_{\mathcal{I}}^2 &: \; 1\leq n \leq n_\alpha\;\&\; i=n-1, \\ 
0&: \; 1\leq n \leq n_\alpha\;\&\; i=n-2, \\ 
 2\epsilon^{-1}|\lambda |C_{\mathcal{I}}^2(1+\mu_{n_\alpha+1})^2 &: \;  n_{\alpha}+1\leq n \leq N\;\&\; i=n-1, \\
 2\epsilon^{-1}|\lambda |C_{\mathcal{I}}^2\mu_{n_\alpha+1}^2 &: \;  n_{\alpha}+1\leq n \leq N\;\&\; i=n-2, \\
0 &: \;  3\leq n \leq N\;\&\; 0 \leq i \leq n-3, \\
    \end{cases}
    \end{align*}
    and
    \begin{align}
\nonumber  \displaystyle{\sum_{i=0}^n \lambda_{n-i}^n} &= \left(  \frac{1}{2}\|\widetilde{b}\|_{L^\infty(\Omega\times J)}  + 2\|\widetilde{c}\|_{L^\infty(\Omega\times J)} 
 + \epsilon|\lambda |\right) +  2\epsilon^{-1}|\lambda |C_{\mathcal{I}}^2\left((1+\mu_{n_\alpha+1})^2+\mu_{n_\alpha+1}^2\right)  = \max_{1 \leq n \leq N} \sum_{i=0}^n \lambda_{n-i}^n =:\Lambda^S.
\end{align}
An appeal to the discrete fractional Gr\"{o}nwall inequality (Theorem~\ref{DFGI}) to (\ref{stability2}) yields
\begin{align*}
&\|u_h^n\| \leq C_\delta E_{\alpha}(C_\delta \Lambda_{n}^S t_n^{\alpha})\left( \|u_h^0\| +2 \max_{1\leq j \leq n} \sum_{i=1}^jP_{\alpha}^{j,i}\|Ef^i\|\right), \quad 1 \leq n \leq N. 
\end{align*}  
Finally, using $\Lambda_n^S\leq \Lambda_\infty^S$ and the non-decreasing property of $E_\alpha(\cdot)$, we obtain the desired estimate (\ref{L2stabilityforu}).
\end{proof}

The next result deals with the stability of flux using the proposed method.
\begin{theorem}\label{stabilitytheorem2} If for some $\epsilon>0$ and $\delta>1$, $\max\limits_{1\leq n \leq N}\Delta t_n\leq \left( \delta\;\lambda^F_\epsilon\Gamma(2-\alpha)\right)^{-1/\alpha}$, then flux $\boldsymbol{\sigma}_h^n$ of the problem (\ref{fullydiscrete1}--\ref{fullydiscrete2}) satisfies
\begin{align}\label{L2stabilityflux}
   & \|\boldsymbol{\sigma}_h^n\| \leq C_F  E_{\alpha}( C_\delta \Lambda_\infty^F t_{n}^{\alpha})\left( \|\boldsymbol{\sigma}_h^{0} \|+\|u_h^0\| + \max_{1\leq j \leq n}\left( \|Ef^{j}\| + \sum_{i=1}^{j}P^{j,i}_\alpha\|Ef^{i}\|\right)\right), \quad 1 \leq n \leq N,
\end{align}
where 
\begin{align*}
    & \lambda^F_\epsilon:=\frac{1+\epsilon}{2}\|\widetilde{b}\|_{L^\infty(\Omega\times J)} + \frac{L_B T^{1-\alpha}}{\beta_0\Gamma(2-\alpha)} ,\quad C_\delta:=\frac{\delta}{\delta-1},\\
    & \Lambda_{\infty}^F:=\lim_{n\to\infty}\Lambda_{n}^F,\quad \Lambda_{n}^F:=\begin{cases}
    \lambda^F_\epsilon &:\;1\leq n \leq n_\alpha:=\min\left(\lfloor{\frac{1}{\alpha}\rfloor},N\right),\\
    \lambda^F_\epsilon\left(1+\max\limits_{n_\alpha +1\leq j \leq n}\frac{P_\alpha^{j,j}}{P_\alpha^{j,j-1}}\right)&:\;n_\alpha + 1\leq n \leq N,
    \end{cases}\\
    & C_F:=\frac{C_\delta}{\sqrt{\beta_0}}\max\left\{ 1, \sqrt{\gamma_0}, \left(2T^\alpha\frac{1+\epsilon}{2\epsilon}\right)^{1/2}, C_{\delta} E_\alpha(C_\delta \Lambda_\infty^S T^\alpha)\left(2T^\alpha\frac{1+\epsilon}{2\epsilon}\right)^{1/2} \left( C_c+ C_{\mathcal{I}}|\lambda |(1+2\mu_{n_\alpha +1})  \right) \right\},\;\text{ and}\\
 &\text{the function $\widetilde{b}$ is defined in Theorem~\ref{stabilitytheorem1}.}
\end{align*}
\end{theorem}
\begin{proof}
We start by taking the discrete fractional derivative of (\ref{fullydiscrete1}):
\begin{align}\label{newfullydiscrete}   (D_{t_n}^{\alpha}\boldsymbol{B}^n\boldsymbol{\sigma}^n_h, \boldsymbol{w}_h) + (D_{t_n}^{\alpha} u_h^n, \nabla \cdot  \boldsymbol{w}_h)~ &=~0.
\end{align}
Choosing $v_h = D^{\alpha}_{t_n} u_h^n$ and $\boldsymbol{w}_h = \boldsymbol{\sigma}^n_h$  in (\ref{fullydiscrete2}) and (\ref{newfullydiscrete}), respectively, and adding them, we get 
\begin{align*}
\|D_{t_n}^{\alpha}u_h^n\|^2 +(D_{t_n}^{\alpha}\boldsymbol{B}^n\boldsymbol{\sigma}^n_h, \boldsymbol{\sigma}^n_h)+( \boldsymbol{b}^n \cdot \boldsymbol{B}^n \boldsymbol{\sigma}_h^n,D^{\alpha}_{t_n} u_h^n) + (c^n u_h^n,D^{\alpha}_{t_n} u_h^n) ~&=~\lambda (\mathcal{I} E u_h^n,D^{\alpha}_{t_n} u_h^n ) + (Ef^n,D^{\alpha}_{t_n} u_h^n).  
\end{align*}
Applying Lemma~\ref{dfdoi}, the Cauchy-Schwarz inequality, and relations (\ref{intbound}) and (\ref{coefbounds}), we obtain
\begin{align}
\|D_{t_n}^{\alpha}u_h^n\|^2 +\frac{1}{2}D_{t_n}^{\alpha}\vertiii{\boldsymbol{\sigma}_h^n}_n^2  \;\leq \;& \frac{L_B t_n^{1-\alpha}}{2\beta_0\Gamma(2-\alpha)}\vertiii{\boldsymbol{\sigma}^n_h}_n^2 - ( \boldsymbol{b}^n \cdot \boldsymbol{B}^n \boldsymbol{\sigma}_h^n,D^{\alpha}_{t_n} u_h^n) - (c^n u_h^n,D^{\alpha}_{t_n} u_h^n)\nonumber\\
&+ \lambda (\mathcal{I} E u_h^n,D^{\alpha}_{t_n} u_h^n ) + ( Ef^n,D^{\alpha}_{t_n} u_h^n) \nonumber\\
\;\leq \;& \frac{L_B t_n^{1-\alpha}}{2\beta_0\Gamma(2-\alpha)}\vertiii{\boldsymbol{\sigma}^n_h}_n^2 + \vertiii{\boldsymbol{b}^n}_n\vertiii{\boldsymbol{\sigma}^n_h}_n \|D_{t_n}^{\alpha}u_h^n\| + C_c   \|u_h^n\|\;\|D_{t_n}^{\alpha}u_h^n\|\nonumber\\
&+|\lambda| C_{\mathcal{I}}\|Eu_h^n\|\;\|D_{t_n}^{\alpha}u_h^n\| + \|Ef^n\|\;\|D^{\alpha}_{t_n} u_h^n\|\label{newfullydiscrete_ast}. 
\end{align}
Using Young's inequality $ab\leq \epsilon a^2 + (1/{4\epsilon})b^2$, for an appropriate $\epsilon>0$, we arrive at the following estimates
\begin{align*}
&\vertiii{\boldsymbol{b}^n}_n\vertiii{\boldsymbol{\sigma}_h^n}_n\|D^{\alpha}_{t_n} u_h^n\|\leq  \frac{1}{1+\epsilon}\|D^{\alpha}_{t_n} u_h^n\|^2+ \frac{1+\epsilon}{4}\vertiii{\boldsymbol{b}^n}_n^2\vertiii{\boldsymbol{\sigma}_h^n}_n^2,\quad\text{and}\\
&\left(C_c\|u_h^n\|+ C_{\mathcal{I}}|\lambda |\|Eu_h^n\| + \|Ef^n\|\right)\|D^{\alpha}_{t_n} u_h^n\|\leq \frac{\epsilon}{1+\epsilon}\|D^{\alpha}_{t_n} u_h^n\|^2+ \frac{1+\epsilon}{4\epsilon}( C_c\|u_h^n\|+ C_{\mathcal{I}}|\lambda |\|Eu_h^n\| + \|Ef^n\| )^2.
\end{align*} Substituting these two estimates into (\ref{newfullydiscrete_ast}), we obtain
\begin{align}
D_{t_n}^{\alpha}\vertiii{\boldsymbol{\sigma}_h^n}_n^2& \leq \left(\frac{1+\epsilon}{2}\|\widetilde{b}\|_{L^\infty(\Omega\times J)} + \frac{L_B T^{1-\alpha}}{\beta_0\Gamma(2-\alpha)}  \right)\vertiii{\boldsymbol{\sigma}^n_h}_n^2 + \frac{1+\epsilon}{2\epsilon}\left( C_c\|u_h^n\|+ C_{\mathcal{I}}|\lambda |\|Eu_h^n\| + \|Ef^n\|  \right)^2. \label{newfullydiscrete_astast}
\end{align}
Under the assumption $\max\limits_{1\leq n \leq N}\Delta t_n\leq \left( \delta\;\lambda^F_\epsilon\Gamma(2-\alpha)\right)^{-1/\alpha}$, applying the discrete fractional Gr\"{o}nwall inequality (Theorem~\ref{DFGI}) to (\ref{newfullydiscrete_astast}), we obtain
\begin{align}
\vertiii{\boldsymbol{\sigma}_h^n}_n & \leq C_\delta E_\alpha\left(C_\delta \Lambda_{n}^F t_n^\alpha \right) \left(\vertiii{\boldsymbol{\sigma}^0_h}_0 + \left(2t_n^\alpha\frac{1+\epsilon}{2\epsilon}\right)^{1/2}\max_{1\leq j \leq n}\left( C_c\|u_h^j\|+ C_{\mathcal{I}}|\lambda |\|Eu_h^j\| + \|Ef^j\|  \right) \right).\label{newfullydiscrete_astastast}
\end{align}
Now, an application of (\ref{equivalentnorms}) with $\Lambda_n^F\leq \Lambda_\infty^F$, and $\max\limits_{n_\alpha +1\leq j\leq n}\mu_j\leq \mu_{n_\alpha +1}$ in (\ref{newfullydiscrete_astastast}) shows
\begin{align}
\|\boldsymbol{\sigma}_h^n\|&  \leq \frac{C_\delta}{\sqrt{\beta_0}} E_\alpha\left(C_\delta \Lambda_{\infty}^F t_n^\alpha \right) \Bigg(\sqrt{\gamma_0} \|\boldsymbol{\sigma}^0_h\| + \Big(2T^\alpha\frac{1+\epsilon}{2\epsilon}\Big)^{1/2}\Big( (C_c+ C_{\mathcal{I}}|\lambda |(1+2\mu_{n_\alpha +1})) \max_{0\leq j \leq n}\|u_h^j\| + \max_{1\leq j \leq n}\|Ef^j\|  \Big) \Bigg). \label{newfullydiscrete_astastastast}
\end{align}
Finally, an appeal to the Theorem~\ref{stabilitytheorem1} in (\ref{newfullydiscrete_astastastast}) yields the result (\ref{L2stabilityflux}). This completes the rest of the proof.
\end{proof}

\subsection{Error analysis}\label{section5}
This subsection derives some auxiliary results and then establishes optimal error estimates. At any temporal grid point $t_n$, the variational problem (\ref{variation1}--\ref{variation2}) implies, for $1\leq n \leq N$,
\begin{align}
\label{variationalatt_n1} (\boldsymbol{B}(t_n)\boldsymbol{\sigma}(t_n), \boldsymbol{w}_h) + ( u(t_n),  \nabla \cdot \boldsymbol{w}_h)~ &=~0 \quad \forall \boldsymbol{w}_h \in \boldsymbol{W}_h,\\
\nonumber(D_{t_n}^{\alpha}u(t_n),v_h)-(\nabla \cdot \boldsymbol{\sigma}(t_n),v_h) +( \boldsymbol{b}(t_n) \cdot \boldsymbol{B}(t_n) \boldsymbol{\sigma}(t_n),v_h) + (c(t_n)u(t_n),v_h) ~&=~\lambda (\mathcal{I}u(t_n),v_h ) \\
\label{variationalatt_n2}+ (f(t_n),v_h) + (\Upsilon^n +r^n, v_h) &\; \forall\;v_h \in V_h.   
\end{align}
After subtracting (\ref{fullydiscrete1}--\ref{fullydiscrete2}) from (\ref{variationalatt_n1}--\ref{variationalatt_n2}), we obtain the following error equation, for $1\leq n \leq N,$
\begin{align}
\label{errorequation1}
(\boldsymbol{B}^n e_{\boldsymbol{\sigma}}^n, \boldsymbol{w}_h) + (e_u^n, \nabla  \cdot \boldsymbol{w}_h)~ &=~0 \quad \forall \boldsymbol{w}_h \in \boldsymbol{W}_h, \\
\label{errorequation2}
(D_{t_n}^{\alpha}e^n_u,v_h) -(\nabla \cdot e_{\boldsymbol{\sigma}}^n,v_h) +( \boldsymbol{b}^n \cdot \boldsymbol{B}^n e_{\boldsymbol{\sigma}}^n,v_h) + (c^n e_{u}^n,v_h)  &= \lambda (\mathcal{I}(Ee^n_u),v_h ) +(\Upsilon^n+ r^n + r_f^n, v_h) \quad \forall\;v_h \in V_h,   
\end{align}
where $\Upsilon^n:= (D^{\alpha}_{t_n}u(t_n) - \partial^{\alpha}_{t}u(t_n))$, $r^n:= \lambda \mathcal{I}(u(t_n) - Eu(t_n))$, and $r_f^n:= f(t_n) - Ef(t_n))$. Also, $e^n_u:= u(t_n)-u_h^n$ and $e^n_{\boldsymbol{\sigma}}:= \boldsymbol{\sigma}(t_n)-\boldsymbol{\sigma}_h^n$ denote the error between the exact solution $u(t_n)$ and the approximate solution $u_h^n$ and between the flux $\boldsymbol{\sigma}(t_n)$ and its approximation $\boldsymbol{\sigma}_h^n$, respectively, at time level $t = t_n$. To obtain optimal error estimates, we decompose the errors $e^n_u$ and $e^n_{\boldsymbol{\sigma}}$ further as follows:
\begin{align*}
    &e^n_u = \eta^n + \theta^n, \; \eta^n = u(t_n)- P_h u(t_n), \; \theta^n =  P_h u(t_n) - u_h^n ,\\
    &e^n_{\boldsymbol{\sigma}} =   \boldsymbol{\zeta}^n + \boldsymbol{\xi}^n, \; \boldsymbol{\zeta}^n = \boldsymbol{\sigma}(t_n) - \Pi_h \boldsymbol{\sigma}(t_n) , \; \boldsymbol{\xi}^n =  \Pi_h \boldsymbol{\sigma}(t_n)  - \boldsymbol{\sigma}_h^n.
\end{align*}
 As the estimates for the projection errors $\eta^n:=u(t_n)- P_h u(t_n) $ and $\boldsymbol{\zeta}^n:=\boldsymbol{\sigma}(t_n) - \Pi_h \boldsymbol{\sigma}(t_n) $ are already known from (\ref{approxprop}), it is enough to estimate $\theta^n$ and $\boldsymbol{\xi}^n$. Using (\ref{errorequation1}-\ref{errorequation2}) $\theta^n$ and $\boldsymbol{\xi}^n$ satisfy
 \begin{align}
     \label{theta1}(\boldsymbol{B}^n ({\boldsymbol{\xi}}^n + {\boldsymbol{\zeta}}^n), \boldsymbol{w}_h) + ( \theta^n,  \nabla \cdot \boldsymbol{w}_h)~ &=~0 \quad \forall \boldsymbol{w}_h \in \boldsymbol{W}_h, \\
\nonumber
(D_{t_n}^{\alpha}\theta^n,v_h) -(\nabla \cdot {\boldsymbol{\xi}}^n,v_h)  &= \lambda (\mathcal{I}(E (\theta^n+\eta^n)),v_h )-(D_{t_n}^{\alpha}\eta^n,v_h)+ (\Upsilon^n + r^n + r_f^n, v_h)  \\
\label{theta2}& -( \boldsymbol{b}^n \cdot \boldsymbol{B}^n   (\boldsymbol{\xi}^n+ \boldsymbol{\zeta}^n) ,v_h)- (c^n (  \theta^n + \eta^n ),v_h)\quad \forall\;v_h \in V_h, \; 1\leq n \leq N.  
 \end{align}
\begin{lemma}\label{thetaestimate_l2norm}
If for some $\delta>1$ and $\epsilon>0$, $\max\limits_{1\leq n \leq N}\Delta t_n\leq \left( \delta\;\lambda^S_\epsilon\Gamma(2-\alpha)\right)^{-1/\alpha}$, then $\theta^{n}$ satisfies
\begin{align}
\|\theta^n\| \;\leq \;& C_\delta E_{\alpha}(C_\delta \widetilde{\Lambda}_{\infty}^S t_n^{\alpha})\Bigg( \|\theta^0\|  + 2\max_{1\leq j \leq n} \sum_{i=1}^j P_{\alpha}^{j,i}\Big(\|\Upsilon^i \|+ \|r^i\|+ \|r_f^i\|  + \| D^{\alpha}_{t_i}\eta^i\| + C_c\|\eta^i\| + C_{\mathcal{I}}|\lambda|\|E\eta^i\|\nonumber\\
&\hspace{1.5cm}   + \big(\gamma_0  \|\widetilde{b}\|_{L^\infty(\Omega\times J)}\big)^{1/2}\|\boldsymbol{\zeta}^{i}\| \Big) + \max_{1\leq j \leq n}\Big(\gamma_0\frac{1+\epsilon}{2\epsilon} \sum_{i=1}^{j}P_{\alpha}^{j,i} \|\boldsymbol{\zeta}^{i}\|^2 \Big)^{1/2}  \Bigg), \quad 1 \leq n \leq N, \label{L2errorestimateeqn1} 
\end{align}  
where
\begin{align*}
    &C_\delta := \delta/(\delta-1),\;\lambda^S_\epsilon:=  \frac{1+\epsilon}{2}\|\widetilde{b}\|_{L^\infty(\Omega\times J)}  + 2\|\widetilde{c}\|_{L^\infty(\Omega\times J)} 
 + \epsilon|\lambda |,\\
&\widetilde{\Lambda}_{\infty}^S:=\lim_{n\to\infty}\widetilde{\Lambda}_{n}^S,\quad \widetilde{\Lambda}_{n}^S:=\begin{cases}
    \widetilde{\Lambda}^S &:\;1\leq n \leq n_\alpha:=\min\left(\lfloor{\frac{1}{\alpha}\rfloor},N\right),\\
    \widetilde{\Lambda}^S\Big(1+\max\limits_{n_\alpha +1\leq j \leq n}\frac{P_\alpha^{j,j}}{P_\alpha^{j,j-1}}\Big)&:\;n_\alpha + 1\leq n \leq N,
\end{cases} \\
 & \widetilde{\Lambda}^S :=\Big(  \frac{1+\epsilon}{2}\|\widetilde{b}\|_{L^\infty(\Omega\times J)}  + 2\|\widetilde{c}\|_{L^\infty(\Omega\times J)} 
 + \epsilon|\lambda |\Big) +  2\epsilon^{-1}|\lambda |C_{\mathcal{I}}^2\left((1+\mu_{n_\alpha+1})^2+\mu_{n_\alpha+1}^2\right),\;\text{ and}\\
 &\text{the functions $\widetilde{b}$ and $ \widetilde{c}$ are defined in Theorem~\ref{stabilitytheorem1}.}
\end{align*}
\end{lemma}
\begin{proof}
Substitute $\boldsymbol{w}_h = \boldsymbol{\xi}^n$ and $v_h = \theta^n$ in (\ref{theta1}) and (\ref{theta2}), respectively, and then apply (\ref{theta1}) in (\ref{theta2}) to obtain
\begin{align*}
(D_{t_n}^{\alpha}\theta^n,\theta^n) + (\boldsymbol{B}^n {\boldsymbol{\xi}}^n, \boldsymbol{\xi}^n) =& - (\boldsymbol{B}^n {\boldsymbol{\zeta}}^n, \boldsymbol{\xi}^n) - ( \boldsymbol{b}^n \cdot \boldsymbol{B}^n  \boldsymbol{\xi}^n ,\theta^n)- (c^n  \theta^n ,\theta^n) + \lambda (\mathcal{I}E\theta^n, \theta^n) \\
&+ (\Upsilon^n + r^n + r_f^n -D_{t_n}^{\alpha}\eta^n - \boldsymbol{b}^n \cdot \boldsymbol{B}^n\boldsymbol{\zeta}^n - c^n \eta^n +\lambda \mathcal{I}E\eta^n  , \theta^n). 
\end{align*}
Now, an appeal to Lemma~\ref{dfdoi} and the Cauchy–Schwarz inequality along with the relations (\ref{intbound}) and (\ref{equivalentnorms}) yields 
\begin{align}
\frac{1}{2}D_{t_n}^{\alpha}\|\theta^n\|^2 + \vertiii{\boldsymbol{\xi}^n}_n^2 \leq &\sqrt{\gamma_0}\|\boldsymbol{\zeta}^n\|\;\vertiii{\boldsymbol{\xi}^n}_n + \vertiii{\boldsymbol{b}^n}_n\vertiii{\boldsymbol{\xi}^n}_n\|\theta^n\| +\|\widetilde{c}\|_{L^\infty(\Omega\times J)} \|\theta^n\|^2+C_{\mathcal{I}}|\lambda |\|E\theta^n\|\;\|\theta^n\|  \nonumber \\
&+ \|\Upsilon^n + r^n + r_f^n -D_{t_n}^{\alpha}\eta^n - \boldsymbol{b}^n \cdot \boldsymbol{B}^n\boldsymbol{\zeta}^n - c^n \eta^n +\lambda \mathcal{I}E\eta^n \|\|  \theta^n\|.\label{L2errorestimateeqn1_ast}
\end{align}
Using Young's inequality, for any $\epsilon>0$, we obtain
\begin{align}
&\sqrt{\gamma_0}\|\zeta^n\|\vertiii{\boldsymbol{\xi}^n}_n \leq  \frac{\epsilon}{1+\epsilon}\vertiii{\boldsymbol{\xi}^n}_n^2+ \gamma_0\frac{1+\epsilon}{4\epsilon}\|\zeta^n\|^2,\label{L2errorestimateeqn1_ast1}\\
&\vertiii{\boldsymbol{b}^n}_n\vertiii{\boldsymbol{\xi}^n}_n\|\theta\|\leq  \frac{1}{1+\epsilon}\vertiii{\boldsymbol{\xi}^n}_n^2+ \frac{1+\epsilon}{4}\vertiii{\boldsymbol{b}^n}_n^2\|\theta^n\|^2,\quad\text{and}\label{L2errorestimateeqn1_ast2}\\
& C_{\mathcal{I}}|\lambda |\|E\theta^n\|\|\theta^n\|\leq \epsilon\frac{|\lambda|}{2}\|\theta^n\|^2 + C_{\mathcal{I}}^2\frac{|\lambda|}{2\epsilon}\|E\theta^n\|^2.\label{L2errorestimateeqn1_ast3}
\end{align}
Thus, by applying the estimates (\ref{L2errorestimateeqn1_ast1}), (\ref{L2errorestimateeqn1_ast2}), and (\ref{L2errorestimateeqn1_ast3}) in (\ref{L2errorestimateeqn1_ast}), we obtain 
\begin{align}
D_{t_n}^{\alpha}\|\theta^n\|^2   \leq & \left( \frac{1+\epsilon}{2}\vertiii{\boldsymbol{b}^n}_n^2 + 2\|\widetilde{c}\|_{L^\infty(\Omega\times J)} 
 + \epsilon|\lambda | \right)\|\theta^n\|^2  +\epsilon^{-1}|\lambda |C_{\mathcal{I}}^2\|E\theta^n\|^2 \nonumber\\
 &+   2\|\Upsilon^n + r^n + r_f^n -D_{t_n}^{\alpha}\eta^n - \boldsymbol{b}^n \cdot \boldsymbol{B}^n\boldsymbol{\zeta}^n - c^n \eta^n +\lambda \mathcal{I}E\eta^n \|\;\|\theta^n\|+\gamma_0\frac{1+\epsilon}{2\epsilon}\|\zeta^n\|^2 \nonumber\\ 
 \leq & \sum_{i=0}^{n}\widetilde{\lambda}_{n-i}^{n}\|\theta^{i}\|^2 +   2\Big(\|\Upsilon^n\| + \|r^n\| + \|r_f^n\| + \|D_{t_n}^{\alpha}\eta^n\| + \big(\gamma_0  \|\widetilde{b}\|_{L^\infty(\Omega\times J)}\big)^{1/2}\| \boldsymbol{\zeta}^n\| \nonumber\\
 &\hspace{3cm}   +  C_c\|\eta^n\| +|\lambda|C_{\mathcal{I}} \|E\eta^n \|\Big)\;\|\theta^n\| +\gamma_0\frac{1+\epsilon}{2\epsilon}\|\zeta^n\|^2 , \label{L2errorestimateeqn1_ast4}
\end{align}
where 
\begin{align*}
& \widetilde{\lambda}^n_{n-i} := \begin{cases}
 \displaystyle\frac{1+\epsilon}{2}\|\widetilde{b}\|_{L^\infty(\Omega\times J)}  + 2\|\widetilde{c}\|_{L^\infty(\Omega\times J)} 
 + \epsilon|\lambda | &: \;  1\leq n \leq N\;\&\; i=n, \\[4pt]
\epsilon^{-1}|\lambda |C_{\mathcal{I}}^2 &: \; 1\leq n \leq n_\alpha\;\&\; i=n-1, \\ 
0&: \; 1\leq n \leq n_\alpha\;\&\; i=n-2, \\ 
 2\epsilon^{-1}|\lambda |C_{\mathcal{I}}^2(1+\mu_{n_\alpha+1})^2 &: \;  n_{\alpha}+1\leq n \leq N\;\&\; i=n-1, \\
 2\epsilon^{-1}|\lambda |C_{\mathcal{I}}^2\mu_{n_\alpha+1}^2 &: \;  n_{\alpha}+1\leq n \leq N\;\&\; i=n-2, \\
0 &: \;  3\leq n \leq N\;\&\; 0 \leq i \leq n-3, \\
    \end{cases}
    \end{align*}
    and
    \begin{align*}
 \displaystyle{\sum_{i=0}^n \widetilde{\lambda}_{n-i}^n} &= \Big(  \frac{1+\epsilon}{2}\|\widetilde{b}\|_{L^\infty(\Omega\times J)}  + 2\|\widetilde{c}\|_{L^\infty(\Omega\times J)} 
 + \epsilon|\lambda |\Big) +  2\epsilon^{-1}|\lambda |C_{\mathcal{I}}^2\left((1+\mu_{n_\alpha+1})^2+\mu_{n_\alpha+1}^2\right)\\
 &=\; \max_{1 \leq n \leq N} \sum_{i=0}^n \widetilde{\lambda}_{n-i}^n =:\widetilde{\Lambda}^S.
\end{align*}
Finally, an appeal to the discrete fractional Gr\"{o}nwall inequality (Theorem~\ref{DFGI}) in (\ref{L2errorestimateeqn1_ast4}) and non-decreasing property of $E_\alpha(\cdot)$ yields the required estimate (\ref{L2errorestimateeqn1}).
\end{proof}

We now revisit the following essential result required to achieve the optimal error estimate for the flux.
\begin{lemma}(Lemma~4.4, \cite{TOMAR2024137})\label{DiscreteDerivativeDifference}
Let $T^j:= D_{t_j}^{\alpha} (t_j e_u^j) - t_j D_{t_j}^{\alpha} e_u^j$. Then
\begin{align}
 \nonumber \Big(\sum_{j=1}^n P^{n,j}_{\alpha}\|T^j\|^2 \Big)^{1/2} &\leq C\; t_n^{1-\frac{\alpha}{2}}\max_{0\leq j \leq n}\|e_u^j\|,\quad 1\leq n\leq N,
 \end{align}
where the positive constant $C$ remains bounded as $\alpha\to 1^{-}$.
\end{lemma}
\begin{lemma}\label{xiestimate_l2norm}
If for some $\epsilon>0$ and $\delta>1$, $\max\limits_{1\leq n \leq N}\Delta t_n\leq \left( \delta\;\lambda^F_\epsilon\Gamma(2-\alpha)\right)^{-1/\alpha}$, then $\boldsymbol{\xi}^{n}$ satisfies
\begin{align}\label{L2xierrorestimateeqn1}
 \|t_n{\boldsymbol{\xi}}^n\|  \leq &\frac{C_\delta}{\sqrt{\beta_0}} E_{\alpha}(C_\delta \Lambda_\infty^F t_n^\alpha )\max_{1\leq j \leq n}\Bigg( \frac{2}{\sqrt{\beta_0}}\sum_{i=1}^j P_{\alpha}^{j,i}\|D_{t_i}^{\alpha}( \boldsymbol{B}^i t_i{\boldsymbol{\zeta}}^i)\| + \Big(\frac{1+\epsilon}{2\epsilon}\Big)^{1/2}\Big(\sum_{i=1}^j P_{\alpha}^{j,i}\Big(\|T^i\|+ \|t_i\Upsilon^i \| \nonumber \\
 &+\|t_i r^i \|+\|t_i r_f^i \| +\|D_{t_i}^{\alpha}(t_i\eta^i)\|+ +C_c\|t_i\theta^i\|  + C_c\|t_i\eta^i\| + C_{\mathcal{I}}|\lambda|\|E(t_i\theta^i)\|+ C_{\mathcal{I}}|\lambda|\|E(t_i\eta^i)\|  \nonumber\\
 &+\big( \gamma_0\|\widetilde{b}\|_{L^\infty(\Omega\times J)}\big)^{1/2} \|t_i\zeta^i\|\Big)^2 \Big)^{\frac{1}{2}}\;\; \Bigg),\quad 1\leq n\leq N,
 \end{align}
where $\lambda_\epsilon^F$ and $\Lambda_{\infty}^F$ are defined in Theorem~\ref{stabilitytheorem2}, the functions $\widetilde{b}$ and $ \widetilde{c}$ are defined in Theorem~\ref{stabilitytheorem1}, and $T^n:= D_{t_n}^{\alpha} (t_n e_u^n) - t_n D_{t_n}^{\alpha} e_u^n$. 
\end{lemma}
\begin{proof}
Multiply $t_n$ in (\ref{theta1}) and then operate the discrete fractional derivative to obtain
\begin{align}\label{newtheta2}
    (D_{t_n}^{\alpha} \boldsymbol{B}^n (t_n{\boldsymbol{\xi}}^n+t_n{\boldsymbol{\zeta}}^n), \boldsymbol{w}_h) + (D_{t_n}^{\alpha} (t_n\theta^n),  \nabla \cdot \boldsymbol{w}_h)~ &=~0 \quad \forall \boldsymbol{w}_h \in \boldsymbol{W}_h,\;1\leq n\leq N.
\end{align}
Further, after choosing $\boldsymbol{w}_h = t_n\boldsymbol{\xi}^n$ and $v_h = D_{t_n}^{\alpha} (t_n\theta^n
)$ in (\ref{newtheta2})  and (\ref{theta2}), respectively, substitute (\ref{newtheta2}) in (\ref{theta2}) to get the following relation 
\begin{align*}
&(D_{t_n}^{\alpha}(t_n\theta^n
),D_{t_n}^{\alpha} (t_n\theta^n
) ) + (D_{t_n}^{\alpha} \boldsymbol{B}^n t_n{\boldsymbol{\xi}}^n,  t_n\boldsymbol{\xi}^n ) = - (\boldsymbol{b}^n \cdot \boldsymbol{B}^n  t_n \boldsymbol{\xi}^n, D_{t_n}^{\alpha} (t_n\theta^n
) ) - (D_{t_n}^{\alpha} \boldsymbol{B}^n t_n{\boldsymbol{\zeta}}^n, t_n{\boldsymbol{\xi}}^n)\\
&+ \Big( T^n + t_n\Upsilon^n + t_n r^n + t_n r_f^n - D_{t_n}^{\alpha}(t_n\eta^n)     - c^n (  t_n\theta^n + t_n\eta^n) + \lambda \mathcal{I}(E (t_n\theta^n +t_n\eta^n))  - \boldsymbol{b}^n \cdot \boldsymbol{B}^n  t_n \boldsymbol{\zeta}^n , D_{t_n}^{\alpha} (t_n\theta^n
)\Big)  ,
\end{align*}
where $T^n:= D_{t_n}^{\alpha} (t_n e_u^n) - t_n D_{t_n}^{\alpha} e_u^n$. Now, an appeal to the estimate (\ref{pollutionterm}) of Lemma~\ref{dfdoi}, the Cauchy-Schwarz inequality, and the estimates (\ref{intbound}), (\ref{coefbounds}) and (\ref{equivalentnorms}) yields
\begin{align}
\|D_{t_n}^{\alpha}(t_n\theta^n
)\|^2 &+ \frac{1}{2}D_{t_n}^{\alpha}\vertiii{ t_n{\boldsymbol{\xi}}^n}^2_n \; \leq\;  \frac{L_B t_n^{1-\alpha}}{2\beta_0\Gamma(2-\alpha)}\vertiii{t_n\boldsymbol{\xi}^n}^2_{n} + \big(\|\widetilde{b}\|_{L^\infty(\Omega\times J)}\big)^{1/2}\vertiii{ t_n{\boldsymbol{\xi}}^n}_n\|D_{t_n}^{\alpha} (t_n\theta^n
)\|  \nonumber\\
&+\frac{1}{\sqrt{\beta_0}}\| D_{t_n}^{\alpha} (\boldsymbol{B}^n t_n{\boldsymbol{\zeta}}^n)\|\vertiii{ t_n{\boldsymbol{\xi}}^n}_n +\Big(\|T^n\|+\|t_n\Upsilon^n\| + \|t_n r^n\| + \|t_n r_f^n \|+ \| D_{t_n}^{\alpha}(t_n\eta^n)\|+C_c\|t_n\theta^n\|  \nonumber\\
& + C_c\|t_n\eta^n\| + C_{\mathcal{I}}|\lambda|\|E(t_n\theta^n)\|+ C_{\mathcal{I}}|\lambda|\|E(t_n\eta^n)\| + \big( \gamma_0\|\widetilde{b}\|_{L^\infty(\Omega\times J)}\big)^{1/2} \|t_n\zeta^n\|\Big)\|D_{t_n}^{\alpha} (t_n\theta^n
)\|. \label{newtheta3}
\end{align}
Applying Young's inequality in (\ref{newtheta3}) with appropriate $\epsilon>0$, we obtain the following estimate
\begin{align}
D_{t_n}^{\alpha}\vertiii{ t_n{\boldsymbol{\xi}}^n}^2_n \; \leq\;&  \Big(\frac{1+\epsilon}{2}\|\widetilde{b}\|_{L^\infty(\Omega\times J)} + \frac{L_B t_n^{1-\alpha}}{\beta_0\Gamma(2-\alpha)}  \Big)
 \vertiii{t_n\boldsymbol{\xi}^n}^2_{n} + \frac{2}{\sqrt{\beta_0}}\| D_{t_n}^{\alpha} (\boldsymbol{B}^n t_n{\boldsymbol{\zeta}}^n)\|\vertiii{ t_n{\boldsymbol{\xi}}^n}_n  \nonumber\\
&+\frac{1+\epsilon}{2\epsilon}\Big(\|T^n\|+\|t_n\Upsilon^n\| + \|t_n r^n\| + \|t_n r_f^n \|+ \| D_{t_n}^{\alpha}(t_n\eta^n)\|+C_c\|t_n\theta^n\|  \nonumber\\
& + C_c\|t_n\eta^n\| + C_{\mathcal{I}}|\lambda|\|E(t_n\theta^n)\|+ C_{\mathcal{I}}|\lambda|\|E(t_n\eta^n)\| + \big( \gamma_0\|\widetilde{b}\|_{L^\infty(\Omega\times J)}\big)^{1/2} \|t_n\zeta^n\|\Big)^2. \label{newtheta4}
\end{align}
Under the condition $\max\limits_{1\leq n \leq N}\Delta t_n\leq \left( \delta\;\lambda^F_\epsilon\Gamma(2-\alpha)\right)^{-1/\alpha}$,  the discrete fractional Gr\"{o}nwall inequality (Theorem \ref{DFGI}) can be applied to (\ref{newtheta4}) to get the following estimate
\begin{align}
 \vertiii{ t_n{\boldsymbol{\xi}}^n}_n  \leq & C_\delta E_{\alpha}(C_\delta \Lambda_n^F t_n^\alpha )\max_{1\leq j \leq n}\Bigg( \frac{2}{\sqrt{\beta_0}}\sum_{i=1}^j P_{\alpha}^{j,i}\|D_{t_i}^{\alpha}( \boldsymbol{B}^i t_i{\boldsymbol{\zeta}}^i)\| + \Big(\frac{1+\epsilon}{2\epsilon}\Big)^{1/2}\Big(\sum_{i=1}^j P_{\alpha}^{j,i}\Big(\|T^i\|+ \|t_i\Upsilon^i \| \nonumber \\
 &+\|t_i r^i \|+\|t_i r_f^i \| +\|D_{t_i}^{\alpha}(t_i\eta^i)\|+ +C_c\|t_i\theta^i\|  + C_c\|t_i\eta^i\| + C_{\mathcal{I}}|\lambda|\|E(t_i\theta^i)\|+ C_{\mathcal{I}}|\lambda|\|E(t_i\eta^i)\|  \nonumber\\
 &+\big( \gamma_0\|\widetilde{b}\|_{L^\infty(\Omega\times J)}\big)^{1/2} \|t_i\zeta^i\|\Big)^2 \Big)^{\frac{1}{2}}\;\; \Bigg),\quad 1\leq n\leq N.\nonumber
 \end{align}
Finally, by applying  (\ref{equivalentnorms}), we obtain the desired estimate (\ref{L2xierrorestimateeqn1}).
\end{proof}

Now, we state the following truncation error estimate.
\begin{lemma}\label{truncationerror} 
If the grading parameter $\gamma$ satisfies $1\;\leq\;\gamma\;\leq\;\frac{2(2-\alpha)}{\alpha} $ then under the assumptions in Theorem \ref{Regularity_condition}, there hold
\begin{enumerate}\label{truncationI}
    \item[\namedlabel{trunk:i}{(i)}] $\sum_{j=1}^nP_{\alpha}^{n,j}\|r^j\| + t_n^{\alpha/2}\left(\sum_{j=1}^nP_{\alpha}^{n,j}\|r^j\|^2 \right)^{{1}/{2}}  \leq C \;t_n^{\alpha}\;N^{-\min\{\gamma\alpha,\; 2 \} }, \quad 1\leq n \leq N,$
    \item[\namedlabel{trunk:ii}{(ii)}] $\sum_{j=1}^nP_{\alpha}^{n,j}\|r_f^j\| + t_n^{\alpha/2\left(\sum_{j=1}^nP_{\alpha}^{n,j}\|r^j_f\|^2\right)^{{1}/{2}}}  \leq C \;t_n^{\alpha}\;N^{-\min\{\gamma\alpha, \;2 \} }, \quad 1\leq n \leq N,$
    \item[\namedlabel{trunk:iii}{(iii)}] $\sum_{j=1}^nP_{\alpha}^{n,j}\|\Upsilon^j\|  \leq C \log_e(n+2)\;N^{-\min\{\gamma\alpha,\; 2-\alpha \} }, \quad 1\leq n \leq N,$ and
    \item[\namedlabel{trunk:iv}{(iv)}] $\left(\sum_{j=1}^n P_{\alpha}^{n,j}\|t_j\Upsilon^j\|^2\right)^{{1}/{2}}  \leq C\;t_n^{1-\alpha/2} N^{-\min\{\gamma\alpha, \;2-\alpha \} }, \quad 1\leq n \leq N,$
\end{enumerate}
where the positive constant $C$ remains bounded as $\alpha\to 1^{-}$. Moreover, if $\frac{(2-\alpha)}{\alpha}\;<\;\gamma\;\leq\;\frac{2(2-\alpha)}{\alpha} $ then \ref{trunk:iii} holds without the $\log_e(n+2)$ factor.
\end{lemma}
\begin{proof}
Apply the regularity result Theorem~\ref{Regularity_condition} in 
\begin{align*}
    &u(t_n) - Eu(t_n) = \begin{cases}
        \displaystyle \int_0^{t_1} \partial_t u(t)\;dt & {:} \;n=1,\\[4pt]
        \displaystyle \mu_n\int_{t_{n-2}}^{t_n} (t-t_{n-2})\partial_t^2 u(t)\;dt - (1+\mu_n)\int_{t_{n-1}}^{t_n} (t-t_{n-1})\partial_t^2 u(t)\;dt &{:}\;n\geq 2,
    \end{cases}
\end{align*}
to obtain 
\begin{align}
    \label{extrapolationerr1} \|u(t_1) - Eu(t_1)\|\;& \leq\; C\int_0^{t_1}\;t^{\alpha -1}\;dt\;\leq\; \widetilde{C}_1 \;N^{-\gamma\alpha},\quad \widetilde{C}_1:=CT^{\alpha}/\alpha,\\
\label{extrapolationerr2} \|u(t_2) - Eu(t_2)\|\;& \leq\; C \left(\mu_2\int_0^{t_2} t^{\alpha-1}\;dt + (1+\mu_2)\int_{t_1}^{t_2}(t-t_1)\;t^{\alpha-2}\;dt\right),\\
\label{extrapolationerr3} \|u(t_n) - Eu(t_n)\|\;& \leq\; C \left(\mu_n\int_{t_{n-2}}^{t_n} (t-t_{n-2})t^{\alpha-2}\;dt + (1+\mu_n)\int_{t_{n-1}}^{t_n}(t-t_{n-1})\;t^{\alpha-2}\;dt\right),\;3\leq n \leq N,
\end{align}
and then use $\mu_2 = 2^{\gamma}-1$ and $\Delta t_2\leq \gamma 2^{\gamma-1}T\;N^{-\gamma}$ in the estimate (\ref{extrapolationerr2}) to get the following estimate
\begin{align}
\label{extrapolationerr4} \|u(t_2) - Eu(t_2)\|\; & \leq\; 2^\gamma\;C \left(\frac{t_2^{\alpha}}{\alpha} + t_1^{\alpha-2}\Delta t_2^2\right)\;\leq \;\widetilde{C}_2\;N^{-\gamma\alpha},\quad \widetilde{C}_2:=C\;2^{\gamma}\;T^{\alpha}\max(2^{\gamma\alpha}/\alpha,\;\gamma^2\;2^{2\gamma-2}).
\end{align}
Now, an appeal to $\max_{3\leq n \leq N}\mu_n\leq 3^{\gamma-1}$, $\Delta t_{n-1}\leq \Delta t_n \leq \gamma T n^{\gamma-1} N^{-\gamma}$, and $\max(t_n/t_{n-2},\;t_{n-1}/t_{n-2})\leq 3^\gamma$, $3\leq n \leq N,$ in the estimate (\ref{extrapolationerr3}) yields 
\begin{align}
    \nonumber \|u(t_n) - Eu(t_n)\|\; & \leq \; C 3^{\gamma} \left((\Delta t_{n-1}+\Delta t_n)^2 t_{n-2}^{\alpha-2} + \Delta t_n^2 t_{n-1}^{\alpha-2} \right)\\
    \nonumber &\leq \; 5C 3^{2\gamma} \;\Delta t_n^2 t_{n}^{\alpha-2}\\
    \nonumber &\leq \; 5C 3^{2\gamma} \gamma^2 T^2 \;n^{\gamma\alpha-2}N^{-\gamma\alpha}\\
    \label{extrapolationerr5} &\leq\; \widetilde{C}_3\;N^{-\min\{\gamma\alpha,\;2\}},\quad \widetilde{C}_3:= 5C 3^{2\gamma} \gamma^2 T^2,\;3\leq n \leq N.
\end{align}
Thus, by combining the above estimates (\ref{extrapolationerr1}), (\ref{extrapolationerr4}), (\ref{extrapolationerr5}), and the estimate (\ref{intbound}) with $j=0$, we obtain 
\begin{align}
    \nonumber& \| r^n \|\;\leq\; \lambda C \;\|u(t_n) - Eu(t_n)\|\;\leq\;\lambda C \max\{\widetilde{C}_1, \widetilde{C}_2, \widetilde{C}_3 \}\; N^{-\min\{\gamma\alpha, \;2\}},
\end{align}
and hence the estimate \ref{trunk:i} follows after applying the Lemma~\ref{discretekernelproperties}\ref{itm:c} with $j=1.$ Estimate \ref{trunk:ii} can be established by a similar procedure. To establish \ref{trunk:iii}, recall the following estimate from Stynes et al.  \cite{MR3639581} (see, Lemma~5.2)
\begin{align}\label{trunc1}
    \|\Upsilon^n\| \;&\leq\; C\; n^{-\min\{\gamma \alpha,\; 2-\alpha\}} \;=\; CT^{\min\{\alpha,\;\frac{2-\alpha}{\gamma} \}}\; t_n^{-\min\{\alpha,\;\frac{2-\alpha}{\gamma} \}}\;  N^{-\min\{\gamma\alpha ,\;2-\alpha\}},
    \end{align}
 where the positive constant $C$ remains bounded as $\alpha \rightarrow 1^-$, and then apply $t_n^{-\min\{\alpha,\;\frac{2-\alpha}{\gamma} \}}\;\leq\;\max\{1,\;T^\alpha\}t_n^{-\alpha}$ and Lemma~\ref{discretekernelproperties}\ref{itm:g}. Estimate \ref{trunk:iv} follows after applying Lemma~\ref{discretekernelproperties}\ref{itm:c} with $j=1$ to the following estimate obtained from (\ref{trunc1})
 \begin{align*}
      \|t_n\Upsilon^n\|^2 \;&\leq\; C^2 T^{2\min\{\alpha,\;\frac{2-\alpha}{\gamma} \}}\; t_n^{2-2\min\{\alpha,\;\frac{2-\alpha}{\gamma} \}}\;  N^{-2\min\{\gamma\alpha ,\;2-\alpha\}}\;\leq\; C^2 T^{2\min\{\alpha,\;\frac{2-\alpha}{\gamma} \}}\; t_n^{2-2\alpha}\;  N^{-2\min\{\gamma\alpha ,\;2-\alpha\}}. 
 \end{align*}
Now, if $\frac{2-\alpha}{\alpha} < \gamma \leq \frac{2(2-\alpha)}{\alpha}$, the estimate (\ref{trunc1}) can be written as
\begin{align*}
    \|\Upsilon^n\| \;&\leq\;  CT^{\frac{2-\alpha}{\gamma}}\; t_n^{(\alpha - \frac{2-\alpha}{\gamma})-\alpha}\;  N^{-\min\{\gamma\alpha ,\;2-\alpha\}},\quad \left(\alpha - \frac{2-\alpha}{\gamma}\right) \in (0,\;1),
\end{align*}
 and thus an appeal to the Lemma~\ref{discretekernelproperties}\ref{itm:f} yields
 \begin{align}
     \nonumber\sum_{j=1}^nP_{\alpha}^{n,j}\|\Upsilon^j\| \;& \leq\; CT^{\frac{2-\alpha}{\gamma}}\; \frac{\Gamma(1-\frac{2-\alpha}{\gamma})}{\Gamma(1+\alpha - \frac{2-\alpha}{\gamma})}\;t_n^{(\alpha - \frac{2-\alpha}{\gamma})}\;  N^{-\min\{\gamma\alpha ,\;2-\alpha\}}\\
     \label{extrapolationerr6}&\leq\;2CT^{\frac{2-\alpha}{\gamma}}\; \Gamma(1-\frac{2-\alpha}{\gamma})\;t_n^{(\alpha - \frac{2-\alpha}{\gamma})}\;  N^{-\min\{\gamma\alpha ,\;2-\alpha\}}.
\end{align}
Finally, as $\lim\limits_{\alpha\to 1^{-}} (2-\alpha)/\gamma\;<\;1$ for all $\gamma:\;(2-\alpha)/\alpha \;<\;\gamma\;\leq\;2(2-\alpha)/\alpha$, above estimate (\ref{extrapolationerr6}) establishes \ref{trunk:iii} without any logarithmic factor $\log_e(n+2)$. 

\end{proof}
 
\begin{remark}\label{remarkontruncationerror}
For any fixed $\alpha \in (0, 1)$, the estimates (\ref{trunc1}), $t_n^{-\min\{\alpha,\;\frac{2-\alpha}{\gamma} \}}\;\leq\;\max\{1,\;T^\alpha\}t_n^{-\alpha}$, and Lemma~\ref{discretekernelproperties}\ref{itm:c} with $j=0$ guarantee that the estimate \ref{trunk:iii} holds without the presence of any logarithmic factor $\log_e(n+2)$ for all $\gamma:\;1\leq \gamma\leq 2(2-\alpha/\alpha)$.
\end{remark}

\begin{lemma}\label{etaestimate}
Under the assumptions in Theorem~\ref{Regularity_condition}, the following estimate holds
\begin{enumerate}
\item[\namedlabel{etaestimate:i}{(i)}] $\sum_{j=1}^nP_{\alpha}^{n,j}\|D^{\alpha}_{t_j} \eta^j\| \leq C\;\left(h^{2}\log_e(n+2) + \sum_{j=1}^{n}P^{n,j}_{\alpha}\|\Upsilon^j\| \right), \quad 1\leq n \leq N$,
\item[\namedlabel{etaestimate:ii}{(ii)}] $\sum_{j=1}^nP_{\alpha}^{n,j}\| \boldsymbol{\zeta}^j\|  \leq C\;h^2 \;t_n^{{\alpha}/{2}},  \quad \text{and}\quad  \left(\sum_{j=1}^nP_{\alpha}^{n,j}\| \boldsymbol{\zeta}^j\|^2 \right)^{{1}/{2}} \leq C\;h^2\log_e(n+2), \quad 1\leq n \leq N,$
\item[\namedlabel{etaestimate:iii}{(iii)}] $ \left(\sum_{j=1}^nP_{\alpha}^{n,j}\|D^{\alpha}_{t_j} t_j\eta^j\|^2 \right)^{{1}/{2}}  \leq C \;h^2 \;t_n^{1-\frac{\alpha}{2}}, \quad 1\leq n \leq N,$ and
\item[\namedlabel{etaestimate:iv}{(iv)}] $\sum_{j=1}^nP_{\alpha}^{n,j}\|D_{t_j}^{\alpha} (\boldsymbol{B}^j t_j{\boldsymbol{\zeta}}^j)\| \leq C\;h^2\;t_n^{1-\frac{\alpha}{2}} , \quad 1\leq n \leq N,$
\end{enumerate}
where the positive constant $C$ remains bounded as $\alpha\to 1^{-}$. Moreover, if $u_0\in\dot{H}^{2+\epsilon}(\Omega)$ for some $\epsilon >0$, then \ref{etaestimate:i} and \ref{etaestimate:ii} hold without any $\log_e(n+2)$ factor.
\end{lemma}
\begin{proof}
An appeal to the triangle inequality, $L^2$-norm stability of $L^2$-projection, and the approximation property (\ref{approxprop}) yields
\begin{align}
    \nonumber \sum_{j=1}^{n}P^{n,j}_{\alpha}\|D_{t_j}^{\alpha} \eta^j\| & \leq \sum_{j=1}^{n}P^{n,j}_{\alpha}\left(\|\Upsilon^j\| + \|\partial_t^{\alpha} u(t_j) - P_h \partial_t^{\alpha} u(t_j)\| + \|P_h\Upsilon^j\|\right) \\
    \label{alpharobitaest} &\leq C  \sum_{j=1}^{n}P^{n,j}_{\alpha}\|\Upsilon^j\|  +  Ch^2\sum_{j=1}^{n}P^{n,j}_{\alpha}\|\partial_t^{\alpha} u(t_j) \|_2.
\end{align}
Apply the regularity results of Theorem~\ref{Regularity_condition} and Lemma~\ref{discretekernelproperties}\ref{itm:g} in (\ref{alpharobitaest}) to get the estimate \ref{etaestimate:i}. Now, apply the approximation property (\ref{approxprop}) to obtain the following two estimates 

\begin{align}
\label{alpharobzitaest}
\sum_{j=1}^nP_{\alpha}^{n,j}\| \boldsymbol{\zeta}^j\| &\leq C h^2 \sum_{j=1}^nP_{\alpha}^{n,j}\|\nabla\cdot\sigma(t_j)\|_1,\\
\label{alpharobzitaest_ast}\sum_{j=1}^n P_{\alpha}^{n,j}\| \boldsymbol{\zeta}^j\|^2 &\leq C h^4 \sum_{j=1}^n P_{\alpha}^{n,j}\|\nabla\cdot\sigma(t_j)\|_1^2.
\end{align}
Both estimates in \ref{etaestimate:ii} follow by applying the regularity result from Theorem~\ref{Regularity_condition}, along with Lemma~\ref{discretekernelproperties}\ref{itm:f} and Lemma~\ref{discretekernelproperties}\ref{itm:g} to (\ref{alpharobzitaest}) and (\ref{alpharobzitaest_ast}). We obtain the result \ref{etaestimate:iii}, by applying the approximation property (\ref{approxprop}) and regularity result Theorem~\ref{Regularity_condition} as follows
\begin{align*}
    \|D_{t_n}^{\alpha} (t_n \eta^n)\|& \leq \sum_{j=1}^n K^{n,j}_{1-\alpha}\|t_j\eta^j - t_{j-1}\eta^{j-1}\|\leq  \sum_{j=1}^n K^{n,j}_{1-\alpha}\int_{t_{j-1}}^{t_j}\|\partial_s(s\eta)\|ds\\
     &\leq C\;h^2 \sum_{j=1}^n K^{n,j}_{1-\alpha}\Delta t_j \leq   C\;h^2 t_n^{1-\alpha}
\end{align*}
and then by using Lemma~\ref{discretekernelproperties}\ref{itm:c} with $j=1$. Now, an application of Lemma~\ref{discretekernelproperties}\ref{itm:b} yields
\begin{align*} \sum_{j=1}^{n}P^{n,j}_{\alpha}\|D_{t_j}^{\alpha} (\boldsymbol{B}^jt_j\boldsymbol{\zeta}^j)\| & \leq \sum_{j=1}^{n}P^{n,j}_{\alpha}\sum_{i=1}^{j}K^{j,i}_{1-\alpha}\|\boldsymbol{B}^i t_i\boldsymbol{\zeta}^i - \boldsymbol{B}^{i-1}t_{i-1}\boldsymbol{\zeta}^{i-1}\|  \\ &=\sum_{i=1}^{n}\left(\sum_{j=i}^{n}P^{n,j}_{\alpha}K^{j,i}_{1-\alpha}\right)\|\boldsymbol{B}^i t_i\boldsymbol{\zeta}^i - \boldsymbol{B}^{i-1}t_{i-1}\boldsymbol{\zeta}^{i-1}\| \\
 &\leq \sum_{i=1}^{n}(\|(\boldsymbol{B}^i - \boldsymbol{B}^{i-1})t_i\boldsymbol{\zeta}^{i}\| + \|\boldsymbol{B}^{i-1}(t_i\boldsymbol{\zeta}^i - t_{i-1}\boldsymbol{\zeta}^{i-1})\| ). 
 \end{align*}
 Using the Lipschitz continuity (\ref{Blips}) of $\boldsymbol{B}$, equivalent norms (\ref{equivalentnorms}), and the regularity result Theorem~\ref{Regularity_condition}, we obtain
 \begin{align*}
    \sum_{j=1}^{n}P^{n,j}_{\alpha}\|D_{t_j}^{\alpha} (\boldsymbol{B}^j t_j\boldsymbol{\zeta}^j)\| \;& \leq \sum_{i=1}^{n} \left (L_B \Delta t_i \|t_i\boldsymbol{\zeta}^i\|  + \sqrt{\gamma_0}\|t_i\boldsymbol{\zeta}^i - t_{i-1}\boldsymbol{\zeta}^{i-1}\| \right)\\
    &\leq  \sum_{i=1}^{n} \left (L_B \Delta t_i \|t_i\boldsymbol{\zeta}^i\| \right)  + \sqrt{\gamma_0} \sum_{i=1}^{n}\int_{t_{i-1}}^{t_i}\|\partial_s (s\boldsymbol{\zeta})\|ds\\
   &\leq C\;\left(h^2\;t_n^{1-\frac{\alpha}{2}}+h^2\int_{0}^{t_n} s^{-\frac{\alpha}{2}}ds \right) ,
\end{align*}
and hence, the result \ref{etaestimate:iv} follows. Moreover, if $u_0\in\dot{H}^{2+\epsilon}(\Omega)$ for some $\epsilon >0$, then an application of the regularity result Theorem~\ref{Regularity_condition1} in estimates (\ref{alpharobitaest}) and (\ref{alpharobzitaest}) implies
\begin{align*}
    \sum_{j=1}^{n}P^{n,j}_{\alpha}\|D_{t_j}^{\alpha} \eta^j\| & \leq  C  \sum_{j=1}^{n}P^{n,j}_{\alpha}\|\Upsilon^j\|  +  Ch^2\sum_{j=1}^{n}P^{n,j}_{\alpha} t_j^{\epsilon\alpha/2 - \alpha},\;\quad \text{and}\\
    \sum_{j=1}^nP_{\alpha}^{n,j}\| \boldsymbol{\zeta}^j\|^2 &\leq C h^4 \sum_{j=1}^nP_{\alpha}^{n,j} t_j^{\epsilon\alpha - \alpha},
\end{align*}
respectively. Finally, an appeal to the Lemma~\ref{discretekernelproperties}\ref{itm:f} ensures the validity of the estimates \ref{etaestimate:i} and \ref{etaestimate:ii} without any logarithmic factor $\log_e(n+2)$. 
\end{proof}
\begin{remark}\label{remarkonetaandzitaestimates}
For any fixed $\alpha \in (0, 1)$, assuming the regularity conditions in Theorem~\ref{Regularity_condition}, the estimates (\ref{alpharobitaest}), (\ref{alpharobzitaest}), and Lemma~\ref{discretekernelproperties}\ref{itm:c} with $j=0$ guarantee that the estimates \ref{etaestimate:i} and \ref{etaestimate:ii} hold without the presence of any logarithmic factor $\log_e(n+2)$, even if $u_0 \notin \dot{H}^{2+\epsilon}$, $\epsilon >0$.
\end{remark}
\begin{lemma}\label{VTlemma}
For $d=1, 2,$ let $\boldsymbol{w}\in L^2(\Omega;\mathbb{R}^d)$ and $v_h\in V_h$ be such that
\begin{align}
   \nonumber&(\boldsymbol{w},\boldsymbol{w}_h)+(v_h,\nabla\cdot \boldsymbol{w}_h) = 0,\quad \forall \boldsymbol{w}_h\in\boldsymbol{W}_h.
\end{align}
Then there exists a constant $C>0$, such that,
\begin{align}
    \nonumber & \|v_h\|_{L^\infty(\Omega)} \leq C \ell_{h,d}\|\boldsymbol{w}\|,
\end{align}
where $\displaystyle \ell_{h,d}:=\begin{cases}
    1 & {:}\; d=1,\\
    1+|\log{h}| & {:}\; d=2.
\end{cases}$
\end{lemma}
\begin{proof}
    For $d=2$, see Lemma~1.2 in \cite{MR0610597}. In the case of $d=1$, the constant $C_q$ appeared in the proof of the Lemma~1.2 in \cite{MR0610597} is independent of $q$, and hence the result follows.
\end{proof}
Finally, optimal error estimates are established in the following Theorem.
\begin{theorem}\label{L2H1errortheorem}
Let the pair $(u_h^n, \boldsymbol{\sigma}_h^n)$, satisfying (\ref{fullydiscrete1})-(\ref{fullydiscrete2}), be the approximation of the solution pair $(u(t_{n}),\boldsymbol{\sigma}(t_n))$, which satisfies (\ref{variation1})-(\ref{variation2}) at the temporal grid $t_{n}$. If the time partition $\{t_n\}_{n=0}^{N}$ of $[0,T]$ is given by (\ref{gradedmesh1}), such that for some $\epsilon>0$ and $\delta>1$, $\max\limits_{1\leq n \leq N}\Delta t_n \leq \min\Big\{\left( \delta\;\lambda^S_\epsilon\Gamma(2-\alpha)\right)^{-1/\alpha}, \left( \delta\;\lambda^F_\epsilon\Gamma(2-\alpha)\right)^{-1/\alpha}\Big\}$, then, under the assumptions of Theorem~\ref{Regularity_condition} and $1\;\leq\;\gamma\;\leq\;\frac{2(2-\alpha)}{\alpha} $, the following error estimate holds:
\begin{align}
\label{finalestimateforuI} \max_{1\leq n \leq N} \left(\|u_h^n - u(t_{n}) \| +t_n^{\alpha/2}\|\boldsymbol{\sigma}_h^n - \boldsymbol{\sigma}(t_{n}) \| \right)& \leq 
C\;\log_e(N)\; (h^{2} + N^{ -\min\{\gamma\alpha,2-\alpha\} }),\;N>2,
\end{align}
where $\lambda_\epsilon^S$ and $\lambda_\epsilon^F$ are defined in Lemma~\ref{thetaestimate_l2norm} and Theorem~\ref{stabilitytheorem2}, respectively. Moreover, if $u_0\in\dot{H}^{2+\epsilon}(\Omega)$ for some $\epsilon>0$, and $\frac{2-\alpha}{\alpha}\;<\;\gamma\;\leq\;\frac{2(2-\alpha)}{\alpha} $, then the estimate (\ref{finalestimateforuI}) is valid without the logarithmic factor $\log_e(N)$. The positive constant $C$ in (\ref{finalestimateforuI}) remains bounded as $\alpha\to 1^{-}$. 
\end{theorem}
\begin{proof}
Under the assumptions of Theorem~\ref{Regularity_condition}, an appeal to the estimate (\ref{approxprop}) yields
\begin{align}\label{zetaestimate}
    \|\eta^n\| + t_n^{\alpha/2}\|\boldsymbol{\zeta}^n\| \leq C h^2, \quad 1 \leq n \leq N.
\end{align}
Now, apply the above estimate (\ref{zetaestimate}), Lemma~\ref{truncationI}, and Lemma~\ref{etaestimate} in Lemma~\ref{thetaestimate_l2norm} to obtain the following estimate 
\begin{align}\label{thetaestimate}
   \|\theta^n\| \leq C\;\log_e(n+2)\;(h^2 + N^{-\min\{\gamma\alpha, 2-\alpha\}}), \quad 1 \leq n \leq N, 
\end{align}
and then apply the estimate (\ref{zetaestimate}), (\ref{thetaestimate}), Lemma~\ref{DiscreteDerivativeDifference}, Lemma~\ref{truncationI} and Lemma~\ref{etaestimate} in Lemma~\ref{xiestimate_l2norm} to get 
\begin{align}
    \label{xiestimate}  t_n^{\alpha/2}\|\boldsymbol{\xi}^n\| \leq C\;\log_e(n+2)\;(h^2 + N^{-\min\{\gamma\alpha, 2-\alpha\}}), \quad 1 \leq n \leq N.
\end{align}
Thus, an application of the triangle inequality and the above three estimates (\ref{zetaestimate}), (\ref{thetaestimate}) and (\ref{xiestimate}) yields the estimate (\ref{finalestimateforuI}). Finally, if $u_0\in\dot{H}^{2+\epsilon}(\Omega)$ for some $\epsilon >0$, and $\frac{(2-\alpha)}{\alpha}\;<\;\gamma\;\leq\;\frac{2(2-\alpha)}{\alpha} $ then an appeal to Lemma~\ref{truncationI}, and Lemma~\ref{etaestimate} guarantee the validity of the estimate (\ref{finalestimateforuI}) without any logarithmic factor $\log_e(N)$.
\end{proof}

\begin{remark}\label{remarkonmaintheorem}
For any fixed $\alpha \in (0, 1)$, assuming the regularity conditions in Theorem~\ref{Regularity_condition}, Remark~\ref{remarkontruncationerror} and Remark~\ref{remarkonetaandzitaestimates} guarantee that the estimate (\ref{finalestimateforuI}) holds without the presence of any logarithmic factor $\log_e(N)$, even if $u_0 \notin \dot{H}^{2+\epsilon}$ for $\epsilon >0$ and $1\leq \gamma \leq 2(2-\alpha)/\alpha$.
\end{remark}

As a consequence of the estimate (\ref{finalestimateforuI}) in the previous Theorem~\ref{L2H1errortheorem}, in one and two dimensions, we obtain the following $L^\infty$-norm estimate.
\begin{corollary}\label{Linftynormestimate}
Under the assumptions in Theorem~\ref{L2H1errortheorem}, for $d=1,2$, and quasi-uniform triangulation $\mathcal{T}_h$, there exists a positive constant $C$ independent of $h$, $N$ and $p$, such that,
\begin{align*}
    & \|u_h^n - u(t_{n}) \|_{L^\infty(\Omega)}  \leq C \left( h^{2-\frac{d}{p}}\|u(t_n)\|_{W^{2,p}(\Omega)} + \ell_{h,d}\;\|\boldsymbol{\sigma}_h^n - \boldsymbol{\sigma}(t_{n}) \|\right),\;1\leq n\leq N, \; p\geq d,\;N>2,
\end{align*}
where $\|\boldsymbol{\sigma}_h^n - \boldsymbol{\sigma}(t_{n}) \|$ is given by the estimate (\ref{finalestimateforuI}) and the constant $C$ remains bounded as $\alpha\to 1^{-}$.
\end{corollary}
\begin{proof}
By applying the Lemma~\ref{VTlemma} in relation (\ref{theta1}), we obtain
\begin{align}
   \nonumber& \|P_h u(t_n) - u_h^n\|_{L^\infty(\Omega)} \leq C \ell_{h,d}\;\|\boldsymbol{\sigma}_h^n - \boldsymbol{\sigma}(t_{n}) \|.
\end{align}
Thus, the result follows after applying triangle inequality and the following standard estimate for the $L^2$-projection $P_h: L^2(\Omega)\to V_h$ (see, \cite{MR2249024}),
\begin{align}
   \nonumber& \|P_h \phi - \phi\|_{L^\infty(\Omega)} \leq C h^{2-\frac{d}{p}}\|\phi\|_{W^{2,p}(\Omega)}, \quad  \phi\in H_0^1(\Omega)\cap W^{2,p}(\Omega),\;p \geq d.
\end{align}
This completes the rest of the proof.
\end{proof}
\begin{remark}
If $\boldsymbol{A}$ is time-independent, $\boldsymbol{b}=\boldsymbol{0}$, and $c\geq 0$, then $\widetilde{b}:=\boldsymbol{b}^T\boldsymbol{A}^{-1}\boldsymbol{b} =0$, $\widetilde{c}:=\max(0,-c)=0$, and $L_B =0$. Consequently, the expressions 
\begin{align*}
    &\lambda^S_\epsilon:= \frac{1+\epsilon}{2}\|\widetilde{b}\|_{L^\infty(\Omega\times J)}  + 2\|\widetilde{c}\|_{L^\infty(\Omega\times J)} 
 + \epsilon|\lambda |\;\text{ and} \\
&\lambda^F_\epsilon:=\frac{1+\epsilon}{2}\|\widetilde{b}\|_{L^\infty(\Omega\times J)} + \frac{L_B T^{1-\alpha}}{\beta_0\Gamma(2-\alpha)},
\end{align*}
shows that the time-step restriction in Theorem~\ref{stabilitytheorem1} is quite mild. If we further assume $\lambda=0$ (i.e., no integral term is involved), there is no time-step restriction in Theorem~\ref{stabilitytheorem1}.
\end{remark}
\section{Numerical results}\label{section7}  
 This section focusses on several numerical experiments whose results confirm our theoretical findings.
  Let $h^2 = \frac{1}{2}N^{-(2-\alpha)}$ in Theorem~\ref{L2H1errortheorem} and Corollary~\ref{Linftynormestimate}, respectively. Then the rate of convergence with respect to $L^2$-norm and max-norm will be computed using the formulae 
\begin{align*}
    &R_{\phi,h} = \log_e\left(\frac{E_{\phi,h_1}}{E_{\phi,h_2}}\right)/\log_e\left(\frac{h_1}{h_2}\right),\quad R_{\phi,\Delta t} = \log_e\left(\frac{E_{\phi,h_1}}{E_{\phi,h_2}}\right)/\log_e\left(\frac{\Delta t_1}{\Delta t_2}\right),\;\text{ and}\\
    &R_{h}^\infty = \log_e\left(\frac{E^\infty_{u,h_1}}{E^\infty_{u,h_2}}\right)/\log_e\left(\frac{h_1}{h_2}\right),\quad R_{\Delta t}^\infty = \log_e\left(\frac{E^\infty_{u,h_1}}{E^\infty_{u,h_2}}\right)/\log_e\left(\frac{\Delta t_1}{\Delta t_2}\right),
\end{align*}
respectively, where 
\begin{align}
    \nonumber & E_{u,h}:= \begin{cases}
        \max\limits_{1 \leq n \leq N}\|u_h^n - u(t_n)\| & \text{: if the analytical solution $u$ is known,}\\
        \max\limits_{1 \leq n \leq N}\|u_{h}^n - \widetilde{u}_h(t_n)\|  & \text{: otherwise,}
    \end{cases}\\
     \nonumber & E_{\boldsymbol{\sigma},h}:= \begin{cases}
     \max\limits_{1 \leq n \leq N}\;t_n^{\frac{\alpha}{2}}\|\boldsymbol{\sigma}_h^n - \boldsymbol{\sigma}(t_n)\| & \text{: if the analytical solution $u$ is known,}\\
     \max\limits_{1 \leq n \leq N}t_n^\frac{\alpha}{2}\|\boldsymbol{\sigma}_h^n - \widetilde{\boldsymbol{\sigma}}_h(t_n)\|  & \text{: otherwise, }
    \end{cases}\\
    \nonumber &
    E_{h}^\infty:= \begin{cases}
        \max\limits_{\boldsymbol{x}_j\in\mathcal{N},\;1 \leq n \leq N}\;t_n^{\frac{\alpha}{2}}|u_h^n(\boldsymbol{x}_j) - u(\boldsymbol{x}_j,\;t_n)| &\text{: if the analytical solution $u$ is known,}\\
        \max\limits_{\boldsymbol{x}_j\in\mathcal{N},\;1 \leq n \leq N}\;t_n^{\frac{\alpha}{2}}|u_h^n(\boldsymbol{x}_j) - \widetilde{u}_h(\boldsymbol{x}_j,\;t_n)| &\text{: otherwise,}
        \end{cases}
    \end{align}
    \noindent $\mathcal{N}:=\{\boldsymbol{x}_j\}_{j=1}^{N_h + 500} $ is a collection of points in the triangulation $\mathcal{T}_h\cap\overline{\Omega}$, $N_h$ denotes the number of nodal points in $\mathcal{T}_h\cap\overline{\Omega}$, and $(\widetilde{u}_h(t),\;\widetilde{\boldsymbol{\sigma}}_h(t)),\;t\in[0,\;T]$, is obtained from $ (u^k_{\frac{h}{2}},\;\boldsymbol{\sigma}^k_{\frac{h}{2}}),\;0\leq k \leq 2N$, by applying a piece-wise linear interpolation in time and a piece-wise quadratic interpolation in space direction. 
    
In all tables, we have used graded mesh (\ref{gradedmesh1}) with the grading parameter $\gamma = \frac{2-\alpha}{\alpha}+0.1$ to address the initial singularity effectively for non-smooth solution in the temporal direction. An upper bound $\widetilde{\Delta t}$ for the time-step restriction in Theorem~\ref{L2H1errortheorem} is estimated as follows:
\begin{align*}
    \max\limits_{1\leq n \leq N}\Delta t_n \;& \leq \;\min\Big\{\left( \delta\;\lambda^S_\epsilon\Gamma(2-\alpha)\right)^{-1/\alpha}, \left( \delta\; \lambda^F_\epsilon\Gamma(2-\alpha)\right)^{-1/\alpha}\Big\}\\
    &\approx \;\min\Big\{\left( \delta\;\widetilde{\lambda}^S_\epsilon\Gamma(2-\alpha)\right)^{-1/\alpha}, \left( \delta\; \widetilde{\lambda}^F_\epsilon\Gamma(2-\alpha)\right)^{-1/\alpha}\Big\}\;=:\widetilde{\Delta t},
\end{align*}
where 
\begin{align*}
\widetilde{\lambda}^S_\epsilon & = \displaystyle{\frac{1+\epsilon}{2}\max\limits_{\boldsymbol{x}_j\in N_h,\;1 \leq n \leq N}}|\widetilde{b}(x_j,t_n)|  + \;2 \max\limits_{\boldsymbol{x}_j\in N_h,\;1 \leq n \leq N} |\widetilde{c}(x_j,t_n)| 
 + \epsilon|\lambda | \approx \lambda^S_\epsilon,\quad N=64,\\
 \widetilde{\lambda}^F_\epsilon  &= \displaystyle{\frac{1+\epsilon}{2}}\max\limits_{\boldsymbol{x}_j\in N_h,\;1 \leq n \leq N}|\widetilde{b}(x_i,t_n)| + \frac{L_B T^{1-\alpha}}{\beta_0\Gamma(2-\alpha)} \approx \lambda^F_\epsilon,\quad N=64, \epsilon = 0.1,\;\delta = 1.1,\\
 \widetilde{L}_B  &= \max\limits_{\boldsymbol{x}_j\in N_h,\;1 \leq n \leq N} \|\boldsymbol{A}^{-1}(\boldsymbol{x}_j,t_n)\|_2^2\times \max\limits_{\boldsymbol{x}_j\in N_h,\;1 \leq n \leq N}\left\|\frac{d}{dt}\boldsymbol{A}(\boldsymbol{x}_j,t_n) \right\|_2\\
 &\approx\; \max\limits_{t\in[0,T]}\|\boldsymbol{A}^{-1}(\cdot,t)\|_{L^\infty(\Omega;\mathbb{R}^{d\times d})}^2\times \max\limits_{t\in[0,T]}\left\|\frac{d}{dt}\boldsymbol{A}(\cdot,t) \right\|_{L^\infty(\Omega;\mathbb{R}^{d\times d})}=:L_B,\\
 \widetilde{\beta}_0 & = \left( \max\limits_{\boldsymbol{x}_j\in N_h,\;1 \leq n \leq N} \| \boldsymbol{A}^{1/2}(\boldsymbol{x}_j,t_n)  \|_2\right)^{-2}\;\approx\;\left(\max\limits_{t\in[0,T]}\|\boldsymbol{A}^{1/2}(\cdot,t)\|_{L^\infty(\Omega;\mathbb{R}^{d\times d})}\right)^{-2} =:\beta_0,\\
 \widetilde{b}&:= \boldsymbol{b}^T\boldsymbol{A}^{-1}\boldsymbol{b},\;\text{ and }\; \widetilde{c}:=\max(0,-c).
\end{align*}  
The results include the error $E_{u,h}$, $E_{\boldsymbol{\sigma},h}$, and $E_{h}^\infty$, as well as the computed rate of convergence $R_{u,h}$, $R_{u,\Delta t}$, $R_{\boldsymbol{\sigma},h}$, $R_{\boldsymbol{\sigma},\Delta t}$, $R^\infty_h$ and $R^\infty_{\Delta t}$ for $\alpha = 0.2, \;0.5, \;0.8$ and $0.99$. {The implementation has been performed in FreeFem++ using the Raviart-Thomas finite elements $(P1dc, RT1)$, where $P1dc$ denotes the piecewise linear discontinuous finite element.

In the following three examples , say, Example~\ref{Time-indep-VariablePDEMFEM1},Example~\ref{VariablePIDEMFEM2_f8} and Example~\ref{VariablePIDEMFEM2_f10}, the time step length restriction as given in  Theorem~\ref{L2H1errortheorem} is almost satisfied.}

\begin{example}\label{Time-indep-VariablePDEMFEM1} 
We begin by considering a time-fractional PDE (\ref{pide}) over the domain $\Omega\times(0,T] = (-1,1)^2\times (0,0.5]$ with  coefficients specified as $\textbf{A}(\boldsymbol{x},t)=\begin{bmatrix} 1 + 0.1x_1^2 & 0.1x_1x_2 \\ 0.1x_1x_2 & 1 + 0.2x_2^2 \\ \end{bmatrix}$, $\textbf{b}(\boldsymbol{x},t) =\begin{bmatrix} 0\\ 0\\ \end{bmatrix}$, $c(\boldsymbol{x},t) =1 -x_1x_2e^{-1}$, and $\lambda =0$. The initial condition $u_0$ and source term $f$ are chosen to match the exact solution $ u(\boldsymbol{x},t)= \sin(\pi x_1)\sin(\pi x_2)(1 + t^{\alpha})$. Table~\ref{Time-indep-VariablePDEMFEM-table} shows the computed rates of convergence in both spatial and temporal directions, which confirm consistency between the computed and theoretical convergence rates. For this example, since $\textbf{A}$ is time-independent, $\textbf{b}=0$ and $c\geq 0$, there is no time-step restriction.

\begin{table}[H]
\centering
\begin{tabular}{||l|l|l|l|l|l|l||}
\hline
\multicolumn{2}{||l|}{$N$} & 4 & 8 & 16 & 32 & 64  \\ \hline \hline
\multicolumn{1}{||l|}{\multirow{6}{*}{$\alpha$=0.2}} 
 & $E_{u,h}$ & 
4.707e-01 & 1.284e-01 & 3.284e-02 & 8.992e-03 & 2.701e-03\\ \cline{2-7} 
\multicolumn{1}{||l|}{}  & $R_{u,h}$ & - & 1.87 & 1.97 & 1.99 & 2.00\\ \cline{2-7}
\multicolumn{1}{||l|}{}  & $R_{u,\Delta t}$ & - & 4.70 & 2.97 & 2.27 & 1.91\\ \cline{2-7}
 & $E_{\boldsymbol{\sigma},h}$ & 
1.431e+0 & 3.686e-01 & 9.313e-02 & 2.547e-02 & 7.655e-03\\ \cline{2-7} 
\multicolumn{1}{||l|}{}  & $R_{\boldsymbol{\sigma},h}$ & - & 1.96 & 1.98 & 1.99 & 2.00\\ \cline{2-7}
\multicolumn{1}{||l|}{}  & $R_{\boldsymbol{\sigma}, \Delta t}$ & - & 4.91 & 2.99 & 2.27 & 1.91\\ \cline{2-7}
 & $E^\infty_h$ & 
1.059e+0 & 3.411e-01 & 8.874e-02 & 2.412e-02 & 7.302e-03\\ \cline{2-7} 
\multicolumn{1}{||l|}{}  & $R^\infty_h$ & - & 1.64 & 1.94 & 2.00 & 1.98\\ \cline{2-7}
\multicolumn{1}{||l|}{}  & $R^\infty_{\Delta t}$ & - & 4.10 & 2.93 & 2.28 & 1.89
\\ \hline \hline
\multicolumn{1}{||l|}{\multirow{6}{*}{$\alpha$=0.5}} 
 & $E_{u,h}$ & 
4.295e-01 & 1.665e-01 & 6.680e-02 & 2.557e-02 & 8.206e-03\\ \cline{2-7} 
\multicolumn{1}{||l|}{}  & $R_{u,h}$ & - & 1.85 & 1.94 & 1.98 & 1.99\\ \cline{2-7}
\multicolumn{1}{||l|}{}  & $R_{u,\Delta t}$ & - & 1.71 & 1.46 & 1.46 & 1.68\\ \cline{2-7}
 & $E_{\boldsymbol{\sigma},h}$ & 
1.177e+0 & 4.344e-01 & 1.714e-01 & 6.531e-02 & 2.094e-02\\ \cline{2-7} 
\multicolumn{1}{||l|}{}  & $R_{\boldsymbol{\sigma},h}$ & - & 1.95 & 1.98 & 1.99 & 1.99\\ \cline{2-7}
\multicolumn{1}{||l|}{}  & $R_{\boldsymbol{\sigma}, \Delta t}$ & - & 1.8 & 1.49 & 1.46 & 1.68\\ \cline{2-7}
 & $E^\infty_h$ & 
8.714e-01 & 3.83e-01 & 1.547e-01 & 6.406e-02 & 1.986e-02\\ \cline{2-7} 
\multicolumn{1}{||l|}{}  & $R^\infty_h$ & - & 1.61 & 1.93 & 1.82 & 2.05\\ \cline{2-7}
\multicolumn{1}{||l|}{}  & $R^\infty_{\Delta t}$ & - & 1.48 & 1.45 & 1.34 & 1.73
\\ \hline \hline
\multicolumn{1}{||l|}{\multirow{6}{*}{$\alpha$=0.8}} 
 & $E_{u,h}$ & 
9.630e-01 & 3.961e-01 & 1.535e-01 & 6.16e-02 & 2.764e-02\\ \cline{2-7} 
\multicolumn{1}{||l|}{}  & $R_{u,h}$ & - & 2.19 & 1.85 & 1.94 & 1.98\\ \cline{2-7}
\multicolumn{1}{||l|}{}  & $R_{u,\Delta t}$ & - & 1.36 & 1.41 & 1.34 & 1.16\\ \cline{2-7}
 & $E_{\boldsymbol{\sigma},h}$ & 
2.213e+0 & 9.784e-01 & 3.61e-01 & 1.425e-01 & 6.366e-02\\ \cline{2-7} 
\multicolumn{1}{||l|}{}  & $R_{\boldsymbol{\sigma},h}$ & - & 2.01 & 1.95 & 1.98 & 1.99\\ \cline{2-7}
\multicolumn{1}{||l|}{}  & $R_{\boldsymbol{\sigma}, \Delta t}$ & - & 1.25 & 1.48 & 1.36 & 1.17\\ \cline{2-7}
 & $E^\infty_h$ & 
9.143e-01 & 7.253e-01 & 3.186e-01 & 1.286e-01 & 6.074e-02\\ \cline{2-7} 
\multicolumn{1}{||l|}{}  & $R^\infty_h$ & - & 0.57 & 1.61 & 1.93 & 1.85\\ \cline{2-7}
\multicolumn{1}{||l|}{}  & $R^\infty_{\Delta t}$ & - & 0.36 & 1.22 & 1.33 & 1.09
\\  \hline \hline
\multicolumn{1}{||l|}{\multirow{6}{*}{$\alpha$=0.99}} 
 & $E_{u,h}$ & 
9.196e-01 & 3.783e-01 & 2.237e-01 & 1.032e-01 & 5.883e-02\\ \cline{2-7} 
\multicolumn{1}{||l|}{}  & $R_{u,h}$ & - & 2.19 & 1.83 & 1.91 & 1.95\\ \cline{2-7}
\multicolumn{1}{||l|}{}  & $R_{u,\Delta t}$ & - & 1.30 & 0.76 & 1.12 & 0.81 \\ \cline{2-7}
 & $E_{\boldsymbol{\sigma},h}$ & 
1.979e+0 & 8.749e-01 & 4.999e-01 & 2.252e-01 & 1.274e-01\\ \cline{2-7} 
\multicolumn{1}{||l|}{}  & $R_{\boldsymbol{\sigma},h}$ & - & 2.01 & 1.95 & 1.97 & 1.98\\ \cline{2-7}
\multicolumn{1}{||l|}{}  & $R_{\boldsymbol{\sigma}, \Delta t}$ & - & 1.19 & 0.812 & 1.15 & 0.823\\ \cline{2-7}
 & $E^\infty_h$ & 
8.177e-01 & 6.487e-01 & 3.759e-01 & 2.088e-01 & 1.150e-01\\ \cline{2-7} 
\multicolumn{1}{||l|}{}  & $R^\infty_h$ & - & 0.57 & 1.90 & 1.45 & 2.07\\ \cline{2-7}
\multicolumn{1}{||l|}{}  & $R^\infty_{\Delta t}$ & - & 0.34 & 0.79 & 0.85 & 0.86 \\ \hline
\end{tabular}
\caption{Error $E_{u,h}$, $E_{\boldsymbol{\sigma},h}$, $E^\infty_h$ and rate of convergence $R_{u,h}$, $R_{u,\Delta t}$, $R_{\boldsymbol{\sigma},h}$, $R_{\boldsymbol{\sigma},\Delta t}$, $R^\infty_h$ and $R^\infty_{\Delta t}$ of the proposed method for Example~\ref{Time-indep-VariablePDEMFEM1}.}\label{Time-indep-VariablePDEMFEM-table}
\end{table}

\end{example}

\begin{example}\label{VariablePIDEMFEM2_f8} 
We now consider a time-fractional PDE (\ref{pide}) over the domain $\Omega\times(0,T] = (-1,1)^2\times (0,0.5]$ with space-time dependent coefficients $\textbf{A}(\boldsymbol{x},t)=\begin{bmatrix} 1 + 0.1x_1^2t & 0.1x_1x_2t \\ 0.1x_1x_2t & 1 + 0.2x_2^2t \\ \end{bmatrix}$, $\textbf{b}(\boldsymbol{x},t) =\begin{bmatrix} x_1e^{-t}\\ x_2 e^{-t}\\ \end{bmatrix}$, $c(\boldsymbol{x},t) =1 -x_1x_2e^{-t}$, and $\lambda =0$. The initial condition $u_0$ and the source term $f$ are chosen to align with the exact solution $ u(\boldsymbol{x},t)= \sin(\pi x_1)\sin(\pi x_2)(1 + t^{\alpha})$. The computed rates of convergence for both spatial and temporal directions are presented in Table~\ref{2dM222_f1}. These results confirm agreement between the computed and theoretical orders of convergence.

\begin{table}[H]
\centering
\begin{tabular}{||l|l|l|l|l|l|l||}
\hline
\multicolumn{2}{||l|}{$N$} & 4 & 8 & 16 & 32 & 64  \\ \hline \hline
\multicolumn{1}{||l|}{\multirow{6}{*}{$\alpha$=0.2}} 
& $\widetilde{\Delta t}$ & \multicolumn{5}{|l||} {\hspace{3.7cm}8.859e-01}
\\ \cline{2-7} 
\multicolumn{1}{||l|}{}  & $\Delta t$
& 4.635e-01 & 3.516e-01 & 2.22e-01 & 1.254e-01 & 6.675e-02\\ \cline{2-7} 
 & $E_{u,h}$ & 
4.716e-01 & 1.284e-01 & 3.284e-02 & 8.992e-03 & 2.701e-03\\ \cline{2-7} 
\multicolumn{1}{||l|}{}  & $R_{u,h}$ & - & 1.88 & 1.97 & 1.99 & 2.00\\ \cline{2-7}
\multicolumn{1}{||l|}{}  & $R_{u,\Delta t}$ & - & 4.71 & 2.97 & 2.27 & 1.91\\ \cline{2-7}
 & $E_{\boldsymbol{\sigma},h}$ & 
1.380e+0 & 3.564e-01 & 8.988e-02 & 2.455e-02 & 7.377e-03\\ \cline{2-7} 
\multicolumn{1}{||l|}{}  & $R_{\boldsymbol{\sigma},h}$ & - & 1.95 & 1.99 & 1.99 & 2.00\\ \cline{2-7}
\multicolumn{1}{||l|}{}  & $R_{\boldsymbol{\sigma}, \Delta t}$ & - & 4.90 & 3.00 & 2.27 & 1.91\\ \cline{2-7}
 & $E^\infty_h$ & 
1.053e+0 & 3.41e-01 & 8.872e-02 & 2.411e-02 & 7.301e-03\\ \cline{2-7} 
\multicolumn{1}{||l|}{}  & $R^\infty_h$ & - & 1.63 & 1.94 & 2.00 & 1.98\\ \cline{2-7}
\multicolumn{1}{||l|}{}  & $R^\infty_{\Delta t}$ & - & 4.08 & 2.93 & 2.28 & 1.89
\\ \hline \hline
\multicolumn{1}{||l|}{\multirow{6}{*}{$\alpha$=0.5}} 
& $\widetilde{\Delta t}$ & \multicolumn{5}{|l||} {\hspace{3.7cm}1.052e+0}
\\ \cline{2-7} 
\multicolumn{1}{||l|}{}  & $\Delta t$
& 2.95e-01 & 1.694e-01 & 9.066e-02 & 4.686e-02 & 2.382e-02\\ \cline{2-7} 
 & $E_{u,h}$ & 
4.304e-01 & 1.665e-01 & 6.68e-02 & 2.557e-02 & 8.206e-03\\ \cline{2-7} 
\multicolumn{1}{||l|}{}  & $R_{u,h}$ & - & 1.86 & 1.94 & 1.98 & 1.99\\ \cline{2-7}
\multicolumn{1}{||l|}{}  & $R_{u,\Delta t}$ & - & 1.71 & 1.46 & 1.46 & 1.68\\ \cline{2-7}
 & $E_{\boldsymbol{\sigma},h}$ & 
1.135e+0 & 4.202e-01 & 1.656e-01 & 6.302e-02 & 2.019e-02\\ \cline{2-7} 
\multicolumn{1}{||l|}{}  & $R_{\boldsymbol{\sigma},h}$ & - & 1.95 & 1.98 & 1.99 & 1.99\\ \cline{2-7}
\multicolumn{1}{||l|}{}  & $R_{\boldsymbol{\sigma}, \Delta t}$ & - & 1.79 & 1.49 & 1.46 & 1.68\\ \cline{2-7}
 & $E^\infty_h$ & 
8.656e-01 & 3.822e-01 & 1.546e-01 & 6.404e-02 & 1.985e-02\\ \cline{2-7} 
\multicolumn{1}{||l|}{}  & $R^\infty_h$ & - & 1.60 & 1.93 & 1.82 & 2.05\\ \cline{2-7}
\multicolumn{1}{||l|}{}  & $R^\infty_{\Delta t}$ & - & 1.47 & 1.45 & 1.34 & 1.73
\\ \hline \hline
\multicolumn{1}{||l|}{\multirow{6}{*}{$\alpha$=0.8}} 
& $\widetilde{\Delta t}$& \multicolumn{5}{|l||} {\hspace{3.7cm}9.877e-01}
\\ \cline{2-7} 
\multicolumn{1}{||l|}{}  & $\Delta t$
& 1.844e-01 & 9.618e-02 & 4.905e-02 & 2.476e-02 & 1.244e-02\\ \cline{2-7} 
 & $E_{u,h}$ & 
9.668e-01 & 3.969e-01 & 1.536e-01 & 6.160e-02 & 2.764e-02\\ \cline{2-7} 
\multicolumn{1}{||l|}{}  & $R_{u,h}$ & - & 2.20 & 1.86 & 1.94 & 1.98\\ \cline{2-7}
\multicolumn{1}{||l|}{}  & $R_{u,\Delta t}$ & - & 1.37 & 1.41 & 1.34 & 1.16\\ \cline{2-7}
 & $E_{\boldsymbol{\sigma},h}$ & 
2.165e+0 & 9.437e-01 & 3.492e-01 & 1.376e-01 & 6.144e-02\\ \cline{2-7} 
\multicolumn{1}{||l|}{}  & $R_{\boldsymbol{\sigma},h}$ & - & 2.05 & 1.95 & 1.98 & 1.99\\ \cline{2-7}
\multicolumn{1}{||l|}{}  & $R_{\boldsymbol{\sigma}, \Delta t}$ & - & 1.28 & 1.48 & 1.36 & 1.17\\ \cline{2-7}
 & $E^\infty_h$ & 
9.296e-01 & 7.203e-01 & 3.179e-01 & 1.285e-01 & 6.073e-02\\ \cline{2-7} 
\multicolumn{1}{||l|}{}  & $R^\infty_h$ & - & 0.63 & 1.60 & 1.93 & 1.85\\ \cline{2-7}
\multicolumn{1}{||l|}{}  & $R^\infty_{\Delta t}$ & - & 0.39 & 1.21 & 1.32 & 1.09 \\ \hline \hline
\multicolumn{1}{||l|}{\multirow{6}{*}{$\alpha$=0.99}} 
& $\widetilde{\Delta t}$& \multicolumn{5}{|l||} {\hspace{3.7cm}9.135e-01}
\\ \cline{2-7} 
\multicolumn{1}{||l|}{}  & $\Delta t$
& 1.377e-01 & 6.946e-02 & 3.487e-02 & 1.746e-02 & 8.743e-03\\ \cline{2-7} 
 & $E_{u,h}$ & 
9.232e-01 & 3.79e-01 & 2.239e-01 & 1.032e-01 & 5.883e-02\\ \cline{2-7} 
\multicolumn{1}{||l|}{}  & $R_{u,h}$ & - & 2.20 & 1.83 & 1.91 & 1.95\\ \cline{2-7}
\multicolumn{1}{||l|}{}  & $R_{u,\Delta t}$ & - & 1.30 & 0.76 & 1.12 & 0.81 \\ \cline{2-7}
 & $E_{\boldsymbol{\sigma},h}$ & 
1.936e+0 & 8.440e-01 & 4.838e-01 & 2.178e-01 & 1.230e-01\\ \cline{2-7} 
\multicolumn{1}{||l|}{}  & $R_{\boldsymbol{\sigma},h}$ & - & 2.05 & 1.93 & 1.97 & 1.98\\ \cline{2-7}
\multicolumn{1}{||l|}{}  & $R_{\boldsymbol{\sigma}, \Delta t}$ & - & 1.21 & 0.81 & 1.15 & 0.82\\ \cline{2-7}
 & $E^\infty_h$ & 
    8.306e-01 & 6.440e-01 & 3.761e-01 & 2.087e-01 & 1.150e-01\\ \cline{2-7} 
\multicolumn{1}{||l|}{}  & $R^\infty_h$ & - & 0.63 & 1.87 & 1.45 & 2.07\\ \cline{2-7}
\multicolumn{1}{||l|}{}  & $R^\infty_{\Delta t}$ & - & 0.37 & 0.78 & 0.85 & 0.86 \\ \hline
\end{tabular}
\caption{Error $E_{u,h}$, $E_{\boldsymbol{\sigma},h}$, $E^\infty_h$ and rate of convergence $R_{u,h}$, $R_{u,\Delta t}$, $R_{\boldsymbol{\sigma},h}$, $R_{\boldsymbol{\sigma},\Delta t}$, $R^\infty_h$ and $R^\infty_{\Delta t}$ of the proposed method for Example~\ref{VariablePIDEMFEM2_f8}.}\label{2dM222_f1}
\end{table}

\end{example}

\begin{example}\label{VariablePIDEMFEM2_f10} 
In this example, we consider the time-fractional PIDE (\ref{pide}) with variable coefficients over the domain $\Omega\times(0,T] = (-1,1)^2\times (0,1]$ with the initial condition as $u_0 = x_1(1-|x_1|)x_2(1-|x_2|)$ and $\textbf{A}(\boldsymbol{x},t)=\begin{bmatrix} 1 + 0.1x_1^2t & 0.1x_1x_2t \\ 0.1x_1x_2t & 1 + 0.2x_2^2t \\ \end{bmatrix}$, $\textbf{b}(\boldsymbol{x},t) =\begin{bmatrix} x_1e^{-t}\\ x_2 e^{-t}\\ \end{bmatrix}$, $c(\boldsymbol{x},t) =1 -x_1x_2e^{-t}$, $\lambda = \frac{1}{2}\;$, $g(\boldsymbol{x},\boldsymbol{y}) = e^{-\|\boldsymbol{x}-\boldsymbol{y}\|^2}$, and the source term $f(\boldsymbol{x}, t)=e^{-t}\sin(\pi x_1)\sin(\pi x_2)$.
Table~\ref{2dpide_f10} demonstrates that the computational rate is second-order for both variables $u$ and $\boldsymbol{\sigma}$, consistent with our theoretical findings. Here, we have avoided the inversion of a dense matrix without compromising the rate of convergence by using the proposed IMEX method. This method explicitly handles the integral term, employing a second-order extrapolation as defined in (\ref{2ndOrderExtrapolation}).

\begin{table}[H]
\centering
\begin{tabular}{||l|l|l|l|l|l|l||}
\hline
\multicolumn{2}{||l|}{$N$} & 4 & 8 & 16 & 32 & 64  \\ \hline \hline
\multicolumn{1}{||l|}{\multirow{6}{*}{$\alpha$=0.2}} 
& $\widetilde{\Delta t}$ & \multicolumn{5}{|l||} {\hspace{3.7cm}6.941e-01}
\\ \cline{2-7} 
\multicolumn{1}{||l|}{}  & $\Delta t$
& 9.270e-01 & 7.033e-01 & 4.441e-01 & 2.509e-01 & 1.335e-01
\\ \cline{2-7} 
\multicolumn{1}{||l|}{} & 
 $E_{u,h}$ & 1.789e-02 & 5.412e-03 & 1.392e-03 & 3.781e-04 & 1.146e-04 \\ \cline{2-7}
\multicolumn{1}{||l|}{}  & $R_{u,h}$ & - & 1.72 & 1.96 & 2.00 & 1.98\\ \cline{2-7}
\multicolumn{1}{||l|}{}  & $R_{u,\Delta t}$ & - & 4.33 & 2.96 & 2.28 & 1.89\\ \cline{2-7}
\multicolumn{1}{||l|}{}  & $E_{\boldsymbol{\sigma},h}$ & 3.059e-02 & 8.138e-03 & 2.055e-03 & 5.611e-04 & 1.687e-04 \\ \cline{2-7}
\multicolumn{1}{||l|}{}  & $R_{\boldsymbol{\sigma},h}$ & - & 1.91 & 1.99 & 2.00 & 2.00\\ \cline{2-7}
\multicolumn{1}{||l|}{}  & $R_{\boldsymbol{\sigma}, \Delta t}$ & - & 4.80 & 2.99 & 2.27 & 1.90\\ \cline{2-7}
\multicolumn{1}{||l|}{}  & $E^\infty_h$ & 1.841e-02 & 5.541e-03 & 1.405e-03 & 3.861e-04 & 1.199e-04 \\ \cline{2-7}
\multicolumn{1}{||l|}{}  & $R^\infty_h$ & - & 1.73 & 1.98 & 1.99 & 1.94\\ \cline{2-7}
\multicolumn{1}{||l|}{}  & $R^\infty_{\Delta t}$ & - & 4.35 & 2.98 & 2.26 & 1.85
\\ \hline \hline
\multicolumn{1}{||l|}{\multirow{6}{*}{$\alpha$=0.5}} 
& $\widetilde{\Delta t}$ & \multicolumn{5}{|l||} {\hspace{3.7cm}9.544e-01}
\\ \cline{2-7} 
\multicolumn{1}{||l|}{}  & $\Delta t$
& 5.901e-01 & 3.389e-01 & 1.813e-01 & 9.373e-02 & 4.764e-02
\\ \cline{2-7} 
\multicolumn{1}{||l|}{} & 
 $E_{u,h}$ & 1.788e-02 & 7.328e-03 & 3.094e-03 & 1.168e-03 & 3.781e-04\\ \cline{2-7} 
\multicolumn{1}{||l|}{}  & $R_{u,h}$ & - & 1.75 & 1.83 & 2.01 & 1.98\\ \cline{2-7}
\multicolumn{1}{||l|}{}  & $R_{u,\Delta t}$ & - & 1.61 & 1.38 & 1.48 & 1.67\\ \cline{2-7}
\multicolumn{1}{||l|}{}   & $E_{\boldsymbol{\sigma},h}$ & 
2.249e-02 & 8.876e-03 & 3.514e-03 & 1.336e-03 & 4.283e-04\\ \cline{2-7} 
\multicolumn{1}{||l|}{}  & $R_{\boldsymbol{\sigma},h}$ & - & 1.82 & 1.97 & 1.99 & 1.99\\ \cline{2-7}
\multicolumn{1}{||l|}{}  & $R_{\boldsymbol{\sigma}, \Delta t}$ & - & 1.68 & 1.48 & 1.46 & 1.68\\ \cline{2-7}
 \multicolumn{1}{||l|}{}  & $E^\infty_h$ & 
1.367e-02 & 5.724e-03 & 2.354e-03 & 9.46e-04 & 2.97e-04\\ \cline{2-7} 
\multicolumn{1}{||l|}{}  & $R^\infty_h$ & - & 1.71 & 1.89 & 1.88 & 2.03\\ \cline{2-7}
\multicolumn{1}{||l|}{}  & $R^\infty_{\Delta t}$ & - & 1.57 & 1.42 & 1.38 & 1.71
\\ \hline \hline
\multicolumn{1}{||l|}{\multirow{6}{*}{$\alpha$=0.8}} 
& $\widetilde{\Delta t}$& \multicolumn{5}{|l||} {\hspace{3.7cm}9.292e-01}
\\ \cline{2-7} 
\multicolumn{1}{||l|}{}  & $\Delta t$
& 3.689e-01 & 1.923e-01 & 9.810e-02 & 4.952e-02 & 2.488e-02\\ \cline{2-7} 
\multicolumn{1}{||l|}{}  & $E_{u,h}$ & 
4.044e-02 & 1.788e-02 & 7.328e-03 & 3.094e-03 & 1.391e-03\\ \cline{2-7} 
\multicolumn{1}{||l|}{}  & $R_{u,h}$ & - & 2.01 & 1.75 & 1.83 & 1.97\\ \cline{2-7}
\multicolumn{1}{||l|}{}  & $R_{u,\Delta t}$ & - & 1.25 & 1.33 & 1.26 & 1.16\\ \cline{2-7}
\multicolumn{1}{||l|}{}   & $E_{\boldsymbol{\sigma},h}$ & 
5.075e-02 & 1.975e-02 & 7.406e-03 & 2.931e-03 & 1.308e-03\\ \cline{2-7} 
\multicolumn{1}{||l|}{}  & $R_{\boldsymbol{\sigma},h}$ & - & 2.33 & 1.92 & 1.97 & 1.99\\ \cline{2-7}
\multicolumn{1}{||l|}{}  & $R_{\boldsymbol{\sigma}, \Delta t}$ & - & 1.45 & 1.46 & 1.36 & 1.17\\ \cline{2-7}
 \multicolumn{1}{||l|}{}  & $E^\infty_h$ & 
2.099e-02 & 1.191e-02 & 4.892e-03 & 1.992e-03 & 9.281e-04\\ \cline{2-7} 
\multicolumn{1}{||l|}{}  & $R^\infty_h$ & - & 1.40 & 1.74 & 1.91 & 1.88\\ \cline{2-7}
\multicolumn{1}{||l|}{}  & $R^\infty_{\Delta t}$ & - & 0.87 & 1.32 & 1.31 & 1.11
\\  \hline \hline
\multicolumn{1}{||l|}{\multirow{6}{*}{$\alpha$=0.99}} 
& $\widetilde{\Delta t}$& \multicolumn{5}{|l||} {\hspace{3.7cm}8.695e-01}
\\ \cline{2-7} 
\multicolumn{1}{||l|}{}  & $\Delta t$
& 2.754e-01 & 1.389e-01 & 6.974e-02 & 3.493e-02 & 1.748e-02\\ \cline{2-7} 
\multicolumn{1}{||l|}{}  & $E_{u,h}$ & 
4.044e-02 & 1.788e-02 & 1.164e-02 & 5.411e-03 & 3.094e-03\\ \cline{2-7} 
\multicolumn{1}{||l|}{}  & $R_{u,h}$ & - & 2.01 & 1.49 & 1.89 & 1.94\\ \cline{2-7}
\multicolumn{1}{||l|}{}  & $R_{u,\Delta t}$ & - & 1.19 & 0.62 & 1.11 & 0.81 \\ \cline{2-7}
\multicolumn{1}{||l|}{}   & $E_{\boldsymbol{\sigma},h}$ & 
4.707e-02 & 1.808e-02 & 1.043e-02 & 4.717e-03 & 2.668e-03\\ \cline{2-7} 
\multicolumn{1}{||l|}{}  & $R_{\boldsymbol{\sigma},h}$ & - & 2.36 & 1.91 & 1.96 & 1.98\\ \cline{2-7}
\multicolumn{1}{||l|}{}  & $R_{\boldsymbol{\sigma}, \Delta t}$ & - & 1.40 & 0.80 & 1.15 & 0.82 \\ \cline{2-7}
 \multicolumn{1}{||l|}{}  & $E^\infty_h$ & 
1.969e-02 & 1.091e-02 & 6.682e-03 & 3.354e-03 & 1.826e-03\\ \cline{2-7} 
\multicolumn{1}{||l|}{}  & $R^\infty_h$ & - & 1.46 & 1.71 & 1.70 & 2.11\\ \cline{2-7}
\multicolumn{1}{||l|}{}  & $R^\infty_{\Delta t}$ & - & 0.86 & 0.71 & 1.00 & 0.88
\\  \hline 
\end{tabular}
\caption{Error $E_{u,h}$, $E_{\boldsymbol{\sigma},h}$, $E^\infty_h$ and rate of convergence $R_{u,h}$, $R_{u,\Delta t}$, $R_{\boldsymbol{\sigma},h}$, $R_{\boldsymbol{\sigma},\Delta t}$, $R^\infty_h$ and $R^\infty_{\Delta t}$ of the proposed method for Example~\ref{VariablePIDEMFEM2_f10}.}\label{2dpide_f10}
\end{table}

Furthermore, Table~\ref{2dpide_f101} demonstrates that if we consider the error estimate in Theorem~\ref{L2H1errortheorem} without the weight $t_n^{\alpha/2}$, the rate of convergence is not optimal. In this table, the $L^2$-norm error $E_{\boldsymbol{\sigma},h}$ for $\boldsymbol{\sigma}$ and the $L^\infty$-norm error $E_{h}^\infty$ for $u$ are computed using the following formula
\begin{align*}
& E_{\boldsymbol{\sigma},h} = \max\limits_{1 \leq n \leq N}\|\boldsymbol{\sigma}_h^n - \widetilde{\boldsymbol{\sigma}}_h(t_n)\|,\;\text{ and}\\
    &    E_{h}^\infty= 
        \max\limits_{\boldsymbol{x}_j\in\mathcal{N},\;1 \leq n \leq N}\;|u_h^n(\boldsymbol{x}_j) - \widetilde{u}_h(\boldsymbol{x}_j,\;t_n)|. 
\end{align*}

\begin{table}[H]
\centering
\begin{tabular}{||l|l|l|l|l|l|l||}
\hline
\multicolumn{2}{||l|}{$N$} & 4 & 8 & 16 & 32 & 64  \\ \hline \hline
\multicolumn{1}{||l|}{\multirow{6}{*}{$\alpha$=0.2}} 
& $\widetilde{\Delta t}$ & \multicolumn{5}{|l||} {\hspace{3.7cm}6.941e-01}
\\ \cline{2-7} 
\multicolumn{1}{||l|}{}  & $\Delta t$
& 9.270e-01 & 7.033e-01 & 4.441e-01 & 2.509e-01 & 1.335e-01
\\ \cline{2-7} 
\multicolumn{1}{||l|}{} & 
 $E_{u,h}$ & 1.789e-02 & 5.412e-03 & 1.392e-03 & 3.781e-04 & 1.146e-04 \\ \cline{2-7}
\multicolumn{1}{||l|}{}  & $R_{u,h}$ & - & 1.72 & 1.96 & 2.00 & 1.98\\ \cline{2-7}
\multicolumn{1}{||l|}{}  & $R_{u,\Delta t}$ & - & 4.33 & 2.96 & 2.28 & 1.89\\ \cline{2-7}
\multicolumn{1}{||l|}{}  & $E_{\boldsymbol{\sigma},h}$ & 4.857e-02 & 1.796e-02 & 4.536e-03 & 1.827e-03 & 4.872e-4  \\ \cline{2-7}
\multicolumn{1}{||l|}{}  & $R_{\boldsymbol{\sigma},h}$ & - & 1.44 & 1.98 & {\color{red}1.40} & 2.20\\ \cline{2-7}
\multicolumn{1}{||l|}{}  & $R_{\boldsymbol{\sigma}, \Delta t}$ & - & {\color{red}3.60} & {\color{red}2.99} & {\color{red}1.59} & 2.10\\ \cline{2-7}
\multicolumn{1}{||l|}{}  & $E^\infty_h$ & 3.190e-02 & 1.379e-02 & 6.018e-03 & 2.390e-04 &  2.426e-04\\ \cline{2-7}
\multicolumn{1}{||l|}{}  & $R^\infty_h$ & - & 1.60 & 1.89 & 1.97 & 1.91\\ \cline{2-7}
\multicolumn{1}{||l|}{}  & $R^\infty_{\Delta t}$ & - & 4.02 & 2.85 & 2.24 & 1.83
\\ \hline \hline
\multicolumn{1}{||l|}{\multirow{6}{*}{$\alpha$=0.5}} 
& $\widetilde{\Delta t}$ & \multicolumn{5}{|l||} {\hspace{3.7cm}9.544e-01}
\\ \cline{2-7} 
\multicolumn{1}{||l|}{}  & $\Delta t$
& 5.901e-01 & 3.389e-01 & 1.813e-01 & 9.373e-02 & 4.764e-02
\\ \cline{2-7} 
\multicolumn{1}{||l|}{} & 
 $E_{u,h}$ & 1.788e-02 & 7.328e-03 & 3.094e-03 & 1.168e-03 & 3.781e-04\\ \cline{2-7} 
\multicolumn{1}{||l|}{}  & $R_{u,h}$ & - & 1.75 & 1.83 & 2.01 & 1.98\\ \cline{2-7}
\multicolumn{1}{||l|}{}  & $R_{u,\Delta t}$ & - & 1.61 & 1.38 & 1.48 & 1.67\\ \cline{2-7}
\multicolumn{1}{||l|}{}   & $E_{\boldsymbol{\sigma},h}$ & 4.857e-2 & 2.163e-2 & 1.016e-2 & 4.53e-3 & 1.827e-3\\ \cline{2-7} 
\multicolumn{1}{||l|}{}  & $R_{\boldsymbol{\sigma},h}$ & - & {\color{red}1.58} & {\color{red}1.61} & {\color{red}1.66} & {\color{red}1.59} \\ \cline{2-7}
\multicolumn{1}{||l|}{}  & $R_{\boldsymbol{\sigma}, \Delta t}$ & - & {\color{red}1.46} & {\color{red}1.21} & {\color{red}1.22} & {\color{red}1.34} \\ \cline{2-7}
 \multicolumn{1}{||l|}{}  & $E^\infty_h$ & 
3.108e-2 & 1.379e-2 & 6.018e-3 & 2.39e-3 & 7.676e-4\\ \cline{2-7} 
\multicolumn{1}{||l|}{}  & $R^\infty_h$ & - & 1.59 & 1.76 & 1.90 & 1.99 \\ \cline{2-7}
\multicolumn{1}{||l|}{}  & $R^\infty_{\Delta t}$ & - & 1.47 & 1.33 & 1.40 & 1.68
\\ \hline \hline
\multicolumn{1}{||l|}{\multirow{6}{*}{$\alpha$=0.8}} 
& $\widetilde{\Delta t}$& \multicolumn{5}{|l||} {\hspace{3.7cm}9.292e-01}
\\ \cline{2-7} 
\multicolumn{1}{||l|}{}  & $\Delta t$
& 3.689e-1 & 1.923e-1 & 9.81e-2 & 4.952e-2 & 2.488e-2\\ \cline{2-7} 
\multicolumn{1}{||l|}{}  & $E_{u,h}$ & 
4.044e-2 & 1.788e-2 & 7.328e-3 & 3.094e-3 & 1.391e-3\\ \cline{2-7} 
\multicolumn{1}{||l|}{}  & $R_{u,h}$ & - & 2.01 & 1.75 & 1.83 & 1.97\\ \cline{2-7}
\multicolumn{1}{||l|}{}  & $R_{u,\Delta t}$ & - & 1.25 & 1.33 & 1.26 & 1.16\\ \cline{2-7}
\multicolumn{1}{||l|}{}   & $E_{\boldsymbol{\sigma},h}$ & 
1.304e-1 & 4.857e-2 & 2.163e-2 & 1.016e-2 & 4.536e-3\\ \cline{2-7} 
\multicolumn{1}{||l|}{}  & $R_{\boldsymbol{\sigma},h}$ & - & 2.44 & {\color{red}1.58} & {\color{red}1.61} & {\color{red}1.99} \\ \cline{2-7}
\multicolumn{1}{||l|}{}  & $R_{\boldsymbol{\sigma}, \Delta t}$ & - & {\color{red}1.52} & {\color{red}1.20} & {\color{red}1.10} & {\color{red}1.17}\\ \cline{2-7}
 \multicolumn{1}{||l|}{}  & $E^\infty_h$ & 
4.542e-2 & 3.108e-2 & 1.379e-2 & 6.018e-3 & 2.759e-3\\ \cline{2-7} 
\multicolumn{1}{||l|}{}  & $R^\infty_h$ & - & 0.94 & 1.59 & 1.76 & 1.92\\ \cline{2-7}
\multicolumn{1}{||l|}{}  & $R^\infty_{\Delta t}$ & - & 0.58 & 1.21 & 1.21 & 1.13
\\  \hline \hline
\multicolumn{1}{||l|}{\multirow{6}{*}{$\alpha$=0.99}} 
& $\widetilde{\Delta t}$& \multicolumn{5}{|l||} {\hspace{3.7cm}8.695e-01}
\\ \cline{2-7} 
\multicolumn{1}{||l|}{}  & $\Delta t$
& 2.754e-1 & 1.389e-1 & 6.974e-2 & 3.493e-2 & 1.748e-2\\ \cline{2-7} 
\multicolumn{1}{||l|}{}  & $E_{u,h}$ & 
4.044e-2 & 1.788e-2 & 1.164e-2 & 5.411e-3 & 3.094e-3\\ \cline{2-7} 
\multicolumn{1}{||l|}{}  & $R_{u,h}$ & - & 2.01 & 1.49 & 1.89 & 1.94\\ \cline{2-7}
\multicolumn{1}{||l|}{}  & $R_{u,\Delta t}$ & - & 1.19 & 0.62 & 1.11 & 0.81 \\ \cline{2-7}
\multicolumn{1}{||l|}{}   & $E_{\boldsymbol{\sigma},h}$ & 
1.304e-1 & 4.857e-2 & 3.944e-2 & 1.795e-2 & 1.016e-2\\ \cline{2-7} 
\multicolumn{1}{||l|}{}  & $R_{\boldsymbol{\sigma},h}$ & - & 2.44 & {\color{red}0.72} & 1.94 & 1.98\\ \cline{2-7}
\multicolumn{1}{||l|}{}  & $R_{\boldsymbol{\sigma}, \Delta t}$ & - & 1.44 & {\color{red}0.30} & 1.14 & 0.82 \\ \cline{2-7}
 \multicolumn{1}{||l|}{}  & $E^\infty_h$ & 
4.542e-2 & 3.108e-2 & 2.005e-2 & 1.024e-2 & 6.018e-3\\ \cline{2-7} 
\multicolumn{1}{||l|}{}  & $R^\infty_h$ & - & 0.94 & 1.52 & 1.66 & 1.85\\ \cline{2-7}
\multicolumn{1}{||l|}{}  & $R^\infty_{\Delta t}$ & - & 0.55 & 0.64 & 0.97 & 0.77
\\  \hline 
\end{tabular}
\caption{Error $E_{u,h}$, $E_{\boldsymbol{\sigma},h}$, $E^\infty_h$ and rate of convergence $R_{u,h}$, $R_{u,\Delta t}$, $R_{\boldsymbol{\sigma},h}$, $R_{\boldsymbol{\sigma},\Delta t}$, $R^\infty_h$ and $R^\infty_{\Delta t}$ of the proposed method for Example~\ref{VariablePIDEMFEM2_f10}.}\label{2dpide_f101}
\end{table}

\end{example}

For the following three examples: Examples \ref{VariableceffPDE_newmesh}-\ref{VariablePIDEMFEM2_newmesh},
however, step length restriction is not satisfied, but we still obtain the computed orders of convergence which confirm 
our theoretical results and therefore, this behaviour needs further investigation.

\begin{example}\label{VariableceffPDE_newmesh} Now consider a time-fractional PDE (\ref{pide}) over the domain $\Omega\times(0,T] = (-1,1)^2\times (0,1]$ with  coefficients $\textbf{A}(\boldsymbol{x},t)=\begin{bmatrix} 2-\cos(t) & x_1x_2 \\ x_1x_2 & 2-\sin(t) \\ \end{bmatrix},\; \textbf{b}(\boldsymbol{x},t) =\begin{bmatrix} 1 + 2x_1x_2\\ 1 + x_1x_2\\ \end{bmatrix}$, $c(\boldsymbol{x},t) =1 -\sin(t)$, and $\lambda =0$. The initial condition $u_0$ and the source term $f$ are chosen to align with the exact solution $ u(\boldsymbol{x},t)= \sin(\pi x_1)\sin(\pi x_2)(t^{\alpha}+t^3).$ For this example, the computational convergence rate is provided in Table~\ref{2DPDE1_22}  in the temporal direction. The computed order of convergence confirms the theoretical order of convergence.

\begin{table}[H]
\centering
\begin{tabular}{||l|l|l|l|l|l|l||}
\hline
\multicolumn{2}{||l|}{$N$} & 4 & 8 & 16 & 32 & 64  \\ \hline \hline
\multicolumn{1}{||l|}{\multirow{6}{*}{$\alpha$=0.2}}
& $\widetilde{\Delta t}$& \multicolumn{5}{|l||}  {\hspace{3.7cm}1.830e-08}
\\ \cline{2-7} 
\multicolumn{1}{||l|}{}  &
$\Delta t$
& 9.270e-01 & 7.033e-01 & 4.441e-01 & 2.509e-01 & 1.335e-01
\\ \cline{2-7} 
\multicolumn{1}{||l|}{}  & $E_{u,h}$ & 5.036e-01 & 1.372e-01 & 3.517e-02 & 9.661e-03 & 2.910e-03\\ \cline{2-7} 
\multicolumn{1}{||l|}{}  & $R_{u,h}$ & - & 1.88 & 1.96 & 1.99 & 1.99\\ \cline{2-7}
\multicolumn{1}{||l|}{}  & $R_{u,\Delta t}$ & - & 4.71 & 2.96 & 2.26 & 1.90\\ \cline{2-7}
 & $E_{\boldsymbol{\sigma},h}$ & 
1.891e+0 & 5.049e-01 & 1.288e-01 & 3.56e-02 & 1.079e-02\\ \cline{2-7} 
\multicolumn{1}{||l|}{}  & $R_{\boldsymbol{\sigma},h}$ & - & 1.91 & 1.97 & 1.98 & 1.98\\ \cline{2-7}
\multicolumn{1}{||l|}{}  & $R_{\boldsymbol{\sigma}, \Delta t}$ & - & 4.78 & 2.97 & 2.25 & 1.89\\ \cline{2-7}
 & $E^\infty_h$ & 
1.247e+0 & 3.991e-01 & 1.044e-01 & 2.874e-02 & 8.654e-03\\ \cline{2-7} 
\multicolumn{1}{||l|}{}  & $R^\infty_h$ & - & 1.64 & 1.93 & 1.98 & 1.99\\ \cline{2-7}
\multicolumn{1}{||l|}{}  & $R^\infty_{\Delta t}$ & - & 4.13 & 2.92 & 2.26 & 1.90
\\ \hline \hline
\multicolumn{1}{||l|}{\multirow{6}{*}{$\alpha$=0.5}}    & $\widetilde{\Delta t}$&    \multicolumn{5}{|l||} {\hspace{3.7cm}8.060e-04}
\\ \cline{2-7} 
\multicolumn{1}{||l|}{}  & $\Delta t$
& 5.901e-01 & 3.389e-01 & 1.813e-01 & 9.373e-02 & 4.764e-02
\\ \cline{2-7} 
\multicolumn{1}{||l|}{}  & $E_{u,h}$ & 5.030e-01 & 1.949e-01 & 7.826e-02 & 2.997e-02 & 9.623e-03\\ \cline{2-7} 
\multicolumn{1}{||l|}{}  & $R_{u,h}$ & - & 1.86 & 1.94 & 1.98 & 1.99\\ \cline{2-7}
\multicolumn{1}{||l|}{}  & $R_{u,\Delta t}$ & - & 1.71 & 1.46 & 1.45 & 1.68\\ \cline{2-7}
 & $E_{\boldsymbol{\sigma},h}$ & 
1.885e+0 & 7.205e-01 & 2.867e-01 & 1.095e-01 & 3.521e-02\\ \cline{2-7} 
\multicolumn{1}{||l|}{}  & $R_{\boldsymbol{\sigma},h}$ & - & 1.88 & 1.96 & 1.98 & 1.99\\ \cline{2-7}
\multicolumn{1}{||l|}{}  & $R_{\boldsymbol{\sigma}, \Delta t}$ & - & 1.74 & 1.47 & 1.46 & 1.68\\ \cline{2-7}
 & $E^\infty_h$ & 
1.256e+0 & 5.417e-01 & 2.193e-01 & 9.087e-02 & 2.822e-02\\ \cline{2-7} 
\multicolumn{1}{||l|}{}  & $R^\infty_h$ & - & 1.65 & 1.92 & 1.82 & 2.05\\ \cline{2-7}
\multicolumn{1}{||l|}{}  & $R^\infty_{\Delta t}$ & - & 1.52 & 1.45 & 1.34 & 1.73
\\ \hline \hline
\multicolumn{1}{||l|}{\multirow{6}{*}{$\alpha$=0.8}}          & $\widetilde{\Delta t}$& 
 \multicolumn{5}{|l||} {\hspace{3.7cm}1.163e-02}
\\ \cline{2-7} 
\multicolumn{1}{||l|}{}  & $\Delta t$
& 3.689e-01 & 1.923e-01 & 9.810e-02 & 4.952e-02 & 2.488e-02\\ \cline{2-7} 
\multicolumn{1}{||l|}{}  & $E_{u,h}$ &
1.245e+0 & 5.029e-01 & 1.949e-01 & 7.825e-02 & 3.512e-02\\ \cline{2-7} 
\multicolumn{1}{||l|}{}  & $R_{u,h}$ & - & 2.24 & 1.86 & 1.94 & 1.98\\ \cline{2-7}
\multicolumn{1}{||l|}{}  & $R_{u,\Delta t}$ & - & 1.39 & 1.41 & 1.34 & 1.16\\ \cline{2-7}
 & $E_{\boldsymbol{\sigma},h}$ & 
4.740e+0 & 1.886e+0 & 7.206e-01 & 2.865e-01 & 1.283e-01\\ \cline{2-7} 
\multicolumn{1}{||l|}{}  & $R_{\boldsymbol{\sigma},h}$ & - & 2.27 & 1.88 & 1.96 & 1.98\\ \cline{2-7}
\multicolumn{1}{||l|}{}  & $R_{\boldsymbol{\sigma}, \Delta t}$ & - & 1.42 & 1.43 & 1.35 & 1.17\\ \cline{2-7}
 & $E^\infty_h$ & 
1.740e+0 & 1.251e+0 & 5.393e-01 & 2.187e-01 & 1.031e-01\\ \cline{2-7} 
\multicolumn{1}{||l|}{}  & $R^\infty_h$ & - & 0.81 & 1.65 & 1.92 & 1.85\\ \cline{2-7}
\multicolumn{1}{||l|}{}  & $R^\infty_{\Delta t}$ & - & 0.51 & 1.25 & 1.32 & 1.09
\\  \hline
\hline\multicolumn{1}{||l|}{\multirow{6}{*}{$\alpha$=0.99}}         & $\widetilde{\Delta t}$& \multicolumn{5}{|l||} {\hspace{3.7cm}2.723e-02}
\\ \cline{2-7} 
\multicolumn{1}{||l|}{}  & $\Delta t$
& 2.754e-01 & 1.389e-01 & 6.974e-02 & 3.493e-02 & 1.748e-02\\ \cline{2-7} 
 & $E_{u,h}$ & 
1.243e+0 & 5.025e-01 & 2.973e-01 & 1.372e-01 & 7.825e-02\\ \cline{2-7} 
\multicolumn{1}{||l|}{}  & $R_{u,h}$ & - & 2.23 & 1.82 & 1.91 & 1.95\\ \cline{2-7}
\multicolumn{1}{||l|}{}  & $R_{u,\Delta t}$ & - & 1.32 & 0.76 & 1.12 & 0.81 \\ \cline{2-7}
 & $E_{\boldsymbol{\sigma},h}$ & 
4.729e+0 & 1.882e+0 & 1.107e+0 & 5.047e-01 & 2.865e-01\\ \cline{2-7} 
\multicolumn{1}{||l|}{}  & $R_{\boldsymbol{\sigma},h}$ & - & 2.27 & 1.85 & 1.94 & 1.97\\ \cline{2-7}
\multicolumn{1}{||l|}{}  & $R_{\boldsymbol{\sigma}, \Delta t}$ & - & 1.35 & 0.77 & 1.14 & 0.82 \\ \cline{2-7}
 & $E^\infty_h$ & 
1.734e+0 & 1.255e+0 & 7.276e-01 & 3.982e-01 & 2.185e-01\\ \cline{2-7} 
\multicolumn{1}{||l|}{}  & $R^\infty_h$ & - & 0.80 & 1.90 & 1.49 & 2.09\\ \cline{2-7}
\multicolumn{1}{||l|}{}  & $R^\infty_{\Delta t}$ & - & 0.47 & 0.79 & 0.87 & 0.87
\\  \hline 
\end{tabular}
\caption{Error $E_{u,h}$, $E_{\boldsymbol{\sigma},h}$, $E^\infty_h$ and rate of convergence $R_{u,h}$, $R_{u,\Delta t}$, $R_{\boldsymbol{\sigma},h}$, $R_{\boldsymbol{\sigma},\Delta t}$, $R^\infty_h$ and $R^\infty_{\Delta t}$ of the proposed method for Example~\ref{VariableceffPDE_newmesh}.}\label{2DPDE1_22}
\end{table}

\end{example}

\begin{example}\label{VariableceffPDEMFEM2} In this example, we examine a two-dimensional time-fractional PDE (\ref{pide}) over the domain $\Omega\times(0, T] = (-1,1)^2\times (0,1]$ with $\textbf{A}(\boldsymbol{x},t)=\begin{bmatrix} 
1 & \frac{x_1 x_2 t}{8} \\
 \frac{x_1 x_2 t}{8} & 1 \\
\end{bmatrix},\;
\textbf{b}(\boldsymbol{x},t) =\begin{bmatrix} 
 x_1^2 t\\
 2x_2^2 t\\
\end{bmatrix},\; c(\boldsymbol{x},t) = x_1x_2 t, \;\lambda = 0$ $\forall \boldsymbol{x}=(x_1,\;x_2)\in\Omega,\;t\in(0,1]$. Let $u_0(\boldsymbol{x})= x_1(1-|x_1|)x_2(1-|x_2|)\in \dot{H}^2(\Omega)$ be the initial condition and $f(\boldsymbol{x}, t)=e^{-t}\sin(\pi x_1)\sin(\pi x_2)$ the source term. The numerical results in Table~\ref{2dpdenew22_f21} are consistent with our theoretical findings.

\begin{table}[H]
\centering
\begin{tabular}{||l|l|l|l|l|l|l||}
\hline
\multicolumn{2}{||l|}{$N$} & 4 & 8 & 16 & 32 & 64 \\ \hline \hline
\multicolumn{1}{||l|}{\multirow{6}{*}{$\alpha$=0.2}} 
& $\widetilde{\Delta t}$& \multicolumn{5}{|l||} {\hspace{3.7cm}3.499e-04}
\\ \cline{2-7} 
\multicolumn{1}{||l|}{}  & $\Delta t$
& 9.270e-01 & 7.033e-01 & 4.441e-01 & 2.509e-01 & 1.335e-01
\\ \cline{2-7} 
\multicolumn{1}{||l|}{}  & $E_{u,h}$ & 1.788e-02 & 5.411e-03 & 1.391e-03 & 3.78e-04 & 1.145e-04\\ \cline{2-7} 
\multicolumn{1}{||l|}{}  & $R_{u,h}$ & - & 1.72 & 1.96 & 2.00 & 1.98\\ \cline{2-7}
\multicolumn{1}{||l|}{}  & $R_{u,\Delta t}$ & - & 4.33 & 2.96 & 2.28 & 1.89\\ \cline{2-7}
\multicolumn{1}{||l|}{}   & $E_{\boldsymbol{\sigma},h}$ & 
2.783e-02 & 7.352e-03 & 1.854e-03 & 5.088e-04 & 1.528e-04\\ \cline{2-7} 
\multicolumn{1}{||l|}{}  & $R_{\boldsymbol{\sigma},h}$ & - & 1.92 & 1.99 & 1.99 & 2.00 \\ \cline{2-7}
\multicolumn{1}{||l|}{}  & $R_{\boldsymbol{\sigma}, \Delta t}$ & - & 4.82 & 3.00 & 2.26 & 1.91\\ \cline{2-7}
 \multicolumn{1}{||l|}{}  & $E^\infty_h$ & 
1.838e-02 & 5.705e-03 & 1.451e-03 & 3.913e-04 & 1.214e-04\\ \cline{2-7} 
\multicolumn{1}{||l|}{}  & $R^\infty_h$ & - & 1.69 & 1.97 & 2.01 & 1.94\\ \cline{2-7}
\multicolumn{1}{||l|}{}  & $R^\infty_{\Delta t}$ & - & 4.24 & 2.98 & 2.30 & 1.85
\\ \hline \hline
\multicolumn{1}{||l|}{\multirow{6}{*}{$\alpha$=0.5}} 
&  $\widetilde{\Delta t}$&  \multicolumn{5}{|l||} {\hspace{3.7cm}4.579e-02}
\\ \cline{2-7}  
\multicolumn{1}{||l|}{}  & $\Delta t$
& 5.901e-01 & 3.389e-01 & 1.813e-01 & 9.373e-02 & 4.764e-02
\\ \cline{2-7} 
\multicolumn{1}{||l|}{}  & $E_{u,h}$ & 
1.788e-02 & 7.328e-03 & 3.094e-03 & 1.168e-03 & 3.78e-04\\ \cline{2-7} 
\multicolumn{1}{||l|}{}  & $R_{u,h}$ & - & 1.75 & 1.83 & 2.01 & 1.98\\ \cline{2-7}
\multicolumn{1}{||l|}{}  & $R_{u,\Delta t}$ & - & 1.61 & 1.38 & 1.48 & 1.67\\ \cline{2-7}
\multicolumn{1}{||l|}{}   & $E_{\boldsymbol{\sigma},h}$ & 
2.152e-02 & 8.267e-03 & 3.281e-03 & 1.25e-03 & 4.006e-04\\ \cline{2-7} 
\multicolumn{1}{||l|}{}  & $R_{\boldsymbol{\sigma},h}$ & - & 1.87 & 1.97 & 1.99 & 1.99\\ \cline{2-7}
\multicolumn{1}{||l|}{}  & $R_{\boldsymbol{\sigma}, \Delta t}$ & - & 1.73 & 1.48 & 1.46 & 1.68\\ \cline{2-7}
 \multicolumn{1}{||l|}{}  & $E^\infty_h$ & 
1.455e-02 & 6.21e-03 & 2.475e-03 & 9.915e-04 & 3.053e-04\\ \cline{2-7} 
\multicolumn{1}{||l|}{}  & $R^\infty_h$ & - & 1.67 & 1.96 & 1.88 & 2.06\\ \cline{2-7}
\multicolumn{1}{||l|}{}  & $R^\infty_{\Delta t}$ & - & 1.54 & 1.47 & 1.39 & 1.74
\\ \hline \hline
\multicolumn{1}{||l|}{\multirow{6}{*}{$\alpha$=0.8}} 
& $\widetilde{\Delta t}$& \multicolumn{5}{|l||} {\hspace{3.7cm}1.392e-01}
\\ \cline{2-7} 
\multicolumn{1}{||l|}{}  & $\Delta t$
& 3.689e-01 & 1.923e-01 & 9.810e-02 & 4.952e-02 & 2.488e-02\\ \cline{2-7} 
\multicolumn{1}{||l|}{}  & $E_{u,h}$ & 
4.044e-02 & 1.788e-02 & 7.328e-03 & 3.094e-03 & 1.391e-03\\ \cline{2-7} 
\multicolumn{1}{||l|}{}  & $R_{u,h}$ & - & 2.01 & 1.75 & 1.83 & 1.97\\ \cline{2-7}
\multicolumn{1}{||l|}{}  & $R_{u,\Delta t}$ & - & 1.25 & 1.33 & 1.26 & 1.16\\ \cline{2-7}
\multicolumn{1}{||l|}{}   & $E_{\boldsymbol{\sigma},h}$ & 
5.148e-02 & 1.884e-02 & 7.118e-03 & 2.826e-03 & 1.263e-03\\ \cline{2-7} 
\multicolumn{1}{||l|}{}  & $R_{\boldsymbol{\sigma},h}$ & - & 2.48 & 1.91 & 1.97 & 1.98\\ \cline{2-7}
\multicolumn{1}{||l|}{}  & $R_{\boldsymbol{\sigma}, \Delta t}$ & - & 1.54 & 1.45 & 1.35 & 1.17\\ \cline{2-7}
 \multicolumn{1}{||l|}{}  & $E^\infty_h$ & 
2.203e-02 & 1.274e-02 & 5.346e-03 & 2.134e-03 & 9.888e-04\\ \cline{2-7} 
\multicolumn{1}{||l|}{}  & $R^\infty_h$ & - & 1.35 & 1.70 & 1.95 & 1.90 \\ \cline{2-7}
\multicolumn{1}{||l|}{}  & $R^\infty_{\Delta t}$ & - & 0.84 & 1.29 & 1.34 & 1.12
\\  \hline \hline
\multicolumn{1}{||l|}{\multirow{6}{*}{$\alpha$=0.99}} 
 & $\widetilde{\Delta t}$&  \multicolumn{5}{|l||} {\hspace{3.7cm}1.876e-01}
\\ \cline{2-7} 
\multicolumn{1}{||l|}{}  & $\Delta t$
& 4.044e-02 & 1.788e-02 & 1.164e-02 & 5.411e-03 & 3.094e-03\\ \cline{2-7} 
\multicolumn{1}{||l|}{}  & $R_{u,h}$ & - & 2.01 & 1.49 & 1.89 & 1.94\\ \cline{2-7}
\multicolumn{1}{||l|}{}  & $R_{u,\Delta t}$ & - & 1.19 & 0.62 & 1.11 & 0.81 \\ \cline{2-7}
\multicolumn{1}{||l|}{}   & $E_{\boldsymbol{\sigma},h}$ & 
4.811e-02 & 1.741e-02 & 1.016e-02 & 4.628e-03 & 2.627e-03\\ \cline{2-7} 
\multicolumn{1}{||l|}{}  & $R_{\boldsymbol{\sigma},h}$ & - & 2.51 & 1.87 & 1.94 & 1.97\\ \cline{2-7}
\multicolumn{1}{||l|}{}  & $R_{\boldsymbol{\sigma}, \Delta t}$ & - & 1.48 & 0.78 & 1.14 & 0.82 \\ \cline{2-7}
 \multicolumn{1}{||l|}{}  & $E^\infty_h$ & 
2.083e-02 & 1.183e-02 & 7.294e-03 & 3.579e-03 & 1.977e-03\\ \cline{2-7} 
\multicolumn{1}{||l|}{}  & $R^\infty_h$ & - & 1.39 & 1.68 & 1.76 & 2.06\\ \cline{2-7}
\multicolumn{1}{||l|}{}  & $R^\infty_{\Delta t}$ & - & 0.83 & 0.70 & 1.03 & 0.86
\\  \hline 
\end{tabular}
\caption{Error $E_{u,h}$, $E_{\boldsymbol{\sigma},h}$, $E^\infty_h$ and rate of convergence $R_{u,h}$, $R_{u,\Delta t}$, $R_{\boldsymbol{\sigma},h}$, $R_{\boldsymbol{\sigma},\Delta t}$, $R^\infty_h$ and $R^\infty_{\Delta t}$ of the proposed method for
Example~\ref{VariableceffPDEMFEM2}.}\label{2dpdenew22_f21}
\end{table}

\end{example}

\begin{example}\label{VariablePIDEMFEM2_newmesh} [\cite{TOMAR2024137}, Example~6.5]
Now, we consider the time-fractional PIDE (\ref{pide}) with variable coefficients over the domain $\Omega\times(0,T] = (-1,1)^2\times (0,1]$ with the same initial condition as in Example~\ref{VariableceffPDEMFEM2} and $\textbf{A}(\boldsymbol{x},t)=\begin{bmatrix} 
1 & \frac{x_1 x_2 t}{8} \\
\frac{x_1 x_2 t}{8} & 1
\end{bmatrix},$ $
\textbf{b}(\boldsymbol{x},t) =\begin{bmatrix} 
 x_1^2 t\\
 x_2^2 t\\
\end{bmatrix},\; c(\boldsymbol{x},t) = x_1x_2 t, \;\lambda = \frac{1}{2}\;$, $g(\boldsymbol{x},\boldsymbol{y}) = \frac{1}{2}e^{-\|\boldsymbol{x}-\boldsymbol{y}\|^2}$, and the source term $f(\boldsymbol{x}, t)=e^{-t}\sin(\pi x_1)\sin(\pi x_2)$.
Table~\ref{2dpide_f222} demonstrates that the computational rate is second-order for both variables $u$ and $\boldsymbol{\sigma}$, consistent with our theoretical findings. 

\begin{table}[H]
\centering
\begin{tabular}{||l|l|l|l|l|l|l||}
\hline
\multicolumn{2}{||l|}{$N$} & 4 & 8 & 16 & 32 & 64  \\ \hline \hline
\multicolumn{1}{||l|}{\multirow{6}{*}{$\alpha$=0.2}} 
& $\widetilde{\Delta t}$&  \multicolumn{5}{|l||} {\hspace{3.7cm}2.669e-03}
\\ \cline{2-7} 
\multicolumn{1}{||l|}{}  & $\Delta t$
& 9.270e-01 & 7.033e-01 & 4.441e-01 & 2.509e-01 & 1.335e-01
\\ \cline{2-7} 
\multicolumn{1}{||l|}{} & $E_{u,h}$ & 1.789e-02
& 5.412e-03
& 1.392e-03
& 3.781e-04
& 1.151e-04
  \\ \cline{2-7} 
\multicolumn{1}{||l|}{}                  & $R_{u,h}$ &  -
& 1.72
& 1.96
& 2.00
& 1.98
\\ \cline{2-7} 
\multicolumn{1}{||l|}{}                  & $R_{u,\Delta t}$ &  -
& 4.33
& 2.96
& 2.28
& 1.89
\\ \cline{2-7} 
\multicolumn{1}{||l|}{}                  & $E_{\boldsymbol{\sigma},h}$ & 2.791e-02
& 7.373e-03
& 1.859e-03
& 5.091e-04
& 1.531e-04
\\ \cline{2-7} 
\multicolumn{1}{||l|}{}                  & $R_{\boldsymbol{\sigma},h}$ & - 
& 1.92
& 1.99
& 1.99
& 2.00
\\  \cline{2-7} 
\multicolumn{1}{||l|}{}                  & $R_{\boldsymbol{\sigma}, \Delta t}$ & - 
& 4.82
& 3.00
& 2.27
& 1.91
\\  \cline{2-7}
\multicolumn{1}{||l|}{}                  & $E^\infty_h$ & 1.835e-02
& 5.712e-03
& 1.449e-03
& 3.911e-04
& 1.221e-04
\\ \cline{2-7} 
\multicolumn{1}{||l|}{}                  & $R^\infty_h$ &  - & 1.68
& 1.98
& 2.01
& 1.94
\\ \cline{2-7} 
\multicolumn{1}{||l|}{}                  & $R^\infty_{\Delta t}$ &  - & 4.23
& 2.98
& 2.30
& 1.85
\\ \hline \hline
\multicolumn{1}{||l|}{\multirow{6}{*}{$\alpha$=0.5}} 
& $\widetilde{\Delta t}$&  \multicolumn{5}{|l||} {\hspace{3.7cm}1.032e-01}
\\ \cline{2-7} 
\multicolumn{1}{||l|}{}  & $\Delta t$
& 5.901e-01 & 3.389e-01 & 1.813e-01 & 9.373e-02 & 4.764e-02\\ \cline{2-7} 
\multicolumn{1}{||l|}{}  & $E_{u,h}$ & 
1.788e-02 & 7.328e-03 & 3.094e-03 & 1.168e-03 & 3.78e-04\\ \cline{2-7} 
\multicolumn{1}{||l|}{}  & $R_{u,h}$ & - & 1.75 & 1.83 & 2.01 & 1.98\\ \cline{2-7}
\multicolumn{1}{||l|}{}  & $R_{u,\Delta t}$ & - & 1.61 & 1.38 & 1.48 & 1.67\\ \cline{2-7}
\multicolumn{1}{||l|}{}   & $E_{\boldsymbol{\sigma},h}$ & 
2.143e-02 & 8.261e-03 & 3.271e-03 & 1.246e-03 & 3.996e-04\\ \cline{2-7} 
\multicolumn{1}{||l|}{}  & $R_{\boldsymbol{\sigma},h}$ & - & 1.87 & 1.97 & 1.99 & 1.99\\ \cline{2-7}
\multicolumn{1}{||l|}{}  & $R_{\boldsymbol{\sigma}, \Delta t}$ & - & 1.72 & 1.48 & 1.46 & 1.68\\ \cline{2-7}
 \multicolumn{1}{||l|}{}  & $E^\infty_h$ & 
1.426e-02 & 6.188e-03 & 2.46e-03 & 9.886e-04 & 3.04e-04\\ \cline{2-7} 
\multicolumn{1}{||l|}{}  & $R^\infty_h$ & - & 1.64 & 1.96 & 1.88 & 2.07\\ \cline{2-7}
\multicolumn{1}{||l|}{}  & $R^\infty_{\Delta t}$ & - & 1.51 & 1.47 & 1.38 & 1.74
\\ \hline \hline
\multicolumn{1}{||l|}{\multirow{6}{*}{$\alpha$=0.8}} 
& $\widetilde{\Delta t}$ & \multicolumn{5}{|l||} {\hspace{3.7cm}2.314e-01}
\\ \cline{2-7} 
\multicolumn{1}{||l|}{}  & $\Delta t$
& 3.689e-01 & 1.923e-01 & 9.810e-02 & 4.952e-02 & 2.488e-02\\ \cline{2-7} 
\multicolumn{1}{||l|}{}  & $E_{u,h}$ & 
4.044e-02 & 1.788e-02 & 7.328e-03 & 3.094e-03 & 1.391e-03\\ \cline{2-7} 
\multicolumn{1}{||l|}{}  & $R_{u,h}$ & - & 2.01 & 1.75 & 1.83 & 1.97\\ \cline{2-7}
\multicolumn{1}{||l|}{}  & $R_{u,\Delta t}$ & - & 1.25 & 1.33 & 1.26 & 1.16\\ \cline{2-7}
\multicolumn{1}{||l|}{}   & $E_{\boldsymbol{\sigma},h}$ & 
5.135e-02 & 1.873e-02 & 7.059e-03 & 2.796e-03 & 1.249e-03\\ \cline{2-7} 
\multicolumn{1}{||l|}{}  & $R_{\boldsymbol{\sigma},h}$ & - & 2.49 & 1.91 & 1.97 & 1.99\\ \cline{2-7}
\multicolumn{1}{||l|}{}  & $R_{\boldsymbol{\sigma}, \Delta t}$ & - & 1.55 & 1.45 & 1.35 & 1.17\\ \cline{2-7}
 \multicolumn{1}{||l|}{}  & $E^\infty_h$ & 
2.218e-02 & 1.246e-02 & 5.316e-03 & 2.115e-03 & 9.857e-04\\ \cline{2-7} 
\multicolumn{1}{||l|}{}  & $R^\infty_h$ & - & 1.42 & 1.67 & 1.96 & 1.88\\ \cline{2-7}
\multicolumn{1}{||l|}{}  & $R^\infty_{\Delta t}$ & - & 0.89 & 1.27 & 1.35 & 1.11
\\  \hline \hline
\multicolumn{1}{||l|}{\multirow{6}{*}{$\alpha$=0.99}} 
& $\widetilde{\Delta t}$ & \multicolumn{5}{|l||} {\hspace{3.7cm}2.828e-01}
\\ \cline{2-7} 
\multicolumn{1}{||l|}{}  & $\Delta t$
& 2.754e-01 & 1.389e-01 & 6.974e-02 & 3.493e-02 & 1.748e-02\\ \cline{2-7} 
\multicolumn{1}{||l|}{}  & $E_{u,h}$ & 
4.044e-02 & 1.788e-02 & 1.164e-02 & 5.411e-03 & 3.094e-03\\ \cline{2-7} 
\multicolumn{1}{||l|}{}  & $R_{u,h}$ & - & 2.01 & 1.49 & 1.89 & 1.94\\ \cline{2-7}
\multicolumn{1}{||l|}{}  & $R_{u,\Delta t}$ & - & 1.19 & 0.62 & 1.11 & 0.81\\ \cline{2-7}
\multicolumn{1}{||l|}{}   & $E_{\boldsymbol{\sigma},h}$ & 
4.789e-02 & 1.73e-02 & 1.005e-02 & 4.556e-03 & 2.581e-03\\ \cline{2-7} 
\multicolumn{1}{||l|}{}  & $R_{\boldsymbol{\sigma},h}$ & - & 2.51 & 1.89 & 1.95 & 1.98\\ \cline{2-7}
\multicolumn{1}{||l|}{}  & $R_{\boldsymbol{\sigma}, \Delta t}$ & - & 1.49 & 0.79 & 1.14 & 0.82 \\ \cline{2-7}
 \multicolumn{1}{||l|}{}  & $E^\infty_h$ & 
2.092e-02 & 1.154e-02 & 6.971e-03 & 3.571e-03 & 1.956e-03\\ \cline{2-7} 
\multicolumn{1}{||l|}{}  & $R^\infty_h$ & - & 1.47 & 1.75 & 1.65 & 2.09\\ \cline{2-7}
\multicolumn{1}{||l|}{}  & $R^\infty_{\Delta t}$ & - & 0.87 & 0.73 & 0.97 & 0.87
\\  \hline 
\end{tabular}
\caption{Error $E_{u,h}$, $E_{\boldsymbol{\sigma},h}$, $E^\infty_h$ and rate of convergence $R_{u,h}$, $R_{u,\Delta t}$, $R_{\boldsymbol{\sigma},h}$, $R_{\boldsymbol{\sigma},\Delta t}$, $R^\infty_h$ and $R^\infty_{\Delta t}$ of the proposed method for Example~\ref{VariablePIDEMFEM2_newmesh}.}\label{2dpide_f222}
\end{table}

\end{example}

For the following two examples, although  initial data is in $H^1_0,$  the step length restriction is satisfied.

\begin{example}\label{GTFBS}{(The generalized (time-fractional) Black-Scholes model. \cite{MR1743405, MR4749244, sawangtong2018analytical})}\label{gbsm} Consider the pricing of European style basket put option $P$ under the following generalized (time-fractional) Black-Scholes model
\begin{align}\label{gbsmeq1}
   &D^\alpha_\tau P + \frac{1}{2}\sigma_1^2 S_1^2 \frac{\partial^2 P}{\partial S_1^2} + \rho\sigma_1 \sigma_2 S_1 S_2 \frac{\partial^2 P}{\partial S_1\partial S_2} + \frac{1}{2}\sigma_2^2 S_2^2 \frac{\partial^2 P}{\partial S_2^2}  +rS_1\frac{\partial P}{\partial S_1}+rS_2\frac{\partial P}{\partial S_2} - rP=0,\; (S_1,S_2,\tau)\in\mathbb{R}_{+}^2\times[0,T),\\
  \label{gbsmeq2} & P(S_1, S_2,T)=\max{\left( K-\frac{S_1 + S_2}{2},0\right)}, \; (S_1,S_2)\in\mathbb{R}_{+}^2,
\end{align}
where $D^\alpha_\tau P,~0 <\alpha\leq1$, represents the modified right Riemann-Liouville derivative, defined by
\begin{align*}
&D^\alpha_\tau P:=\begin{cases}
\displaystyle \frac{1}{\Gamma(1-\alpha)}\frac{d}{d\tau}\int_{\tau}^
{T}\frac{P(\cdot,\eta) - P(\cdot,T)}{(\eta-\tau)^{\alpha}}d\eta &~:~0<\alpha<1,\\[8pt]
\displaystyle\frac{\partial P}{\partial \tau}&~:~\alpha = 1.
\end{cases}
\end{align*}
$\rho$ denotes the correlation between $i^{th}$ and $j^{th}$ underlying stock prices $S_i$ and $S_j$;  $\sigma_i$ represents the volatility of the asset $S_i$; and
$r$ is the risk-free rate of return. Note that the above generalized Black-Scholes model reduces to the standard Black-Scholes model when $\alpha\to 1$. By using the transformations $t:=T-\tau$, $S_i:= Ke^{x_i},\;i=1,2,$ and $V(x_1,x_2,t):= \frac{1}{K} P(S_1, S_2, \tau)$, the above degenerate final value problem (\ref{gbsmeq1}--\ref{gbsmeq2}) can be transformed to the following non-degenerate initial value problem
\begin{align}\label{transformedgbsm1}
   \begin{cases}      
   & \frac{\partial^{\alpha} V}{\partial t^{\alpha}} - \mathcal{L}V =0 \quad \text{in}\quad \mathbb{R}^2\times J,\\
   &V(x_1,x_2,0) =\max{\left( 1-\frac{1}{2}(e^{x_1} + e^{x_2}),0\right)},\; (x_1,x_2)\in\mathbb{R}^2,
   \end{cases}
\end{align}
where $\frac{\partial^{\alpha} V}{\partial t^{\alpha}}$ is the Caputo fractional derivative of order $\alpha$ and 
\begin{align}
   &\mathcal{L}V :=~ \frac{1}{2}\sigma_1^2  \frac{\partial^2 V}{\partial x_1^2} + \rho\sigma_1 \sigma_2  \frac{\partial^2 V}{\partial x_1\partial x_2} + \frac{1}{2}\sigma_2^2 \frac{\partial^2 V}{\partial x_2^2}  +\left(r-\frac{1}{2}\sigma_1^2\right)\frac{\partial V}{\partial x_1}+\left(r-\frac{1}{2}\sigma_2^2\right)\frac{\partial V}{\partial x_2} - rV.\nonumber
\end{align}
For computational purposes, the problem (\ref{transformedgbsm1}) is localized on $\Omega:=(-L,L)^2,\;L>0,$ as follows
\begin{align*}
    & \frac{\partial^{\alpha} u}{\partial t^{\alpha}} - \mathcal{L}u =0 \; \text{ in }\Omega\times(0,T],\\
    &u(x_1,x_2,0) = \max{(1-\frac{1}{2}(e^{x_1}+e^{x_2}),0)},\;\forall (x_1,x_2)\in\Omega,\\
    &u(x_1,x_2,t) =  \frac{1}{2}\Big(1-0.5(e^{x_1} + e^{x_2}) + \Big((\varepsilon t)^2+ \big(1-0.5(e^{x_1} + e^{x_2})\big)^2 \Big)^{1/2}\Big),\;\forall (x_1,x_2,t)\in\partial\Omega\times [0,T], 
\end{align*}
where $u(x_1, x_2,t)\approx V(x_2,x_2,t)$. The model parameters are as follows:
\begin{align*}
 & T=1,\; \sigma_1 = 0.2 = \sigma_2,\; \rho= 0.5,\; r=0.06,\;L=1,\;\varepsilon=1.
\end{align*}

In this example, as well as in the following example \ref{GTFM}, the initial data is in $H^1_0(\Omega),$ but not in $H_0^1(\Omega)\cap H^2(\Omega)$ and therefore, the solution is expected to more singularity at $t=0.$ As a result, it is expected to increase the power of $t$ appeared in the error estimate.  In Table~\ref{2dBS22_f22}, the error $E_{u,h}$, $E_{\boldsymbol{\sigma},h}$, $E^\infty_h$ has been calculated by using the formula
\begin{align} 
        \begin{cases}
       E_{u,h} = \max\limits_{1 \leq n \leq N}t\|u_{h}^n - \widetilde{u}_h(t_n)\|, \\
        E_{\boldsymbol{\sigma},h} = \max\limits_{1 \leq n \leq N}t^{1+\alpha/2}\|\boldsymbol{\sigma}_{h}^n - \widetilde{\boldsymbol{\sigma}}_h(t_n)\|,\;\text{ and} \\
        E_{u,h}^\infty = \max\limits_{\boldsymbol{x}_j\in\mathcal{N},\;1 \leq n \leq N}\;t_n^{1+\frac{\alpha}{2}}|u_h^n(\boldsymbol{x}_j) - \widetilde{u}_h(\boldsymbol{x}_j,\;t_n)|.
        \end{cases}\label{modcomputerr}
\end{align}

\begin{table}[H]
\centering
\begin{tabular}{||l|l|l|l|l|l|l||}
\hline
\multicolumn{2}{||l|}{$N$} & 4 & 8 & 16 & 32 & 64  \\ \hline \hline
\multicolumn{1}{||l|}{\multirow{6}{*}{$\alpha$=0.2}} 
& $\widetilde{\Delta t}$ & \multicolumn{5}{|l||} {\hspace{3.7cm}5.662e+03}
\\ \cline{2-7} 
\multicolumn{1}{||l|}{}  & $\Delta t$
& 9.270e-01 & 7.033e-01 & 4.441e-01 & 2.509e-01 & 1.335e-01\\ \cline{2-7} 
\multicolumn{1}{||l|}{}  & $E_{u,h}$ & 
6.033e-02 & 2.931e-02 & 9.711e-03 & 2.841e-03 & 8.796e-04\\ \cline{2-7} 
\multicolumn{1}{||l|}{}  & $R_{u,h}$ & - & 1.04 & 1.59 & 1.89 & 1.95\\ \cline{2-7}
\multicolumn{1}{||l|}{}  & $R_{u,\Delta t}$ & - & 2.61 & 2.40 & 2.15 & 1.86\\ \cline{2-7}
\multicolumn{1}{||l|}{}   & $E_{\boldsymbol{\sigma},h}$ & 
1.009e-02 & 4.071e-03 & 1.299e-03 & 4.155e-04 & 1.357e-04\\ \cline{2-7} 
\multicolumn{1}{||l|}{}  & $R_{\boldsymbol{\sigma},h}$ & - & 1.31 & 1.65 & 1.75 & 1.86\\ \cline{2-7}
\multicolumn{1}{||l|}{}  & $R_{\boldsymbol{\sigma}, \Delta t}$ & - & 3.29 & 2.48 & 2.00 & 1.77\\ \cline{2-7}
 \multicolumn{1}{||l|}{}  & $E^\infty_h$ & 
2.179e-01 & 1.213e-01 & 5.297e-02 & 1.813e-02 & 5.603e-03\\ \cline{2-7} 
\multicolumn{1}{||l|}{}  & $R^\infty_h$ & - & 0.845 & 1.2 & 1.65 & 1.95\\ \cline{2-7}
\multicolumn{1}{||l|}{}  & $R^\infty_{\Delta t}$ & - & 2.12 & 1.80 & 1.88 & 1.86
\\ \hline \hline
\multicolumn{1}{||l|}{\multirow{6}{*}{$\alpha$=0.5}} 
& $\widetilde{\Delta t}$ & \multicolumn{5}{|l||} {\hspace{3.7cm}3.502e+01}
\\ \cline{2-7} 
\multicolumn{1}{||l|}{}  & $\Delta t$
& 5.901e-01 & 3.389e-01 & 1.813e-01 & 9.373e-02 & 4.764e-02\\ \cline{2-7} 
\multicolumn{1}{||l|}{}  & $E_{u,h}$ & 
6.189e-02 & 3.623e-02 & 1.76e-02 & 7.401e-03 & 2.474e-03\\ \cline{2-7} 
\multicolumn{1}{||l|}{}  & $R_{u,h}$ & - & 1.05 & 1.54 & 1.79 & 1.92\\ \cline{2-7}
\multicolumn{1}{||l|}{}  & $R_{u,\Delta t}$ & - & 0.97 & 1.15 & 1.31 & 1.62\\ \cline{2-7}
\multicolumn{1}{||l|}{}   & $E_{\boldsymbol{\sigma},h}$ & 
9.213e-03 & 4.444e-03 & 2.001e-03 & 8.427e-04 & 3.061e-04\\ \cline{2-7} 
\multicolumn{1}{||l|}{}  & $R_{\boldsymbol{\sigma},h}$ & - & 1.43 & 1.70 & 1.78 & 1.77\\ \cline{2-7}
\multicolumn{1}{||l|}{}  & $R_{\boldsymbol{\sigma}, \Delta t}$ & - & 1.32 & 1.27 & 1.31 & 1.50 \\ \cline{2-7}
 \multicolumn{1}{||l|}{}  & $E^\infty_h$ & 
2.177e-01 & 1.437e-01 & 7.917e-02 & 3.778e-02 & 1.445e-02\\ \cline{2-7} 
\multicolumn{1}{||l|}{}  & $R^\infty_h$ & - & 0.812 & 1.27 & 1.52 & 1.68\\ \cline{2-7}
\multicolumn{1}{||l|}{}  & $R^\infty_{\Delta t}$ & - & 0.75 & 0.95 & 1.12 & 1.42
\\ \hline \hline
\multicolumn{1}{||l|}{\multirow{6}{*}{$\alpha$=0.8}} 
& $\widetilde{\Delta t}$&  \multicolumn{5}{|l||} {\hspace{3.7cm}8.831e+0}
\\ \cline{2-7} 
\multicolumn{1}{||l|}{}  & $\Delta t$
& 3.689e-01 & 1.923e-01 & 9.810e-02 & 4.952e-02 & 2.488e-02\\ \cline{2-7} 
\multicolumn{1}{||l|}{}  & $E_{u,h}$ & 
9.326e-02 & 6.598e-02 & 3.549e-02 & 1.571e-02 & 7.235e-03\\ \cline{2-7} 
\multicolumn{1}{||l|}{}  & $R_{u,h}$ & - & 0.85 & 1.21 & 1.73 & 1.91\\ \cline{2-7}
\multicolumn{1}{||l|}{}  & $R_{u,\Delta t}$ & - & 0.53 & 0.92 & 1.19 & 1.13\\ \cline{2-7}
\multicolumn{1}{||l|}{}   & $E_{\boldsymbol{\sigma},h}$ & 
1.241e-02 & 7.997e-03 & 3.259e-03 & 1.309e-03 & 6.307e-04\\ \cline{2-7} 
\multicolumn{1}{||l|}{}  & $R_{\boldsymbol{\sigma},h}$ & - & 1.09 & 1.76 & 1.94 & 1.80 \\ \cline{2-7}
\multicolumn{1}{||l|}{}  & $R_{\boldsymbol{\sigma}, \Delta t}$ & - & 0.68 & 1.33 & 1.33 & 1.06\\ \cline{2-7}
 \multicolumn{1}{||l|}{}  & $E^\infty_h$ & 
2.917e-01 & 2.165e-01 & 1.383e-01 & 6.725e-02 & 3.474e-02\\ \cline{2-7} 
\multicolumn{1}{||l|}{}  & $R^\infty_h$ & - & 0.735 & 0.88 & 1.53 & 1.63\\ \cline{2-7}
\multicolumn{1}{||l|}{}  & $R^\infty_{\Delta t}$ & - & 0.46 & 0.67 & 1.06 & 0.96
\\  \hline \hline
\multicolumn{1}{||l|}{\multirow{6}{*}{$\alpha$=0.99}} 
& $\widetilde{\Delta t}$& \multicolumn{5}{|l||} { \hspace{3.7cm}5.364e+0}
\\ \cline{2-7} 
\multicolumn{1}{||l|}{}  & $\Delta t$
& 2.754e-01 & 1.389e-01 & 6.974e-02 & 3.493e-02 & 1.748e-02\\ \cline{2-7} 
\multicolumn{1}{||l|}{}  & $E_{u,h}$ & 
9.721e-02 & 7.034e-02 & 5.047e-02 & 2.514e-02 & 1.429e-02\\ \cline{2-7} 
\multicolumn{1}{||l|}{}  & $R_{u,h}$ & - & 0.80 & 1.15 & 1.72 & 1.96\\ \cline{2-7}
\multicolumn{1}{||l|}{}  & $R_{u,\Delta t}$ & - & 0.47 & 0.48 & 1.01 & 0.82 \\ \cline{2-7}
\multicolumn{1}{||l|}{}   & $E_{\boldsymbol{\sigma},h}$ & 
1.205e-02 & 7.532e-03 & 4.218e-03 & 1.643e-03 & 1.001e-03\\ \cline{2-7} 
\multicolumn{1}{||l|}{}  & $R_{\boldsymbol{\sigma},h}$ & - & 1.16 & 2.02 & 2.32 & 1.72\\ \cline{2-7}
\multicolumn{1}{||l|}{}  & $R_{\boldsymbol{\sigma}, \Delta t}$ & - & 0.69 & 0.84 & 1.36 & 0.72 \\ \cline{2-7}
 \multicolumn{1}{||l|}{}  & $E^\infty_h$ & 
3.014e-01 & 2.27e-01 & 1.749e-01 & 9.668e-02 & 5.647e-02\\ \cline{2-7} 
\multicolumn{1}{||l|}{}  & $R^\infty_h$ & - & 0.70 & 0.91 & 1.46 & 1.87\\ \cline{2-7}
\multicolumn{1}{||l|}{}  & $R^\infty_{\Delta t}$ & - & 0.41 & 0.38 & 0.86 & 0.78
\\  \hline 
\end{tabular}
\caption{Error $E_{u,h}$, $E_{\boldsymbol{\sigma},h}$, $E^\infty_h$ and rate of convergence $R_{u,h}$, $R_{u,\Delta t}$, $R_{\boldsymbol{\sigma},h}$, $R_{\boldsymbol{\sigma},\Delta t}$, $R^\infty_h$ and $R^\infty_{\Delta t}$ of the proposed method for Example~\ref{GTFBS}.}\label{2dBS22_f22}
\end{table}

\end{example}

\begin{example}\label{GTFM}{(The generalized (time-fractional) Merton's model.\cite{mohapatra2024analytical})}\label{gmm} Finally, we consider the pricing of European style basket put option $P$ under the following generalized (time-fractional) Merton's model
\begin{align}
   &\nonumber D^\alpha_\tau P + \frac{1}{2}\sigma_1^2 S_1^2 \frac{\partial^2 P}{\partial S_1^2} + \rho\sigma_1 \sigma_2 S_1 S_2 \frac{\partial^2 P}{\partial S_1\partial S_2} + \frac{1}{2}\sigma_2^2 S_2^2 \frac{\partial^2 P}{\partial S_2^2}  +rS_1\frac{\partial P}{\partial S_1}+rS_2\frac{\partial P}{\partial S_2} - rP\\
   & \hspace{5cm}+ \lambda JP= 0,\; (S_1,S_2,\tau)\in\mathbb{R}_{+}^2\times[0,T),\label{gmmeq1}\\
   & P(S_1, S_2,T)=\max{\left( K-\frac{S_1 + S_2}{2},0\right)}, \; (S_1,S_2)\in\mathbb{R}_{+}^2,\label{gmmeq2}
\end{align}
where $\rho$, $\sigma_i,\;i=1,2,$ $r$, and $D^\alpha_\tau P,~0 <\alpha\leq1$, the modified right Riemann-Liouville derivative, are defined in the previous example \ref{GTFBS}. The integral operator $J$ is defined as
\begin{align*}
& J(P(\tau,S)) = \int_{\mathbb{R}^2} (P(Se^z, \tau) -P(S,\tau) - (e^z -1)\cdot S \nabla P)g(z)dz\;\text{ with}\\
&g(z)=\displaystyle\frac{1}{2 \pi \sqrt{\det \Sigma}}e^{-\frac{1}{2}(z-\mu_M)\Sigma^{-1}(z-\mu_M)^T},\;z=(z_1,z_2),\;\mu_M= (\mu_{1M}, \mu_{2M}),\;\text{ and}\\
&\Sigma = \begin{bmatrix}
    \sigma_{1M}^2 & \rho_M \sigma_{1M}\sigma_{2M}\\
    \rho_M \sigma_{1M}\sigma_{2M} & \sigma_{2M}^2
\end{bmatrix}.
\end{align*}
 Note that the above generalized Merton's model reduces to the standard Merton's model when $\alpha\to 1$.
By using the transformations $t:=T-\tau$, $S_i:= Ke^{x_i},\;i=1,2,$ and $V(x_1,x_2,t):= \frac{1}{K} P(S_1, S_2, \tau)$, the above degenerate final value problem (\ref{gmmeq1}--\ref{gmmeq2}) can be transformed to the following non-degenerate initial value problem:

\begin{align}\label{transformedgmm1}
   \begin{cases}      
   & \frac{\partial^{\alpha} V}{\partial t^{\alpha}} - \mathcal{L}V - \lambda \mathcal{I}V =0 \quad \text{in}\quad \mathbb{R}^2\times J,\\
   &V(x_1,x_2,0) =\max{\left( 1-\frac{1}{2}(e^{x_1} + e^{x_2}),0\right)},\; (x_1,x_2)\in\mathbb{R}^2,
   \end{cases}
\end{align}
where $\frac{\partial^{\alpha} V}{\partial t^{\alpha}}$ is the Caputo fractional derivative of order $\alpha$, and 
and 
\begin{align}
   &\mathcal{L}V :=~ \frac{1}{2}\sigma_1^2  \frac{\partial^2 V}{\partial x_1^2} + \rho\sigma_1 \sigma_2  \frac{\partial^2 V}{\partial x_1\partial x_2} + \frac{1}{2}\sigma_2^2 \frac{\partial^2 V}{\partial x_2^2}  +\left(r-\frac{1}{2}\sigma_1^2-\lambda\xi_1\right)\frac{\partial V}{\partial x_1}+\left(r-\frac{1}{2}\sigma_2^2-\lambda\xi_2\right)\frac{\partial V}{\partial x_2} - (r+\lambda)V,\nonumber\\
   &\xi_i = \int_{\mathbb{R}^2}(e^{z_i}-1)g(z)dz = e^{\mu_{iM} + \frac{\sigma_{iM}^2}{2}} -1,\;i=1,2,\;\text{ and}\nonumber\\
   &\mathcal{I}V =\int_{\mathbb{R}^2} V(z,t)g(z-x)dz.\nonumber
\end{align}
For computational purposes, the problem (\ref{transformedgmm1}) is localized on $\Omega:=(-L,L)^2,\;L>0,$ as follows
\begin{align*}
    & \frac{\partial^{\alpha} u}{\partial t^{\alpha}} - \mathcal{L}u - \lambda \mathcal{I}u = f  \; \text{in }\Omega\times(0,T],\\
    &u(x_1,x_2,0) = \max{(1-0.5(e^{x_1}+e^{x_2}),0)}\quad\forall (x_1,x_2)\in\Omega,\\
    &u(x_1,x_2,t) = \frac{1}{2}\Big(1-0.5(e^{x_1} + e^{x_2}) + \Big((\varepsilon t)^2+ \big(1-0.5(e^{x_1} + e^{x_2})\big)^2 \Big)^{1/2}\Big),\;\forall (x_1,x_2,t)\in\partial\Omega\times[0,T],
\end{align*}
where $u(x_1, x_2,t)\approx V(x_2,x_2,t)$,  $\mathcal{I}u =\int_{\Omega} u(z,t)g(z-x)dz$ and $f =\int_{\mathbb{R}\backslash\Omega} u(z,t)g(z-x)dz$. The model parameters are as follows:
\begin{align*}
 & T = 1, \;\sigma_1 = 0.2 = \sigma_2,\; \rho= 0.5,\; r=0.06,\;L=1,\\
 & \sigma_{1M} = 0.15,\; \sigma_{2M} = 0.2,\; \mu_{1M} = -0.10,\; \mu_{2M} = 0.10,\; \varepsilon = 1.
\end{align*}
Numerical results are shown in Table~\ref{2dM222_f2}.

\begin{table}[H]
\centering
\begin{tabular}{||l|l|l|l|l|l|l||}
\hline
\multicolumn{2}{||l|}{$N$} & 4 & 8 & 16 & 32 & 64  \\ \hline \hline
\multicolumn{1}{||l|}{\multirow{6}{*}{$\alpha$=0.2}} 
& $\widetilde{\Delta t}$ & \multicolumn{5}{|l||} {\hspace{3.7cm}1.595e+03}
\\ \cline{2-7} 
\multicolumn{1}{||l|}{}  & $\Delta t$
& 9.270e-01 & 7.033e-01 & 4.441e-01 & 2.509e-01 & 1.335e-01\\ \cline{2-7} 
\multicolumn{1}{||l|}{} & 
 $E_{u,h}$ & 1.422e-01 & 7.516e-02 & 2.371e-02 & 6.931e-03 & 2.159e-03 \\ \cline{2-7}
\multicolumn{1}{||l|}{}  & $R_{u,h}$ & - & 0.92 & 1.66 & 1.89 & 1.94\\ \cline{2-7}
\multicolumn{1}{||l|}{}  & $R_{u,\Delta t}$ & - & 2.31 & 2.50 & 2.15 & 1.85\\ \cline{2-7}
\multicolumn{1}{||l|}{}  & $E_{\boldsymbol{\sigma},h}$ & 2.192e-02 & 1.021e-02 & 3.832e-03 & 1.437e-03 & 5.431e-04 \\ \cline{2-7}
\multicolumn{1}{||l|}{}  & $R_{\boldsymbol{\sigma},h}$ & - & 1.10 & 1.41 & 1.51 & 1.62\\ \cline{2-7}
\multicolumn{1}{||l|}{}  & $R_{\boldsymbol{\sigma}, \Delta t}$ & - & 2.76 & 2.13 & 1.72 & 1.54 \\ \cline{2-7}
\multicolumn{1}{||l|}{}  & $E^\infty_h$ & 4.111e-01 & 2.614e-01 & 1.238e-01 & 4.343e-02 & 1.572e-02 \\ \cline{2-7}
\multicolumn{1}{||l|}{}  & $R^\infty_h$ & - & 0.65 & 1.08 & 1.61 & 1.69\\ \cline{2-7}
\multicolumn{1}{||l|}{}  & $R^\infty_{\Delta t}$ & - & 1.64 & 1.63 & 1.83 & 1.61
\\ \hline \hline
\multicolumn{1}{||l|}{\multirow{6}{*}{$\alpha$=0.5}} 
& $\widetilde{\Delta t}$ & \multicolumn{5}{|l||} {\hspace{3.7cm}2.110e+01}
\\ \cline{2-7} 
\multicolumn{1}{||l|}{}  & $\Delta t$
& 5.9e-01 & 3.389e-01 & 1.813e-01 & 9.373e-02 & 4.764e-02\\ \cline{2-7} 
\multicolumn{1}{||l|}{}  & $E_{u,h}$ & 
1.428e-01 & 8.026e-02 & 3.963e-02 & 1.694e-02 & 5.746e-03\\ \cline{2-7} 
\multicolumn{1}{||l|}{}  & $R_{u,h}$ & - & 1.13 & 1.50 & 1.75 & 1.90\\ \cline{2-7}
\multicolumn{1}{||l|}{}  & $R_{u,\Delta t}$ & - & 1.04 & 1.13 & 1.29 & 1.60\\ \cline{2-7}
\multicolumn{1}{||l|}{}   & $E_{\boldsymbol{\sigma},h}$ & 
1.894e-02 & 1.051e-02 & 5.356e-03 & 2.572e-03 & 1.079e-03\\ \cline{2-7} 
\multicolumn{1}{||l|}{}  & $R_{\boldsymbol{\sigma},h}$ & - & 1.15 & 1.43 & 1.51 & 1.52\\ \cline{2-7}
\multicolumn{1}{||l|}{}  & $R_{\boldsymbol{\sigma}, \Delta t}$ & - & 1.06 & 1.08 & 1.11 & 1.28\\ \cline{2-7}
 \multicolumn{1}{||l|}{}  & $E^\infty_h$ & 
3.88e-01 & 2.97e-01 & 1.685e-01 & 8.503e-02 & 3.094e-02\\ \cline{2-7} 
\multicolumn{1}{||l|}{}  & $R^\infty_h$ & - & 0.52 & 1.21 & 1.41 & 1.77\\ \cline{2-7}
\multicolumn{1}{||l|}{}  & $R^\infty_{\Delta t}$ & - & 0.48 & 0.91 & 1.04 & 1.49
\\ \hline \hline
\multicolumn{1}{||l|}{\multirow{6}{*}{$\alpha$=0.8}} 
& $\widetilde{\Delta t}$& \multicolumn{5}{|l||} { \hspace{3.7cm}6.433e+0}
\\ \cline{2-7} 
\multicolumn{1}{||l|}{}  & $\Delta t$
& 3.689e-01 & 1.923e-01 & 9.810e-02 & 4.952e-02 & 2.488e-02\\ \cline{2-7} 
\multicolumn{1}{||l|}{}  & $E_{u,h}$ & 
1.934e-01 & 1.517e-01 & 7.427e-02 & 3.3e-02 & 1.547e-02\\ \cline{2-7} 
\multicolumn{1}{||l|}{}  & $R_{u,h}$ & - & 0.60 & 1.40 & 1.73 & 1.87\\ \cline{2-7}
\multicolumn{1}{||l|}{}  & $R_{u,\Delta t}$ & - & 0.37 & 1.06 & 1.19 & 1.10 \\ \cline{2-7}
\multicolumn{1}{||l|}{}   & $E_{\boldsymbol{\sigma},h}$ & 
2.197e-02 & 1.552e-02 & 7.66e-03 & 3.562e-03 & 1.873e-03\\ \cline{2-7} 
\multicolumn{1}{||l|}{}  & $R_{\boldsymbol{\sigma},h}$ & - & 0.86 & 1.38 & 1.63 & 1.58\\ \cline{2-7}
\multicolumn{1}{||l|}{}  & $R_{\boldsymbol{\sigma}, \Delta t}$ & - & 0.53 & 1.05 & 1.12 & 0.93 \\ \cline{2-7}
 \multicolumn{1}{||l|}{}  & $E^\infty_h$ & 
5.573e-01 & 4.041e-01 & 2.55e-01 & 1.268e-01 & 7.167e-02\\ \cline{2-7} 
\multicolumn{1}{||l|}{}  & $R^\infty_h$ & - & 0.79 & 0.90 & 1.49 & 1.41\\ \cline{2-7}
\multicolumn{1}{||l|}{}  & $R^\infty_{\Delta t}$ & - & 0.49 & 0.68 & 1.02 & 0.83
\\  \hline \hline
\multicolumn{1}{||l|}{\multirow{6}{*}{$\alpha$=0.99}} 
& $\widetilde{\Delta t}$&  \multicolumn{5}{|l||} {\hspace{3.7cm}4.152e+0}
\\ \cline{2-7} 
\multicolumn{1}{||l|}{}  & $\Delta t$
& 2.754e-01 & 1.389e-01 & 6.974e-02 & 3.493e-02 & 1.748e-02\\ \cline{2-7} 
\multicolumn{1}{||l|}{}  & $E_{u,h}$ & 
2.142e-01 & 1.657e-01 & 1.083e-01 & 4.989e-02 & 2.909e-02\\ \cline{2-7} 
\multicolumn{1}{||l|}{}  & $R_{u,h}$ & - & 0.63 & 1.48 & 1.91 & 1.87\\ \cline{2-7}
\multicolumn{1}{||l|}{}  & $R_{u,\Delta t}$ & - & 0.38 & 0.62 & 1.12 & 0.78 \\ \cline{2-7}
\multicolumn{1}{||l|}{}   & $E_{\boldsymbol{\sigma},h}$ & 
2.133e-02 & 1.447e-02 & 9.386e-03 & 4.672e-03 & 2.877e-03\\ \cline{2-7} 
\multicolumn{1}{||l|}{}  & $R_{\boldsymbol{\sigma},h}$ & - & 0.96 & 1.50 & 1.72 & 1.69\\ \cline{2-7}
\multicolumn{1}{||l|}{}  & $R_{\boldsymbol{\sigma}, \Delta t}$ & - & 0.57 & 0.63 & 1.01 & 0.70 \\ \cline{2-7}
 \multicolumn{1}{||l|}{}  & $E^\infty_h$ & 
6.051e-01 & 4.544e-01 & 3.389e-01 & 1.731e-01 & 1.128e-01\\ \cline{2-7} 
\multicolumn{1}{||l|}{}  & $R^\infty_h$ & - & 0.71 & 1.02 & 1.66 & 1.49\\ \cline{2-7}
\multicolumn{1}{||l|}{}  & $R^\infty_{\Delta t}$ & - & 0.42 & 0.43 & 0.97 & 0.62
\\  \hline 
\end{tabular}
\caption{Error $E_{u,h}$, $E_{\boldsymbol{\sigma},h}$, $E^\infty_h$ and rate of convergence $R_{u,h}$, $R_{u,\Delta t}$, $R_{\boldsymbol{\sigma},h}$, $R_{\boldsymbol{\sigma},\Delta t}$, $R^\infty_h$ and $R^\infty_{\Delta t}$ of the proposed method for Example~\ref{GTFM}.}\label{2dM222_f2}
\end{table}

\end{example}

It is observed from examples \ref{GTFBS} and \ref{GTFM} (see tables \ref{2dBS22_f22}, and \ref{2dM222_f2}) that the computational rates of convergence with respect to the above-weighted norms (\ref{modcomputerr}) are almost optimal,  although the initial data $u_0$ is  only in $H^1_0(\Omega)$ with slightly deterioration in the spatial order of convergence in Table 
 \ref{2dM222_f2})  of \ref{GTFM}. Due to nonsmooth initial data, we need more refined spatial mesh to derive  computationally optimal rate in space directions.

The results corresponding to $\alpha = 0.99$ in each table demonstrate the $\alpha$-robustness of the proposed method when $\alpha\to 1^{-}$ and is aligned consistently with established theoretical findings.

Finally, we discuss the time-step constraints highlighted in the main result, Theorem~\ref{L2H1errortheorem}. In the first example, there is no time-step restriction. However, in examples \ref{VariablePIDEMFEM2_f8}, \ref{VariablePIDEMFEM2_f10}, \ref{GTFBS}, and \ref{GTFM}, tables \ref{2dM222_f1}, \ref{2dpide_f10}, \ref{2dBS22_f22}, and \ref{2dM222_f2} reveal only a mild time-step restriction. In contrast, examples \ref{VariableceffPDE_newmesh}, \ref{VariableceffPDEMFEM2}, and \ref{VariablePIDEMFEM2_newmesh} (see tables \ref{2DPDE1_22}, \ref{2dpdenew22_f21}, and \ref{2dpide_f222}, respectively) indicate a stricter time-step constraint. Notably, Table~\ref{2DPDE1_22} demonstrates that, in example \ref{VariableceffPDE_newmesh}, this restriction can become impractical for small values of $\alpha$. Despite this, computational results align with theoretical estimates even when the theoretical time-step condition is not strictly observed in some of the examples. This suggests the possibility of relaxing the time-step restriction in Theorem~\ref{L2H1errortheorem}, potentially by refining the discrete fractional Gr\"{o}nwall inequality or adopting an alternative stability and error analysis approach. Exploring these options could be a direction for future research.
%
\section{Conclusion}\label{section8}
%

This study introduces and analyzes an efficient non-uniform implicit-explicit L1 mixed finite element method (IMEX-L1-MFEM) for a class of time-fractional partial differential and integro-differential equations. Stability and convergence estimates are provided along with optimal error estimates for both the solution and the flux in the $L^{2}$ -norm. All the estimates derived in this article remain valid when $\alpha \rightarrow 1^-$. Numerical experiments support our theoretical findings and demonstrate that the proposed method performs effectively for the class of problems under consideration. The present analysis can be easily generalized to semilinear problems.

\section*{Statements and Declarations}
\subsection*{Ethical Approval} Not Applicable
\subsection*{Data Availability}
The codes during the current study are available from the corresponding author on reasonable request.
\subsection*{Conflict of Interest}
The authors declare that they have no conflict of interest.
\subsection*{Funding}
The work of the  first author was supported by IIT Goa under start up grant project no.  2019/SG/LT/031 and   that  of the second  author again  by  IIT Goa. 

\subsection*{Author Contributions}
All authors contributed equally to prepare this manuscript. All authors read and approved the final manuscript.

\subsection*{Acknowledgments} 
The authors would like to thank the referees for  their constructive comments and suggestions  which lead
to a better version of this paper.

\bibliographystyle{plain}
\bibliography{references.bib}

\begin{thebibliography}{10}

\bibitem{MR0349038}
J.~H. Bramble and V.~Thom\'{e}e.
\newblock Semidiscrete least-squares methods for a parabolic boundary value problem.
\newblock {\em Math. Comp.}, 26:633--648, 1972.

\bibitem{MR1115205}
F.~Brezzi and M.~Fortin.
\newblock {\em Mixed and hybrid finite element methods}, volume~15 of {\em Springer Series in Computational Mathematics}.
\newblock Springer-Verlag, New York, 1991.

\bibitem{MR4246866}
H.~Chen and M.~Stynes.
\newblock Blow-up of error estimates in time-fractional initial-boundary value problems.
\newblock {\em IMA J. Numer. Anal.}, 41(2):974--997, 2021.

\bibitem{Ciarlet1978}
P.~G. Ciarlet.
\newblock {\em The Finite Element Method for Elliptic Problems}.
\newblock North Holland, Amsterdam, 1978.

\bibitem{MR2823474}
D.~Goswami and A.~K. Pani.
\newblock An alternate approach to optimal {$L^2$}-error analysis of semidiscrete {G}alerkin methods for linear parabolic problems with nonsmooth initial data.
\newblock {\em Numer. Funct. Anal. Optim.}, 32(9):946--982, 2011.

\bibitem{MR3914223}
Ahmed~S. H. and J.~E. Mac\'ias-D\'iaz.
\newblock A novel discrete {G}ronwall inequality in the analysis of difference schemes for time-fractional multi-delayed diffusion equations.
\newblock {\em Commun. Nonlinear Sci. Numer. Simul.}, 73:110--119, 2019.

\bibitem{MR4287911}
C.~Huang and M.~Stynes.
\newblock {$\alpha$}-robust error analysis of a mixed finite element method for a time-fractional biharmonic equation.
\newblock {\em Numer. Algorithms}, 87(4):1749--1766, 2021.

\bibitem{MR4402734}
C.~Huang and M.~Stynes.
\newblock A sharp {$\alpha$}-robust {$L^\infty(H^1)$} error bound for a time-fractional {A}llen-{C}ahn problem discretised by the {A}likhanov {$L2-1_\sigma$} scheme and a standard {FEM}.
\newblock {\em J. Sci. Comput.}, 91(2):Paper No. 43, 19, 2022.

\bibitem{MR4290515}
B.~Jin.
\newblock {\em Fractional Differential Equations: An Approach via Fractional Derivatives}, volume 206 of {\em Applied Mathematical Sciences}.
\newblock Springer, Cham, 2021.

\bibitem{MR3335216}
B.~Jin, R.~Lazarov, J.~Pasciak, and Z.~Zhou.
\newblock Error analysis of semidiscrete finite element methods for inhomogeneous time-fractional diffusion.
\newblock {\em IMA J. Numer. Anal.}, 35(2):561--582, 2015.

\bibitem{MR3033018}
B.~Jin, R.~Lazarov, and Z.~Zhou.
\newblock Error estimates for a semidiscrete finite element method for fractional order parabolic equations.
\newblock {\em SIAM J. Numer. Anal.}, 51(1):445--466, 2013.

\bibitem{MR3894161}
B.~Jin, R.~Lazarov, and Z.~Zhou.
\newblock Numerical methods for time-fractional evolution equations with nonsmooth data: a concise overview.
\newblock {\em Comput. Methods Appl. Mech. Engrg.}, 346:332--358, 2019.

\bibitem{MR3957890}
B.~Jin, B.~Li, and Z.~Zhou.
\newblock Subdiffusion with a time-dependent coefficient: analysis and numerical solution.
\newblock {\em Math. Comp.}, 88(319):2157--2186, 2019.

\bibitem{MR4125980}
B.~Jin, B.~Li, and Z.~Zhou.
\newblock Subdiffusion with time-dependent coefficients: improved regularity and second-order time stepping.
\newblock {\em Numer. Math.}, 145(4):883--913, 2020.

\bibitem{MR4572861}
B.~Jin and Z.~Zhou.
\newblock {\em Numerical treatment and analysis of time-fractional evolution equations}, volume 214 of {\em Applied Mathematical Sciences}.
\newblock Springer, Cham, 2023.

\bibitem{MR0610597}
C.~Johnson and V.~Thom\'{e}e.
\newblock Error estimates for some mixed finite element methods for parabolic type problems.
\newblock {\em RAIRO Anal. Num\'{e}r.}, 15(1):41--78, 1981.

\bibitem{MR3633783}
M.~K. Kadalbajoo, L.~P. Tripathi, and A.~Kumar.
\newblock An error analysis of a finite element method with {IMEX}-time semidiscretizations for some partial integro-differential inequalities arising in the pricing of {A}merican options.
\newblock {\em SIAM J. Numer. Anal.}, 55(2):869--891, 2017.

\bibitem{MR3816184}
S.~Karaa.
\newblock Semidiscrete finite element analysis of time fractional parabolic problems: a unified approach.
\newblock {\em SIAM J. Numer. Anal.}, 56(3):1673--1692, 2018.

\bibitem{MR4107214}
S.~Karaa.
\newblock Galerkin type methods for semilinear time-fractional diffusion problems.
\newblock {\em J. Sci. Comput.}, 83(3):Paper No. 46, 22, 2020.

\bibitem{karaa2023mixed}
S.~Karaa, K.~Mustapha, and N.~Ahmed.
\newblock A mixed {FEM} for a time-fractional {Fokker-Planck} model.
\newblock {\em arXiv preprint arXiv:2310.17350}, 2023.

\bibitem{MR3649431}
S.~Karaa, K.~Mustapha, and A.~K. Pani.
\newblock Finite volume element method for two-dimensional fractional subdiffusion problems.
\newblock {\em IMA J. Numer. Anal.}, 37(2):945--964, 2017.

\bibitem{MR3742890}
S.~Karaa, K.~Mustapha, and A.~K. Pani.
\newblock Optimal error analysis of a {FEM} for fractional diffusion problems by energy arguments.
\newblock {\em J. Sci. Comput.}, 74(1):519--535, 2018.

\bibitem{MR3834443}
S.~Karaa and A.~K. Pani.
\newblock Error analysis of a {FVEM} for fractional order evolution equations with nonsmooth initial data.
\newblock {\em ESAIM Math. Model. Numer. Anal.}, 52(2):773--801, 2018.

\bibitem{MR4108629}
S.~Karaa and A.~K. Pani.
\newblock Mixed {FEM} for time-fractional diffusion problems with time-dependent coefficients.
\newblock {\em J. Sci. Comput.}, 83(3):Paper No. 51, 22, 2020.

\bibitem{MR3957889}
N.~Kopteva.
\newblock Error analysis of the {L}1 method on graded and uniform meshes for a fractional-derivative problem in two and three dimensions.
\newblock {\em Math. Comp.}, 88(319):2135--2155, 2019.

\bibitem{kou2002jump}
S.~G. Kou.
\newblock A jump-diffusion model for option pricing.
\newblock {\em Management science}, 48(8):1086--1101, 2002.

\bibitem{MR3814402}
A.~Kubica and M.~Yamamoto.
\newblock Initial-boundary value problems for fractional diffusion equations with time-dependent coefficients.
\newblock {\em Fract. Calc. Appl. Anal.}, 21(2):276--311, 2018.

\bibitem{MR2873249}
Y.~Kwon and Y.~Lee.
\newblock A second-order finite difference method for option pricing under jump-diffusion models.
\newblock {\em SIAM J. Numer. Anal.}, 49(6):2598--2617, 2011.

\bibitem{MR3827604}
Lili L., Boya Z., Xiaoli C., and Zhiyong W.
\newblock Convergence and stability of compact finite difference method for nonlinear time fractional reaction-diffusion equations with delay.
\newblock {\em Appl. Math. Comput.}, 337:144--152, 2018.

\bibitem{MR3870961}
D.~Li, H.~Liao, W.~Sun, J.~Wang, and J.~Zhang.
\newblock Analysis of {$L1$}-{G}alerkin {FEM}s for time-fractional nonlinear parabolic problems.
\newblock {\em Commun. Comput. Phys.}, 24(1):86--103, 2018.

\bibitem{MR3790081}
H.~Liao, D.~Li, and J.~Zhang.
\newblock Sharp error estimate of the nonuniform {L}1 formula for linear reaction-subdiffusion equations.
\newblock {\em SIAM J. Numer. Anal.}, 56(2):1112--1133, 2018.

\bibitem{MR3904430}
H.~Liao, W.~McLean, and J.~Zhang.
\newblock A discrete {G}r\"{o}nwall inequality with applications to numerical schemes for subdiffusion problems.
\newblock {\em SIAM J. Numer. Anal.}, 57(1):218--237, 2019.

\bibitem{MR2349193}
Y.~Lin and C.~Xu.
\newblock Finite difference/spectral approximations for the time-fractional diffusion equation.
\newblock {\em J. Comput. Phys.}, 225(2):1533--1552, 2007.

\bibitem{MR2832607}
W.~McLean.
\newblock Regularity of solutions to a time-fractional diffusion equation.
\newblock {\em ANZIAM J.}, 52(2):123--138, 2010.

\bibitem{MR4023101}
W.~McLean, K.~Mustapha, R.~Ali, and O.~Knio.
\newblock Well-posedness of time-fractional advection-diffusion-reaction equations.
\newblock {\em Fract. Calc. Appl. Anal.}, 22(4):918--944, 2019.

\bibitem{MR4054212}
W.~McLean, K.~Mustapha, R.~Ali, and O.~Knio.
\newblock Regularity theory for time-fractional advection-diffusion-reaction equations.
\newblock {\em Comput. Math. Appl.}, 79(4):947--961, 2020.

\bibitem{Merton1976}
R.~C. Merton.
\newblock {Option pricing when underlying stock returns are discontinuous}.
\newblock {\em Journal of Financial Economics}, 3(1-2):125--144, 1976.

\bibitem{mohapatra2024analytical}
J.~Mohapatra, S.~Santra, and H.~Ramos.
\newblock Analytical and numerical solution for the time fractional black-scholes model under jump-diffusion.
\newblock {\em Computational Economics}, 63(5):1853--1878, 2024.

\bibitem{MR3802434}
K.~Mustapha.
\newblock F{EM} for time-fractional diffusion equations, novel optimal error analyses.
\newblock {\em Math. Comp.}, 87(313):2259--2272, 2018.

\bibitem{MR4090355}
K.~Mustapha.
\newblock An {$L1$} approximation for a fractional reaction-diffusion equation, a second-order error analysis over time-graded meshes.
\newblock {\em SIAM J. Numer. Anal.}, 58(2):1319--1338, 2020.

\bibitem{MR4749244}
D.~Prathumwan, T.~Khonwai, N.~Phoochalong, I.~Chaiya, and K.~Trachoo.
\newblock An improved approximate method for solving two-dimensional time-fractional-order {B}lack-{S}choles model: a finite difference approach.
\newblock {\em AIMS Math.}, 9(7):17205--17233, 2024.

\bibitem{MR483555}
P.~A. Raviart and J.~M. Thomas.
\newblock A mixed finite element method for 2nd order elliptic problems.
\newblock In {\em Mathematical aspects of finite element methods ({P}roc. {C}onf., {C}onsiglio {N}az. delle {R}icerche ({C}.{N}.{R}.), {R}ome, 1975)}, Lecture Notes in Math., Vol. 606, pages 292--315. Springer, Berlin-New York, 1977.

\bibitem{MR4199354}
J.~Ren, H.-L. Liao, J.~Zhang, and Z.~Zhang.
\newblock Sharp {$H^1$}-norm error estimates of two time-stepping schemes for reaction-subdiffusion problems.
\newblock {\em J. Comput. Appl. Math.}, 389:Paper No. 113352, 17, 2021.

\bibitem{sawangtong2018analytical}
P.~Sawangtong, K.~Trachoo, W.~Sawangtong, and B.~Wiwattanapataphee.
\newblock The analytical solution for the black-scholes equation with two assets in the liouville-caputo fractional derivative sense.
\newblock {\em Mathematics}, 6(8):129, 2018.

\bibitem{MR2551194}
R.~K. Sinha, R.~E. Ewing, and R.~D. Lazarov.
\newblock Mixed finite element approximations of parabolic integro-differential equations with nonsmooth initial data.
\newblock {\em SIAM J. Numer. Anal.}, 47(5):3269--3292, 2009.

\bibitem{MR4499608}
M.~Stynes.
\newblock A survey of the {L}1 scheme in the discretisation of time-fractional problems.
\newblock {\em Numer. Math. Theory Methods Appl.}, 15(4):1173--1192, 2022.

\bibitem{MR3639581}
M.~Stynes, E.~O'Riordan, and J.~Gracia.
\newblock Error analysis of a finite difference method on graded meshes for a time-fractional diffusion equation.
\newblock {\em SIAM J. Numer. Anal.}, 55(2):1057--1079, 2017.

\bibitem{MR2249024}
V.~Thom\'{e}e.
\newblock {\em Galerkin finite element methods for parabolic problems}, volume~25 of {\em Springer Series in Computational Mathematics}.
\newblock Springer-Verlag, Berlin, second edition, 2006.

\bibitem{TOMAR2024137}
A.~Tomar, L.~P. Tripathi, and A.~K. Pani.
\newblock Optimal error estimates of a non-uniform {IMEX-L1} finite element method for time fractional {PDEs} and {PIDEs}.
\newblock {\em Applied Numerical Mathematics}, 205:137--168, 2024.

\bibitem{MR3959545}
W.~Wang, Y.~Chen, and H.~Fang.
\newblock On the variable two-step {IMEX} {BDF} method for parabolic integro-differential equations with nonsmooth initial data arising in finance.
\newblock {\em SIAM J. Numer. Anal.}, 57(3):1289--1317, 2019.

\bibitem{MR1743405}
W.~Wyss.
\newblock The fractional {B}lack-{S}choles equation.
\newblock {\em Fract. Calc. Appl. Anal.}, 3(1):51--61, 2000.

\bibitem{MR3592147}
Y.~Zhao, P.~Chen, W.~Bu, X.~Liu, and Y.~Tang.
\newblock Two mixed finite element methods for time-fractional diffusion equations.
\newblock {\em J. Sci. Comput.}, 70(1):407--428, 2017.

\end{thebibliography}

\appendix

\section{Proof of Lemma~\ref{dfdoi}}\label{proofofLemma3.1}
\begin{proof}
The definition of the discrete fractional differential operator $D_{t_n}^\alpha$ along with the property (\ref{Kproperty}) yields the first estimate. Now, we obtain the second estimate as follows-
\begin{align}\label{pollutionterm1}
\nonumber&\left(D_{t_{n}}^{\alpha}\boldsymbol{B}^n\phi^{n},\phi^{n}\right) =K_{1-\alpha}^{n,n}(\boldsymbol{B}^n\phi^n, \phi^{n} )-\sum_{j=1}^{n-1} (K_{1-\alpha}^{n,j+1} - K_{1-\alpha}^{n,j})(\boldsymbol{B}^j\phi^j , \phi^{n} ) - K_{1-\alpha}^{n,1}(\boldsymbol{B}^0\phi^0 , \phi^{n} )\\
 \nonumber& \geq K_{1-\alpha}^{n,n}\vertiii{ \phi^n}_n^2 -\sum_{j=1}^{n-1} (K_{1-\alpha}^{n,j+1} - K_{1-\alpha}^{n,j})\vertiii{ \phi^j}_j \vertiii{\phi^n}_j - K_{1-\alpha}^{n,1}\vertiii{\phi^0}_0 \vertiii{ \phi^n}_0\\
\nonumber&\geq K_{1-\alpha}^{n,n}\vertiii{\phi^n}^2_n  - \frac{1}{2}\sum_{j=1}^{n-1} (K_{1-\alpha}^{n,j+1} - K_{1-\alpha}^{n,j})\vertiii{\phi^n}^2_j   \\
\nonumber&-\frac{1}{2}\sum_{j=1}^{n-1} (K_{1-\alpha}^{n,j+1} - K_{1-\alpha}^{n,j})\vertiii{\phi^j}^2_j- \frac{1}{2}K_{1-\alpha}^{n,1}\vertiii{\phi^0}^2_0 -\frac{1}{2}K_{1-\alpha}^{n,1}\vertiii{\phi^n}^2_0\\
\nonumber&=\frac{1}{2}D_{t_{n}}^{\alpha}\vertiii{\phi^{n}}_n^{2}-\frac{1}{2} \sum_{j=1}^{n-1} (K_{1-\alpha}^{n,j+1} - K_{1-\alpha}^{n,j})\vertiii{\phi^n}^2_j- \frac{1}{2}K_{1-\alpha}^{n,1}\vertiii{\phi^n}^2_0+ \frac{1}{2}K_{1-\alpha}^{n,n}\vertiii{\phi^n}^2_{n}\\
\nonumber&=\frac{1}{2}D_{t_{n}}^{\alpha}\vertiii{\phi^{n}}_n^{2}-\frac{1}{2} \sum_{j=0}^{n-1} K_{1-\alpha}^{n,j+1}\vertiii{\phi^n}^2_j + \frac{1}{2} \sum_{j=1}^{n}K_{1-\alpha}^{n,j}\vertiii{\phi^n}^2_j\\
&=\frac{1}{2}D_{t_{n}}^{\alpha}\vertiii{\phi^{n}}_n^{2} + \frac{1}{2} \sum_{j=0}^{n-1} K_{1-\alpha}^{n,j+1}(\vertiii{\phi^n}^2_{j+1}-\vertiii{\phi^n}^2_{j}). 
\end{align}
Further, an application of (\ref{Blips}) implies that
\begin{align}\label{lipschitzineq}
\Big|\vertiii{\phi^n}^2_{j+1}-\vertiii{\phi^n}^2_{j}\Big| \leq \|\boldsymbol{B}^{j+1}-\boldsymbol{B}^j\|_{L^{\infty}(\Omega;\mathbb{R}^{d\times d})}\|\phi^n\|^2 \leq L_B \Delta t_{j+1} \|\phi^n\|^2 \leq \frac{L_B}{\beta_0} \Delta t_{j+1} \vertiii{\phi^n}_n^2   
\end{align}
and using the definition (\ref{discretekernel}), we obtain
\begin{align}\label{coefficientbound}
  \sum_{j=0}^{n-1}K_{1-\alpha}^{n,j+1}(t_{j+1} - t_{j}) &=\sum_{j=0}^{n-1}\int_{t_{j}}^{t_{j+1}}k_{1-\alpha}(t_n-s)ds = \int_0^{t_{n}}k_{1-\alpha}(t_n-s)ds = \frac{t_{n}^{1-\alpha}}{\Gamma(2-\alpha)}.
\end{align} 
Now, a use of (\ref{lipschitzineq}) with (\ref{pollutionterm1}) completes the rest of the proof.
\end{proof}

\end{document}